\documentclass[12pt,a4paper]{amsart}

\usepackage{fullpage}
\usepackage[pdftex]{epsfig}
\newcommand{\Gr}[2]{\psfig{file=#1.pdf,width=#2}}

\usepackage{amssymb}  
\usepackage{latexsym} 
\usepackage{comment}
\usepackage{color}
\usepackage{enumerate}
\usepackage[all]{xy}

\usepackage{charter,euler} 
\DeclareSymbolFont{bchoperators}{T1}{bch}{m}{n}
\SetSymbolFont{bchoperators}{bold}{T1}{bch}{b}{n}
\makeatletter
\renewcommand{\operator@font}{\mathgroup\symbchoperators}
\makeatother

\usepackage{titlesec} 
\titleformat{\section}{\normalfont\bfseries\filcenter}{\thesection}{1em}{}
\titleformat{\subsection}{\normalfont\bfseries}{\thesubsection}{1em}{}
\titleformat{\subsubsection}{\normalfont\bfseries}{\thesubsubsection}{1em}{}

\newcommand{\Z}{{\mathbb Z}}
\newcommand{\Q}{{\mathbb Q}}
\newcommand{\R}{{\mathbb R}}
\newcommand{\F}{{\mathbb F}}
\newcommand{\C}{{\mathbb C}}
\newcommand{\BP}{{\mathbb P}}

\newcommand{\CH}{{\mathcal H}}
\newcommand{\CC}{{\mathcal C}}
\newcommand{\CW}{{\mathcal W}}
\newcommand{\To}{\longrightarrow}

\renewcommand{\Bar}{\overline}

\newcommand{\diag}{\operatorname{diag}}

\newcommand{\GL}{\operatorname{GL}}
\newcommand{\SL}{\operatorname{SL}}

\newcommand{\Mat}{\operatorname{Mat}}
\newcommand{\Gal}{\operatorname{Gal}}

\newcommand{\eps}{\varepsilon}

\newcommand{\sing}{{\text{\rm sing}}}

\newcommand{\floor}[1]{{\left\lfloor #1 \right\rfloor}}

\newcommand{\one}{{\mathbf 1}}
\newcommand{\zero}{{\mathbf 0}}

\newcommand{\kk}{{\mathbf k}}
\newcommand{\sm}{{\text{small}}}
\newcommand{\lr}{{\text{large}}}
\newcommand{\ppx}[1]{\frac{\partial}{\partial #1}}
\newcommand{\Ue}{\operatorname{\text{\"U}}}

\newcommand{\interval}[2]{{\text{\rm\textbf{[}}#1, #2\text{\rm\textbf{]}}}}

\newcommand{\sq}{\strut\quad}
\newcommand{\sqq}{\strut\qquad}
\newcommand{\sqqq}{\strut\qquad\quad}
\newcommand{\sqqqq}{\strut\qquad\qquad}

   {\begin{list}{}{\settowidth{\labelwidth}{#1}%
                   \setlength{\leftmargin}{\labelwidth}%
                   \addtolength{\leftmargin}{\labelsep}}}%
   {\end{list}}

\newtheorem{Theorem}{Theorem}[section]
\newtheorem{Lemma}[Theorem]{Lemma}
\newtheorem{Proposition}[Theorem]{Proposition}
\newtheorem{Corollary}[Theorem]{Corollary}
\newtheorem{Conjecture}[Theorem]{Conjecture}

\theoremstyle{definition}
\newtheorem{Definition}[Theorem]{Definition}
\newtheorem{Example}[Theorem]{Example}

\newtheorem{Remark}[Theorem]{Remark}
\newtheorem{Question}[Theorem]{Question}
\newtheorem{Algorithm}[Theorem]{Algorithm}

\numberwithin{equation}{section}

\definecolor{darkgreen}{rgb}{0,0.5,0}
\definecolor{gray}{gray}{0.7}
\usepackage[
        colorlinks, citecolor=darkgreen,
        backref,
        pdfauthor={Andreas-Stephan Elsenhans, Michael Stoll},
        pdftitle={Minimization of hypersurfaces},
]{hyperref}
\usepackage[alphabetic,backrefs,lite]{amsrefs} 

\begin{document}

\title{Minimization of hypersurfaces}

\author{Andreas-Stephan Elsenhans}
\address{Institut f\"ur Mathematik,
         Universit\"at W\"urzburg,
         Emil-Fischer-Stra{\ss}e 30,
         97074 W\"urzburg, Germany.}
\email{stephan.elsenhans@mathematik.uni-wuerzburg.de}
\urladdr{\tiny\url{https://www.mathematik.uni-wuerzburg.de/computeralgebra/team/elsenhans-stephan-prof-dr/}}

\author{Michael Stoll}
\address{Mathematisches Institut,
         Universit\"at Bayreuth,
         95440 Bayreuth, Germany.}
\email{Michael.Stoll@uni-bayreuth.de}
\urladdr{\url{http://www.mathe2.uni-bayreuth.de/stoll/}}

\date{October 11, 2023}


\begin{abstract}
  Let $F \in \Z[x_0, \ldots, x_n]$ be homogeneous of degree~$d$ and assume
  that $F$ is not a `nullform', i.e., there is an invariant~$I$ of forms
  of degree~$d$ in $n+1$~variables such that $I(F) \neq 0$. Equivalently,
  $F$ is semistable in the sense of Geometric Invariant Theory.
  Minimizing~$F$ at a prime~$p$
  means to produce $T \in \Mat(n+1, \Z) \cap \GL(n+1, \Q)$ and $e \in \Z_{\ge 0}$
  such that $F_1 = p^{-e} F([x_0, \ldots, x_n] \cdot T)$ has integral
  coefficients and $v_p(I(F_1))$ is minimal among all such~$F_1$.
  Following Koll\'ar~\cite{Kollar}, the minimization process can be
  described in terms of applying weight vectors $w \in \Z_{\ge 0}^{n+1}$ to~$F$.
  We show that for any dimension~$n$ and degree~$d$, there is a complete set
  of weight vectors consisting of $[0,w_1,w_2,\dots,w_n]$ with
  $0 \le w_1 \le w_2 \le \dots \le w_n \le 2 n d^{n-1}$.
  When $n = 2$, we improve the bound to~$d$.
  This answers a question raised by Koll\'ar.
  These results are valid in a more general context, replacing
  $\Z$ and~$p$ by a PID~$R$ and a prime element of~$R$.

  Based on this result and a further study of the minimization process
  in the planar case $n = 2$, we devise an efficient minimization algorithm
  for ternary forms (equivalently, plane curves) of arbitrary degree~$d$.
  We also describe a similar algorithm that allows to minimize (and reduce)
  cubic surfaces. These algorithms are available in the computer algebra
  system Magma.
\end{abstract}

\maketitle


\section{Introduction} 

When one wants to do explicit computations with algebraic varieties
over~$\Q$ (or, more generally, over a number field), it is advantageous
to use an explicit model that is given by equations with small integral
coefficients. So it is an interesting question how one can try to
simplify or optimize a given model in this sense. This involves
two aspects. One aspect is that one strives to minimize the absolute value
(in general, the norm) of a suitable invariant, for example, the
discriminant in the common situation when the variety is smooth.
This can be seen
as optimizing the reduction properties of the model at all primes;
this is usually known as \emph{minimization} of the given model.
The other aspect concerns making the coefficients small while
staying in the same isomorphism class over~$\Z$. This has a different
flavor and is known as \emph{reduction}. Minimization and reduction
have been studied for $2$-, $3$-, $4$- and $5$-coverings of elliptic curves
in~\cite{CFS2010} and~\cite{Fisher2013}. The reduction theory of
binary forms is studied in~\cite{CremonaStoll2003} and~\cite{HutzStoll}
and that of point clusters in projective space in~\cite{Stoll2011b}.
The latter can be used to obtain a reduction method also for more
general projective varieties; for example, we can reduce equations
of plane curves by reducing their multiset of inflection points.

In this paper, we will discuss minimization in the case of hypersurfaces.
This problem has been considered by Koll\'ar in~\cite{Kollar} in some detail.
See also the recent paper~\cite{AFK}, which puts Koll\'ar's
approach in a more general context and extends it to hypersurfaces
and to intersections of two hypersurfaces in weighted projective spaces.
Koll\'ar writes (at the end of the introduction of~\cite{Kollar})
that ``so far I could not prove a bound on the weights occurring in~(4.3),
except in some special cases.'' One goal of this paper is to provide
such a bound, which is completely explicit and close to optimal in the
case of plane curves; see Theorems \ref{ThmMain} and~\ref{Thmn=2} below.
The availability of an explicit bound on the weights (see below for
definitions) leads, at least in principle, to a minimization algorithm
for hypersurfaces of given degree and dimension; see Section~\ref{S:algos}.
In the context of plane curves of degree~$d$, the case $d = 1$ is not
interesting and the case $d = 2$ is classical. (Smooth) plane cubics ($d = 3$)
are $3$-coverings of elliptic curves and are therefore considered
in~\cite{CFS2010}. (Definitions~2.3 and~3.1 in~\cite{CFS2010} lead
to a definition of minimality equivalent to what is used here.)

We work out the case of plane curves in general:
we show that minimization can be achieved by successive steps using
only the two most basic weight vectors. Combined with the bound on
the weights, this leads to a reasonably efficient algorithm that
produces a $p$-minimal (planar) model for any semistable plane curve.
See Section~\ref{S:gencurves}.

We include a short discussion on the minimization of binary forms
(Section~\ref{S:binary}), which can serve as a warm-up section
before dealing with the general theory and the case of plane curves.

When working over~$\Q$ or, more generally, over an algebraic number field
of class number~$1$, minimization can be considered for each prime~$p$
independently, in the sense that we can produce another integral
model whose discriminant (say) has minimal possible $p$-adic valuation
and unchanged valuation at all other primes. So we just have to
perform this \emph{minimization at~$p$} successively for each
potentially non-minimal prime~$p$ to arrive at a minimal model.

We describe how
one can find a small set of primes that contains the primes
at which a given plane curve is not minimal in a reasonably efficient
way and how to reduce a plane curve, i.e., to find a unimodular
transformation that makes the coefficients small;
see Section~\ref{S:gencurves-glob}.
We add some discussion of the problem of finding representatives of
all $\GL(n+1, \Z)$-equivalence classes of (globally) minimal models
in Section~\ref{S:all-minimal}.

As a further application, we give an explicit minimization
algorithm for cubic surfaces in Section~\ref{S:cubsurf};
we add a discussion of reduction for cubic surfaces in Section~\ref{S:csred}
so as to have a complete treatment of this case as well.
Unfortunately, one important ingredient that allows us
to obtain a general algorithm for plane curves whose complexity mainly
depends on the degree~$d$ and only to a small extent on~$p$ does not
carry over to the case of surfaces in~$\BP^3$. This prevents us from
generalizing the minimization algorithm for plane curves to higher
dimensions; see Section~\ref{S:higher}.

Our results and algorithms are formulated in terms of $\Z$, $\Q$
and a prime number~$p$, but we really only need the fact that
$p$ is a prime element and that $\Z$ is a principal ideal domain. In particular,
everything we do remains valid if we replace $\Z$ and~$p$ by a
PID~$R$ and a prime element~$\pi$ of~$R$.
For example, we can take $R = k[t]$, the polynomial ring over
an algebraically closed field~$k$ and $\pi = t - \alpha$ for some $\alpha \in k$;
this allows us to produce minimal models of families of plane
curves over the affine line. Another possibility is to
take $R$ to be a DVR with uniformizer~$\pi$; then we talk about
minimizing fibers of families of projective hypersurfaces over
a one-dimensional base, in an arithmetic or geometric setting.
For the algorithms, we have of course to assume that we can do
computations in~$R$ and in the residue class field $k = R/\langle \pi \rangle$.
For the general statement of Proposition~\ref{P:effectivity},
we also need to assume that $k$ is finite, but we would like to
stress that this assumption is not needed for the minimization algorithms
for plane curves or cubic surfaces.

For the following, Koll\'ar's paper~\cite{Kollar} is the main reference.
We fix $n \ge 1$ and~$d \ge 1$ and consider homogeneous polynomials~$F$
of degree~$d$ in the $n+1$ variables $x_0, \dots, x_n$, with integral
coefficients. We also fix a prime number~$p$ and write $v_p(F)$ for the
minimum of the $p$-adic valuations of the coefficients of~$F$. Vectors
will be row vectors; vectors and matrices are denoted using square brackets.
If $T \in \GL(n+1,\Q)$, then ${}^T F$ denotes
$F([x_0, \dots, x_n] \cdot T)$. If $T \in \GL(n+1, \Z)$,
then it follows that $v_p({}^T F) = v_p(F)$.
As a matter of notation, ${}^T F([x_0, \ldots, x_n] M)$ means
$({}^T F)([x_0, \ldots, x_n] \cdot M) = {}^{MT} F$ and not
${}^T (F([x_0, \ldots, x_n] \cdot M)) = {}^{TM} F$, where $M \in \Mat(n+1, \Q)$
is another matrix. This applies in particular to
\[ {}^T F(p^{w_0} x_0, \ldots, p^{w_n} x_n)
    = ({}^T F)(p^{w_0} x_0, \ldots, p^{w_n} x_n) \,.
\]
We write $E_n$ for the $n \times n$ identity matrix.

\begin{Definition}
  A \emph{weight system} is a pair $(T, w)$, where $T \in \GL(n+1,\Z)$
  and $w \in \Z_{\ge 0}^{n+1}$; $w$ is called the \emph{weight vector}
  of the weight system.
\end{Definition}

\begin{Definition} \label{D:unstable}
  A homogeneous polynomial $F \in \Z[x_0, \dots, x_n]$ is \emph{unstable}
  at~$p$ for a weight system~$(T, w)$ with $w = [w_0, \dots, w_n]$ if
  \[ v_p\bigl({}^T F(p^{w_0} x_0, p^{w_1} x_1, \dots, p^{w_n} x_n)\bigr)
       > \frac{d}{n+1} (w_0 + w_1 + \dots + w_n) \,.
  \]
\end{Definition}

Let $I$ be an invariant of forms of degree~$d$ in $n+1$~variables such that
$I(F) \ne 0$ (see Definition~\ref{D:inv}). Then the condition is equivalent to
\[ v_p\bigl(I(p^{-v_p(F_1)} F_1)\bigr) < v_p\bigl(I(F)\bigr) \,, \]
where $F_1 = {}^T F(p^{w_0} x_0, p^{w_1} x_1, \dots, p^{w_n} x_n)$.
This shows that a semistable form~$F$ is \emph{minimal} at $p$ in the sense
that $v_p(I)$ is minimal among all forms equivalent to~$F$
that have integral coefficients if and only if it is not unstable at~$p$.
Here a form~$F_1$ is \emph{equivalent} to~$F$ if $F_1 = \lambda {}^T F$
with $\lambda \in \Q^\times$ and $T \in \GL(n+1, \Q)$.

Definition~\ref{D:unstable} prompts us to introduce the following notion.

\begin{Definition} \label{D:Exp}
  Let $w$ be a weight vector. We write
  \[ \Sigma w = w_0 + w_1 + \ldots + w_n \]
  for the sum of its entries, and we call
  \[ e(w) = \Bigl\lfloor \frac{d}{n+1} \Sigma w \Bigr\rfloor + 1 \]
  the \emph{exponent} of $w$.
\end{Definition}

The condition in Definition~\ref{D:unstable} is then equivalent to
\[ v_p\bigl({}^T F(p^{w_0} x_0, p^{w_1} x_1, \dots, p^{w_n} x_n)\bigr) \ge e(w) \,. \]

For example, a polynomial~$F$ is unstable at~$p$ for~$(T, [0,\ldots,0])$
if and only if $v_p(F) \ge 1$, i.e., if $p$ divides the gcd of the coefficients
of~$F$.

\begin{Definition} \label{D:completeSet}
  Let $S \subseteq \Z_{\ge 0}^{n+1}$ be a set of weight vectors. The set
  $S$ is a \emph{complete set of weight vectors} (for dimension~$n$
  and degree~$d$) if the following holds.
  If $F \in \Z[x_0, \dots, x_n]$, homogeneous of degree~$d$, is unstable at~$p$
  for some weight system $(T, w)$, then $F$ is also unstable at~$p$ for
  a weight system~$(T', w')$ with $w' \in S$.
\end{Definition}

Koll\'ar raises the question (in~\cite{Kollar}*{1.9}) whether there is
a bound on the weights that one needs to consider, or equivalently,
whether there is always a finite complete set of weight vectors.
This question was answered positively by the first author of this note
in~\cite{Elsenhans}, but without giving explicit bounds.
If we have an explicit bound, then we have an explicit finite complete
set of weight vectors, which allows us to construct an algorithm for minimizing
a given hypersurface at a given prime, see Section~\ref{S:algos}.
Experimental evidence suggests the following.

\begin{Conjecture} \label{Conj}
  For given dimension~$n$ and degree~$d$, there is a complete set of weight
  vectors whose entries are bounded by $d^{n-1}$.
\end{Conjecture}

This is trivially true when $n = 1$; in this case, $\{[0,0], [0,1]\}$
is a complete set of weight vectors for every degree~$d$.

We can prove Conjecture~\ref{Conj} in the case of plane curves, $n = 2$.
This results in the following theorem.

\begin{Theorem} \label{Thmn=2}
  For every~$d \ge 1$, there is a complete set of weight vectors for ternary
  forms of degree~$d$ whose entries are bounded by~$d$.
\end{Theorem}

See Section~\ref{S:n=2} for the proof.

We can also prove the following general result, which is slightly weaker
(by a factor of~$2n$ at worst) than Conjecture~\ref{Conj}.

\begin{Theorem} \label{ThmMain}
  For every dimension $n \ge 2$ and degree $d \ge 1$, the subset of
  \[ W_n = \{[w_0, w_1, \dots, w_n] \in \Z^{n+1} : 0 = w_0 \le w_1 \le \dots \le w_n\} \]
  consisting of (primitive) vectors with
  \[ w_n \le 2 n \frac{d}{\gcd(d,n+1)} d^{n-2} \]
  is a complete set of
  weight vectors for homogeneous polynomials of degree~$d$ in $n+1$~variables.
\end{Theorem}

Note that it is easy to see that $w$ dominates all its positive integral multiples
in the sense of Definition~\ref{DefDom} below; therefore we can restrict to primitive
(i.e., with gcd~$1$) weight vectors.

We give the proof of Theorem~\ref{ThmMain} in Section~\ref{S:proofgeneral}.

We have formalized some of our results using the Lean Interactive Theorem
Prover and its mathematical library~\cite{Lean}. The code is available at~\cite{Github}.

\subsection*{Acknowledgments}

We would like to thank the referees for their helpful suggestions and in particular
for prompting us to revisit the uniqueness of minimal complete systems of weight
vectors. We also thank Robert Nowak for pointing out two (fortunately minor)
mistakes in an earlier version of the paper.


\section{Binary forms} \label{S:binary}

Before we begin with the general theory, we consider the case $n = 1$ of
binary forms. As mentioned in the introduction, the two weight vectors
$[0, 0]$ and~$[0, 1]$ form a (minimal if $d \ge 2$) complete set of weight vectors
in this case, regardless of the degree~$d$. (See also~\cite{Kollar}*{Prop.~6.1.1}.)
Consider a binary form
\[ F = a_0 x_1^d + a_1 x_0 x_1^{d-1} + \ldots + a_{d-1} x_0^{d-1} x_1 + a_d x_0^d \]
of degree~$d$, with coefficients in~$\Z$. This form is unstable at~$p$
for~$(T, [0, 0])$ (with any \hbox{$T \in \GL(2, \Z)$}) if and only if $p$ divides
the gcd of the coefficients.
So the first step in the minimization procedure for binary forms is to
divide $F$ by the gcd of its coefficients. Then we only need to consider
the other weight vector, $[0, 1]$. The condition that $F$ be unstable
at~$p$ for~$(E_2, [0, 1])$ is that
\[ v_p(a_j) > j - \frac{d}{2} \qquad \text{for all $j \ge \frac{d}{2}$.} \]
In particular, the reduction~$\bar{F}$ of~$F$ mod~$p$ must be divisible
by~$x_1^{\lceil (d+1)/2 \rceil}$. This implies that if $F$ is unstable
at~$p$ for~$(T, [0, 1])$ with some $T \in \GL(2, \Z)$, then $\bar{F}$
has a linear factor $L$ of multiplicity $> d/2$; such a linear factor is
then uniquely determined. Let $T \in \GL(2, \Z)$ be such that
${}^{\bar{T}}L = \lambda x_1$. One can check that whether $F$ is unstable
at~$p$ for~$(T, [0, 1])$ or not does not depend on which~$T$ with this
property is chosen. (This is a special case of Lemma~\ref{L:finite}.)

This leads to the following algorithm for minimizing a binary form at a prime~$p$.

\begin{Algorithm}
  The input of {\sf MinimizeBinaryFormOneStep} and {\sf MinimizeBinaryForm}
  consists in a semistable binary form $F \in \Z[x_0, x_1]$ of degree~$d \ge 2$
  and a prime number~$p$. The result of {\sf MinimizeBinaryFormOneStep}
  consists of a boolean flag indicating
  whether a minimization step could be performed successfully and in this case,
  a form~$G$ of degree~$d$, a matrix~$T$ and a number $e \in \Z_{\ge 0}$
  such that $G = p^{-e} \cdot {}^T F$ is the result of the minimization step; otherwise
  $F$, the identity matrix $E_2$ and~$0$ are returned as the last three values.
  The result of {\sf MinimizeBinaryForm} consists of a form~$G$ of degree~$d$
  that is a minimized representative of the orbit of~$F$, together with a
  matrix~$T$ and a number $e \in \Z_{\ge 0}$ as above.

  {\sf MinimizeBinaryFormOneStep}($F$, $p$) \\
  \sq $d := \deg(F)$; \\
  \sq $\bar{F} = F \bmod p \in \F_p[x_0, x_1]$; \\
  \sq \textbf{if} $\bar{F}$ has a factor $L^m$ with $\deg(L) = 1$ and $m > d/2$ \textbf{then} \\
  \sqq $T := $ a matrix in $\GL(2, \Z)$ such that ${}^{\bar{T}} L = \lambda x_1$; \\
  \sqq $G := {}^T F$; // \emph{now $\bar{G}$ is divisible by $x_1^{\lceil (d+1)/2 \rceil}$} \\
  \sqq $G_1 := G(x_0, p x_1)$; $e := v_p(G_1)$; // \emph{apply $w = [0,1]$} \\
  \sqq \textbf{if} $e > d/2$ \textbf{then} // \emph{unstable?} \\
  \sqqq \textbf{return} {\sf true}, $p^{-e} G_1$, $T$, $e$; \\
  \sqq \textbf{end if}; \\
  \sq \textbf{end if}; \\
  \sq \textbf{return} {\sf false}, $F$, $E_2$, 0;

  {\sf MinimizeBinaryForm}($F$, $p$) \\
  \sq $T := E_2$; $e := v_p(F)$; $G := p^{-e} F$; // \emph{initialize; do $w = [0,0]$} \\
  \sq success, $G$, $T_1$, $e_1 := $ {\sf MinimizeBinaryFormOneStep}($G$, $p$); \\
  \sq \textbf{while} success \textbf{do} \\
  \sqq $T := T_1 T$; $e := e + e_1$; // \emph{update transformation data} \\
  \sqq success, $G$, $T_1$, $e_1 := $ {\sf MinimizeBinaryFormOneStep}($G$, $p$); \\
  \sq \textbf{end while}; \\
  \sq \textbf{return} $G$, $T$, $e$;
\end{Algorithm}

This algorithm is available in Magma~\cite{Magma}
under the name~\texttt{MinimizeAtP}.

Note that we use a geometric condition on the reduction~$\bar{F}$ of~$F$
mod~$p$ (existence of a high-multiplicity factor) to restrict to essentially
just one possibility for the minimization step. We will use a similar idea
later when dealing with plane curves (the case $n = 2$).

To obtain a complete minimization procedure, we also have to determine
a finite set of primes~$p$ at which the given form~$F$ might be unstable.
We use the same geometric condition: either all of $a_0, a_1, \ldots, a_{\lfloor d/2 \rfloor}$
are divisible by~$p$ (this is the condition for $x_0^{\lceil (d+1)/2 \rceil}$
to divide~$\bar{F}$), or, setting $f(x) = F(1, x)$, the divided derivatives
\[ f, \; f', \; \frac{1}{2} f'', \; \frac{1}{3!} f''', \; \ldots, \;
   \frac{1}{\lfloor d/2 \rfloor!} f^{(\lfloor d/2 \rfloor)}
\]
have a common root~$\xi$ mod~$p$ (then $(x_1 - \xi x_0)^{\lceil (d+1)/2 \rceil}$
divides~$\bar{F}$). To find the primes satisfying the first
condition, we determine the prime factors of the gcd of the relevant
coefficients. To deal with the second condition, we use a Gr\"obner basis
computation to determine the positive generator of the intersection with~$\Z$
of the ideal generated by the divided derivatives; its prime divisors
are the relevant primes. For each of the finitely many primes~$p$ found
in this way, we then apply {\sf MinimizeBinaryForm} to~$F$ and~$p$ and
replace $F$ by the result (and keep track of the transformations made).
This results in a minimal integral representative~$F_0$ of the orbit of~$F$
(together with the transformation matrix and scaling factor used to
obtain it).

This minimal form can still have quite large coefficients. So we want to
find a matrix $T \in \GL(2, \Z)$ such that ${}^T F_0$ has
small coefficients. (Since $T$ is unimodular, acting on~$F_0$ by~$T$ does
not affect the minimality property.) This is known as \emph{reduction};
algorithms that perform it are described in~\cites{CremonaStoll2003,HutzStoll}.

A combination of minimization and reduction for binary forms with
integral coefficients is available in Magma as~\texttt{MinRedBinaryForm}.


\section{Dominating weights} \label{S:generalities}

In this section $n$ and $d$ are fixed.

The condition on~$S$ in Definition~\ref{D:completeSet} can equivalently
be stated with $T$ replaced by the identity matrix $E = E_{n+1}$, since we can replace $F$
by ${}^T F$. But it still involves an arbitrary matrix $T' \in \GL(n+1, \Z)$,
which is hard to control. We therefore consider a weaker property that eliminates
the matrix and can be reduced to a combinatorial statement. This will be the
key for the proofs of Theorems \ref{Thmn=2} and~\ref{ThmMain}.

\begin{Definition} \label{DefDom}
  Let $w$ and $w'$ be two weight vectors. We say that \emph{$w$ dominates~$w'$}
  if whenever $F \in \Z[x_0, \dots, x_n]$ is a homogeneous polynomial of
  degree~$d$ that is unstable at~$p$ for the weight system $(E, w')$, then
  $F$ is also unstable at~$p$ for the weight system~$(E, w)$.
\end{Definition}

The dominance relation is clearly transitive.

\begin{Lemma} \label{L:DomComplete}
  A set~$S$ of weight vectors with the property that some permutation of every
  weight vector is dominated by some element of~$S$ is complete.
\end{Lemma}

\begin{proof}
  Let $F$ be a form of degree~$d$ in $n+1$ variables and let $(T,w)$
  be a weight system such that $F$ is unstable at~$p$ for~$(T,w)$.
  Let $\tilde{w}$ be a permutation of~$w$ that is dominated by an
  element~$w'$ of~$S$, and let $P$ be the permutation matrix such that
  $F$ is unstable for~$(PT, \tilde{w})$. Equivalently, ${}^{PT} F$ is unstable
  for~$(E, \tilde{w})$. Then by the definition of dominance, ${}^{PT} F$ is
  also unstable for~$(E, w')$, hence $F$ is unstable for~$(PT, w')$.
  This shows that $S$ is complete.
\end{proof}

It is not true in general that the implication in Lemma~\ref{L:DomComplete}
is an equivalence, as the following example demonstrates.

\begin{Example} \label{Ex:compl}
  We consider the case of quadrics in~$\BP^3$, so $n = 3$ and $d = 2$.
  It is not hard to see that $S = \{[0,0,0,1], [0,1,1,1]\}$ is a complete
  set of weight vectors in this case. On the other hand, the weight
  vector $[0,0,1,2]$ (or any of its permutations) is not dominated by either
  of the two vectors in~$S$. A similar phenomenon occurs for all $n \ge 3$
  when $d = 2$.
\end{Example}

On the other hand, we are not aware of any similar example when $d \ge 3$.

\begin{Question}
  Assume that $n \ge 1$, $d \ge 3$ and that $S$ is a complete set of weight vectors.
  Is it necessarily true that for every weight vector~$w$, $S$ contains
  a weight vector that dominates a permutation of~$w$?
\end{Question}

In the following, we will exclusively work with the sufficient condition
for completeness given by Lemma~\ref{L:DomComplete}. We will therefore take
the liberty to use the word `complete' to indicate that $S$ satisfies this
stronger condition. We then say that a complete
(in this sense) set of weight vectors is \emph{minimal} if it is minimal
with respect to inclusion among all complete sets of weight vectors.

We will see below that there is always a finite minimal complete set
of weight vectors for our parameters $n$ and~$d$. Starting from any
finite complete set~$S$ of weight vectors (for example, as provided by
Theorems \ref{Thmn=2} or~\ref{ThmMain}), we arrive at such a minimal set
by successively selecting an element~$w$ of~$S$ and removing all elements
from~$S$ other than~$w$ that are dominated by a permutation of~$w$,
until no element of the remaining set dominates a permutation of
any other element.

We give a combinatorial description of the dominance relation. Let
\[ J = J_{n,d}
     = \{i = [i_0, \dots, i_n] \in \Z_{\ge 0}^{n+1} : i_0 + \dots + i_n = d\}
\]
be the index set for the monomials occurring in homogeneous polynomials of degree~$d$
in $n+1$~variables. We write $F = \sum_{i \in J} a_i x^i$ (with the usual
abbreviation $x^i = x_0^{i_0} \dots x_n^{i_n}$). Then $F$ is unstable at~$p$
for~$(E, w)$ if and only if
\begin{equation} \label{InEqFund}
  v_p(a_i) \ge e(w) - \langle i, w \rangle
\end{equation}
for all $i \in J$. Here $\langle {\cdot}, {\cdot} \rangle$
denotes the standard inner product. Since $v_p(a_i) \ge 0$,
such a condition is vacuous if $\langle i, w \rangle \ge e(w)$.
For $w \in W$, we therefore define the function
\[ f_w \colon J \To \Z_{\ge 0}\,, \quad
              i \longmapsto \max\{0, e(w) - \langle i, w \rangle\}
  \,.
\]
Then $F$ is unstable for~$(E, w)$ if and only if $v_p(a_i) \ge f_w(i)$ for all~$i \in J$.
This implies that
\begin{equation} \label{E:DomF}
  \text{$w$ dominates~$w'$ if and only if $f_{w'} \ge f_w$ (pointwise).}
\end{equation}

Since we can always adjust by a permutation, it suffices
to consider weight vectors with weakly increasing entries. Also, since
$F$ is unstable at~$p$ for~$(E, w)$ if and only if $F$ is unstable at~$p$
for~$(E, w + \one)$, where $\one = \one_{n+1}$ is the vector $[1, \dots, 1]$
of length~$n+1$, it is sufficient to consider weight vectors
whose minimal entry is zero.

\begin{Definition} \label{D:Normalized}
  We say that a weight vector~$w$ is \emph{normalized} if
  \[ 0 = w_0 \le w_1 \le \dots \le w_n \,. \]
  We denote by~$W = W_n$ the set of all normalized weight vectors of
 length~$n+1$.
\end{Definition}

It then suffices to consider subsets~$S$ of~$W$; we call such
sets \emph{sets of normalized weight vectors}.
We will now show that we can simplify the condition for completeness
for sets of normalized weight vectors.

\begin{Lemma} \label{L:DomNorm}
  Let $w, w' \in W$ be such that $w'$ dominates the permutation~$w^\sigma$
  of~$w$ (where $\sigma$ is a permutation of $\{0,1,\ldots,n\}$ and
  $w^\sigma_k = w_{\sigma(k)}$). Then $w'$ dominates~$w$.
\end{Lemma}

\begin{proof}
  We can assume that $w^\sigma \neq w$. Then there are indices $0 \le k < l \le n$
  such that $w^\sigma_k > w^\sigma_l$. Let $\tau$ be the transposition
  swapping $k$ and~$l$. Then $w^{\sigma\tau}$ is strictly less than~$w^\sigma$
  in the lexicographic ordering. Since the set of permutations of~$w$
  is finite, it suffices to show that $w'$ dominates~$w^{\sigma\tau}$:
  after a finite number of such steps, we must reach~$w$, which is the
  lexicographically smallest vector among all its permutations.

  To simplify notation, we set $u = w^\sigma$. We know that $f_{w'} \le f_u$
  and have to show that $f_{w'} \le f_{u^\tau}$. Consider $i \in J$.
  Then for any weight vector~$\tilde{w}$, we have that
  \begin{equation} \label{E:PairSwap}
    \langle i, \tilde{w} \rangle - \langle i^\tau, \tilde{w} \rangle
       = \langle i, \tilde{w} \rangle - \langle i, \tilde{w}^\tau \rangle
       = \langle i, \tilde{w} - \tilde{w}^\tau \rangle
       = (i_k - i_l) (\tilde{w}_k - \tilde{w}_l) \,.
  \end{equation}
  If $i_k \ge i_l$, then~\eqref{E:PairSwap} implies that
  $f_{w'}(i) \le f_u(i) \le f_{u^\tau}(i)$, and we are done. So we can assume
  that $i_k < i_l$. Then $\langle i, w' \rangle \ge \langle i^\tau, w' \rangle$
  and so $f_{w'}(i) \le f_{w'}(i^\tau) \le f_u(i^\tau) = f_{u^\tau}(i)$.
\end{proof}

\begin{Corollary} \label{C:CompleteNorm}
  Let $S \subseteq W$ be such that every $w \in W$ is dominated by some
  $w' \in S$. Then $S$ is complete.
\end{Corollary}

The description of dominance given by~\eqref{E:DomF}
leads to an easy proof that a finite set of weight vectors is always sufficient.

\begin{Proposition} \label{P:finiteness}
  Fix $n \ge 1$ and~$d \ge 1$. Then there is a finite complete set of
  normalized weight vectors
  for forms of degree~$d$ in $n+1$~variables, and every minimal complete
  set of normalized weight vectors for these parameters is finite.
\end{Proposition}

\begin{proof}
  We can consider $f_w$ as a point in $\Z_{\ge 0}^J$. Then $w$ dominates~$w'$ if and only if
  \hbox{$f_w \le f_{w'}$} in the product order on~$\Z_{\ge 0}^J$.
  By Dickson's Lemma (which follows from the fact that a polynomial ring in
  finitely many variables over a field is noetherian, applied to monomial ideals),
  the non-empty set \hbox{$\{f_w : w \in W\} \subseteq \Z_{\ge 0}^J$} has finitely
  many minimal elements, and each
  element of the set is bounded below by a minimal one.
  The corresponding vectors~$w$ then form a finite complete set of weight vectors.
  The minimal complete sets of weight vectors are obtained by taking one~$w \in W$
  such that $f_w = s$ for each minimal element~$s$ of~$\{f_w : w \in W\}$, so in particular,
  such a minimal set is finite (and all minimal complete sets of weight vectors
  for given parameters $n$ and~$d$ have the same cardinality).
\end{proof}

We will now show that there is in fact a \emph{unique} minimal complete set
of normalized weight vectors. By the description of the minimal complete sets
in the proof of Proposition~\ref{P:finiteness}, this amounts to showing that
minimal elements in~$\{f_w : w \in W\}$ have unique preimages in~$W$.

\begin{Definition} \label{D:balanced}
  We say that a weight vector~$w$ is \emph{balanced} if all its entries
  are $\le e(w)$. We define the \emph{truncation}~$w^\downarrow$ of~$w$ to be the vector
  with $k$th entry $\min\{w_k, e(w)\}$.
\end{Definition}

Note that $w$ is balanced if and only if $w = w^\downarrow$.

\begin{Lemma} \label{L:TruncDom}
  Let $w$ be a weight vector.
  \begin{enumerate}[\upshape(i)]
    \item \label{TruncDom1} $w^\downarrow$ dominates~$w$.
    \item \label{TruncDom2} If $e(w^\downarrow) = e(w)$, then $w$ dominates $w^\downarrow$.
    \item \label{TruncDom3} If $w$ dominates $w^\downarrow$ and $w_0 = 0$, then $e(w^\downarrow) = e(w)$.
  \end{enumerate}
\end{Lemma}

\begin{proof}
  We clearly always have $e(w^\downarrow) \le e(w)$.
  \begin{enumerate}[(i)]
    \item Let $i \in J$. If $i_k > 0$ for some index~$k$ such that $w_k \ge e(w)$, then
          $f_{w^\downarrow}(i) = f_w(i) = 0$, and there is nothing to prove.
          Otherwise, $\langle i, w^\downarrow \rangle = \langle i, w \rangle$
          (since $w^\downarrow$ and~$w$ agree on the support of~$i$) and therefore
          \[ f_{w^\downarrow}(i) = \max\{0, e(w^\downarrow) - \langle i, w^\downarrow \rangle\}
                                 = \max\{0, e(w^\downarrow) - \langle i, w \rangle\}
                                 \le \max\{0, e(w) - \langle i, w \rangle\}
                                 = f_w(i) \,,
          \]
          which shows the claim.
    \item If $e(w^\downarrow) = e(w)$, then the inequality in the preceding proof
          is an equality, and we obtain that $f_{w^\downarrow} = f_w$; in particular,
          $w$ dominates $w^\downarrow$.
    \item Taking $i = [d,0,\ldots,0]$, we have
          $\langle i, w \rangle = \langle i, w^\downarrow \rangle = 0$. Then
          \[ e(w^\downarrow) = f_{w^\downarrow}(i) = f_w(i) = e(w) \,. \qedhere \]
  \end{enumerate}
\end{proof}

In particular, if $w^\downarrow \neq w$ and $e(w^\downarrow) = e(w)$,
then $f_{w^\downarrow} = f_w$, and $f_w$ has several preimages in~$W$.
Such $w \in W$ exist if and only if $d \le n$. (We leave the proof as an
exercise for the reader.)

\begin{Lemma} \label{L:DomBalanced}
  Let $w, w'$ be two weight vectors such that $w_0 = 0$, $w'$ is balanced,
  and $e(w) = e(w')$.
  Then $w$ dominates~$w'$ if and only if $w_k \ge w'_k$ for all $0 \le k \le n$.
\end{Lemma}

\begin{proof}
  First assume that $w \ge w'$ component-wise. Then
  $\langle i, w \rangle \ge \langle i, w' \rangle$ for all $i \in J$, which
  implies that $w$ dominates~$w'$ by~\eqref{E:DomF} (using $e(w) = e(w')$).

  For the other direction, assume that $w$ dominates~$w'$ and let $k \in \{0,1,\ldots,n\}$.
  By assumption, we have that $w_0 = 0 \le w'_0$; so we
  can assume $k \ge 1$. Consider $i = [d-1, 0, \ldots, 0, 1, 0, \ldots, 0]$
  with the entry~$1$ at index~$k$. Then (using that $w'$ is balanced)
  \[ e(w) - w_k \le \max\{0, e(w) - w_k\} = f_w(i) \le f_{w'}(i) = e(w') - w'_k = e(w) - w'_k \, \]
  so we obtain $w_k \ge w'_k$ as desired.
\end{proof}

\begin{Lemma} \label{L:MinBalanced}
  Let $w \in W$ be such that $f_w$ is minimal in~$\{f_u : u \in W\}$.
  Then $w$ is balanced.
\end{Lemma}

\begin{proof}
  By Lemma~\ref{L:TruncDom}~\eqref{TruncDom1}, $w^\downarrow$ dominates~$w$.
  By parts \eqref{TruncDom2} and~\eqref{TruncDom3} of the lemma, the domination is
  strict if $e(w^\downarrow) < e(w)$, so in this case $f_w$ cannot be minimal.
  We must therefore have $e(w^\downarrow) = e(w)$.
  Assume that $w$ is not balanced; then $w^\downarrow \neq w$, and we have
  $\Sigma w^\downarrow \le \Sigma w - 1$. Replacing the entry $w^\downarrow_0 = 0$
  by~$1$ results in a weight vector~$w'$ that satisfies $e(w') = e(w)$
  (this follows from $\Sigma w^\downarrow < \Sigma w' \le \Sigma w$ and
  $e(w^\downarrow) = e(w)$), is balanced and strictly greater than~$w^\downarrow$
  in the product order, so by Lemma~\ref{L:DomBalanced}, $w'$ (which is balanced)
  strictly dominates~$w^\downarrow$ (which has $w^\downarrow_0 = 0$)
  and therefore also~$w$. Let $\sigma$ be the permutation
  such that ${w'}^\sigma$ is increasing; then ${w'}^\sigma$ strictly dominates~$w^\sigma$.
  If ${w'_0}^\sigma > 0$ (then it must be~$1$), set $w'' = {w'}^\sigma - \one$,
  else set $w'' = {w'}^\sigma$. Then $w''$ is normalized and strictly dominates~$w^\sigma$
  and therefore also~$w$ by Lemma~\ref{L:DomNorm}. This implies again that
  $f_w$ cannot be minimal. This contradiction shows that $w$ must be balanced.
\end{proof}

Lemma~\ref{L:DomBalanced} then implies that $w$ must be maximal with respect
to the product order in the set of all balanced (normalized) weight vectors with
exponent~$e(w)$.

\begin{Proposition}
  Let $w \in W$ be such that $f_w$ is minimal in~$\{f_u : u \in W\}$.
  Then $w$ is the unique preimage of~$f_w$ under $u \mapsto f_u$.
  In particular, there is exactly one minimal complete set of normalized weight vectors,
  which consists of all $w \in W$ such that $f_w$ is minimal.
\end{Proposition}

\begin{proof}
  Assume that $w' \in W$ is such that $f_{w'} = f_w$. It follows that
  \[ e(w') = f_{w'}([d,0,\ldots,0]) = f_w([d,0,\ldots,0]) = e(w) \,. \]
  By Lemma~\ref{L:MinBalanced}, both $w$ and~$w'$ are balanced. Since $w$
  and~$w'$ dominate each other, Lemma~\ref{L:DomBalanced} implies
  that they are equal. The last statement then follows from the discussion
  following Proposition~\ref{P:finiteness}.
\end{proof}

We obtain the following simple sufficient condition for dominance.
For $w \in W$ and $i \in J$, we set $v_i = d \one - (n + 1) i$.
Then
\[ e(w) - \langle i, w \rangle
       = \Bigl\lfloor \frac{\langle v_i, w \rangle}{n+1} \Bigr\rfloor + 1 \,.
\]

\begin{Lemma} \label{LemmaDom}
  Let $w', w \in W$. If $\langle v_i, w' \rangle \ge \langle v_i, w \rangle$
  for all $i \in J$ such that $\langle v_i, w \rangle \ge 0$, then $w$ dominates~$w'$.
\end{Lemma}

\begin{proof}
  By~\eqref{E:DomF}, $w$ dominates~$w'$ if $f_w(i) \le f_{w'}(i)$
  for all $i \in J$. If $\langle v_i, w \rangle < 0$, then $f_w(i) = 0$
  and there is nothing to show. Otherwise,
  \[ f_w(i) = \Bigl\lfloor \frac{\langle v_i, w \rangle}{n+1} \Bigr\rfloor + 1
            \le \Bigl\lfloor \frac{\langle v_i, w' \rangle}{n+1} \Bigr\rfloor + 1
            = f_{w'}(i)
  \]
  by our assumption.
\end{proof}

Here is a geometric interpretation of the criterion in Lemma~\ref{LemmaDom}.
For every $i \in J$, the condition `$\langle v_i, w \rangle \ge 0$'
defines a closed half-space $H_i \subseteq \R^{n+1}$. For a given \hbox{$w \in W$},
let $C(w) = \bigcap_{i \in J: w \in H_i} H_i$ denote the cone that is the
intersection of the half-spaces containing~$w$. Then all weights
that lie in the shifted cone $w + C(w)$ are dominated by~$w$.

If we write the weight vectors as $[0, z_1, z_1+z_2, \dots, z_1+\dots+z_n]$
with $z_j \ge 0$, then we get a similar picture in~$\R^n$ for the coordinates
$z_j$. For $n = 2$ and some values of~$d$, this is shown in Figure~\ref{FigDom}.
We set $x = z_1$ and $y = z_2$.
The light blue area is the region $x+y \le d$; the cones in
shades between green and red are shifted cones $w + C(w)$ for the points~$w$
in the blue triangle. We see in each case that all lattice points in the
positive quadrant are covered by these shifted cones (which sometimes
degenerate into rays); this illustrates Theorem~\ref{Thmn=2}.

\begin{Remark} \label{Rk:n=2,d=4}
  The criterion given in~Lemma~\ref{LemmaDom} is not an equivalence, since
  the implication `$x \le y \Rightarrow \lfloor x \rfloor \le \lfloor y \rfloor$'
  is not an equivalence. This
  allows a slightly larger value of $\langle v_i, w \rangle$ than $\langle v_i, w' \rangle$ in
  some cases. For example, it turns out that for $n = 2$, $d = 4$, $[0,1,2]$
  is actually dominated by~$[0,1,1]$, even though the criterion of Lemma~\ref{LemmaDom}
  is not satisfied. Here is a table of the values of $\langle v_i, w \rangle$
  for $w = [0,1,1]$ and $w = [0,1,2]$.
  \[ \renewcommand{\arraystretch}{1.2}
     \begin{array}{|r|c*{14}{@{\;\;}c}|} \hline
        i       & 004 & 013 & 022 & 031 & 040 & 103 & 112 & 121 & 130 & 202 & 211 & 220 & 301 & 310 & 400 \\
        \hline
      {}[0,1,1] & <0  & <0  & <0  & <0  & <0  & <0  & <0  & <0  & <0 & \mathbf 2 & 2 & 2 &  5 &  5  &  8 \\
      {}[0,1,2] & <0  & <0  & <0  & <0  &  0  & <0  & <0  &  0  &  3 & \mathbf 0 & 3 & 6 &  6 &  9  & 12 \\
        \hline
     \end{array}
  \]
  We see that the criterion of Lemma~\ref{LemmaDom} is satisfied for all~$i$
  except $i = [2,0,2]$. However, both $[0,1,1]$ and~$[0,1,2]$ give the same
  value~$0$ for $\lfloor \langle v_i, w \rangle/3 \rfloor$,
  and hence $[0,1,1]$ indeed dominates~$[0,1,2]$. So the minimal complete
  set of weight vectors for this case is $\{[0,0,1], [0,1,1], [0,1,3]\}$
  instead of $\{[0,0,1], [0,1,1], [0,1,2], [0,1,3]\}$.

  When $d$ is a multiple of~$n+1$, however, then $\langle v_i, w \rangle$
  is always divisible by~$n+1$, and thus the criterion is indeed an equivalence.
\end{Remark}

\begin{figure}[t]
  \vspace*{10pt}

  \begin{tabular}{ccc}
   \Gr{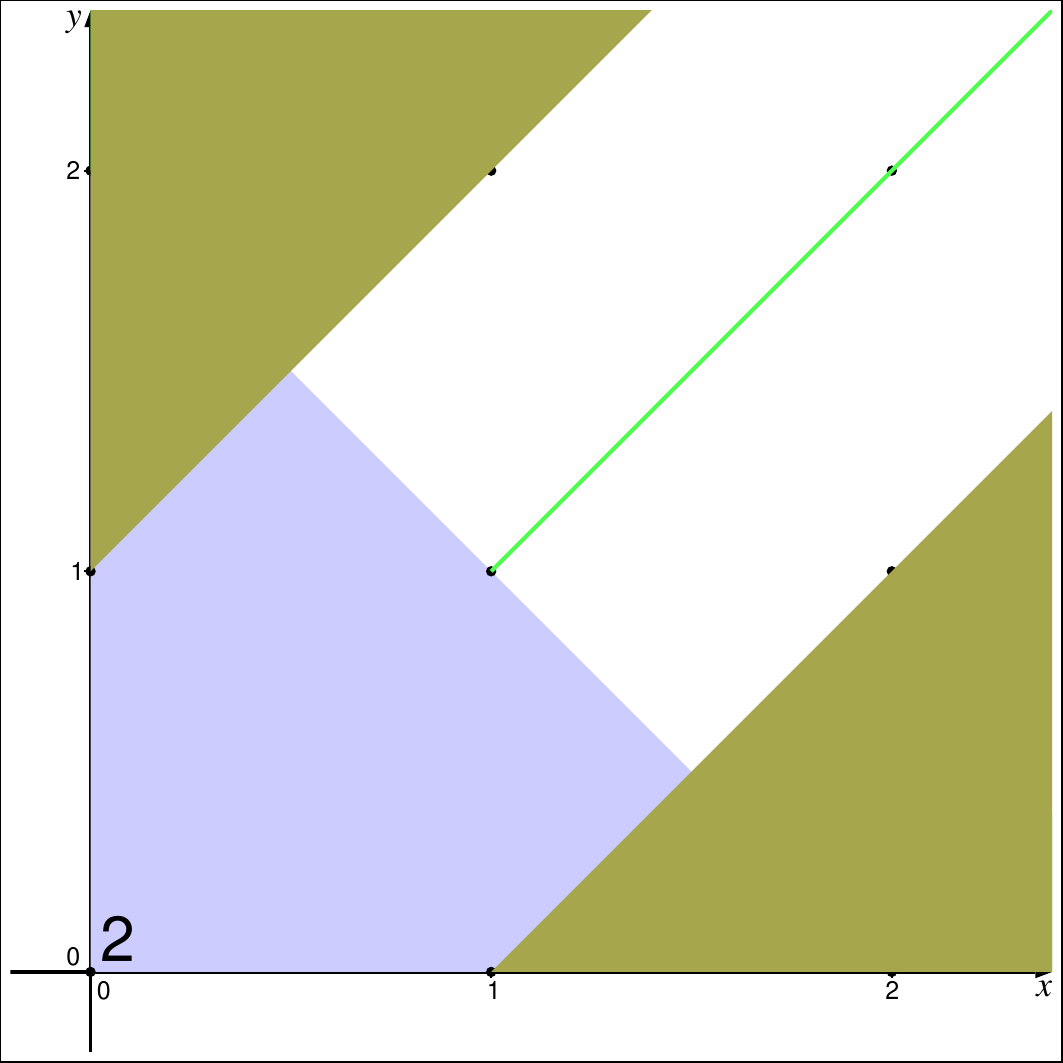}{0.3\textwidth} &
   \Gr{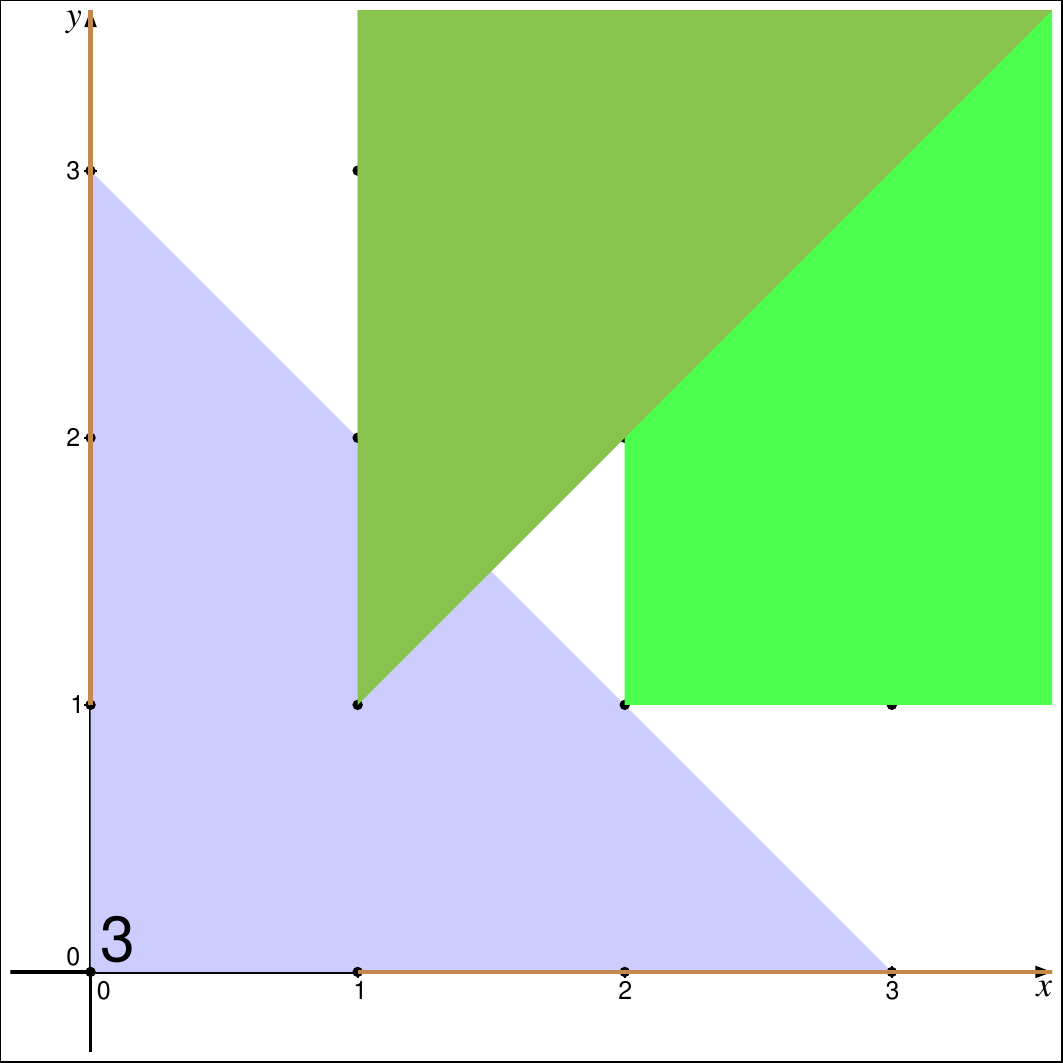}{0.3\textwidth} &
   \Gr{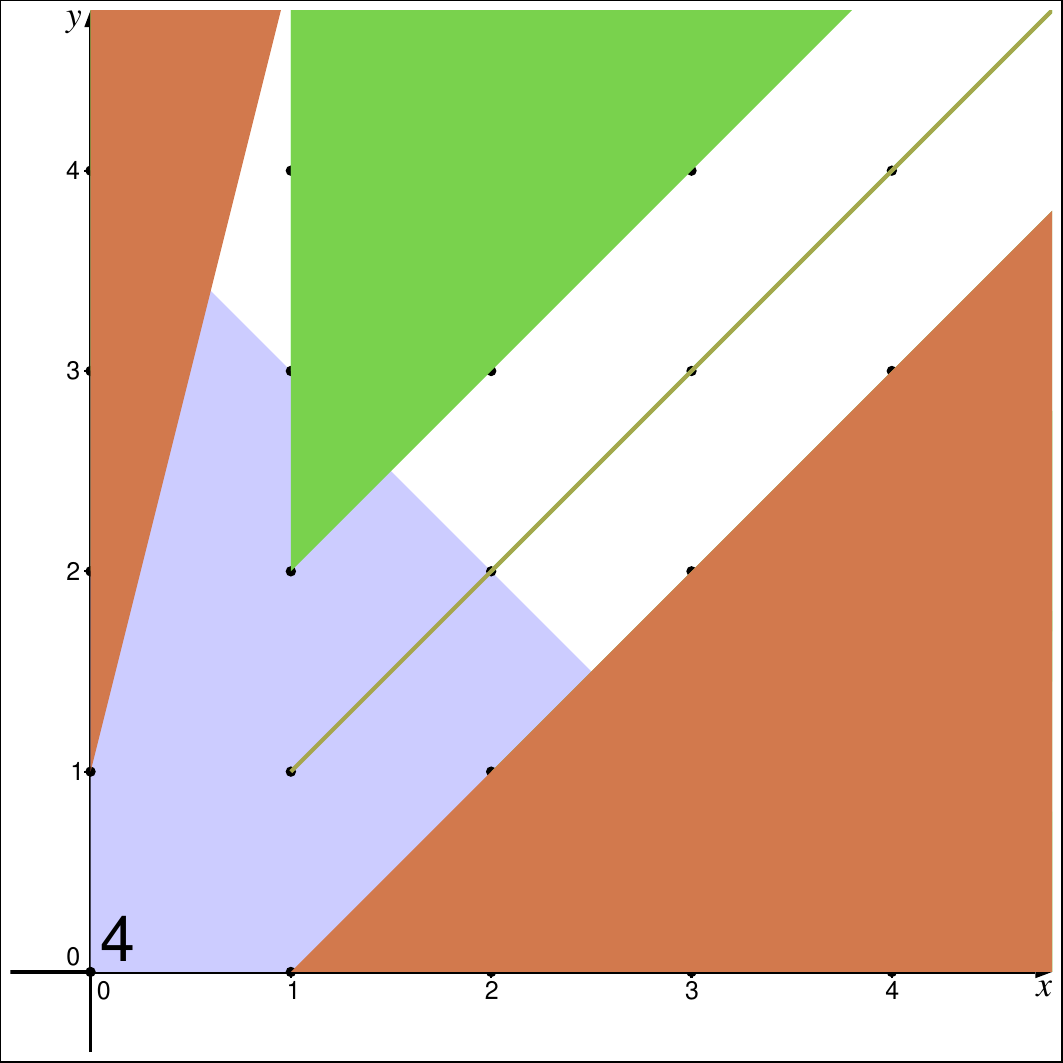}{0.3\textwidth} \\[8pt]
   \Gr{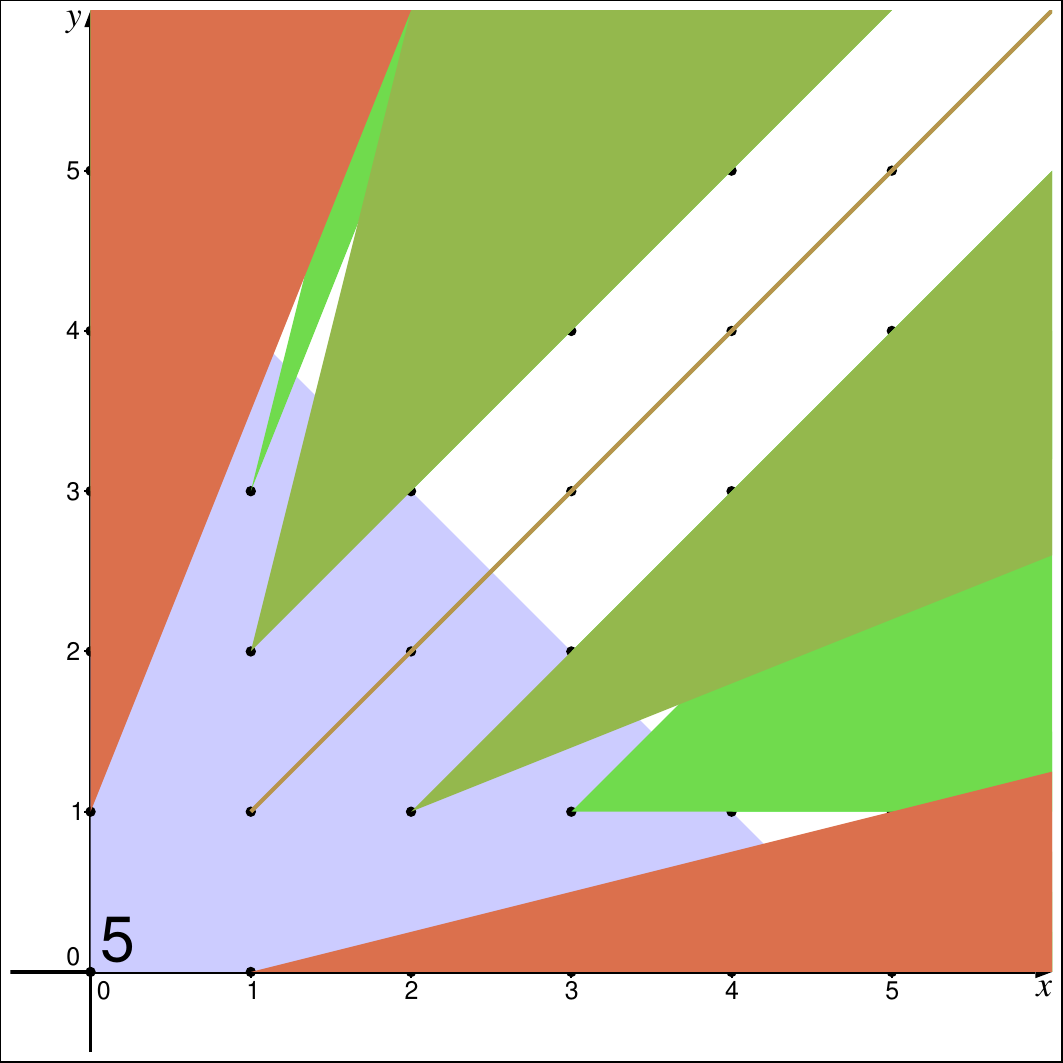}{0.3\textwidth} &
   \Gr{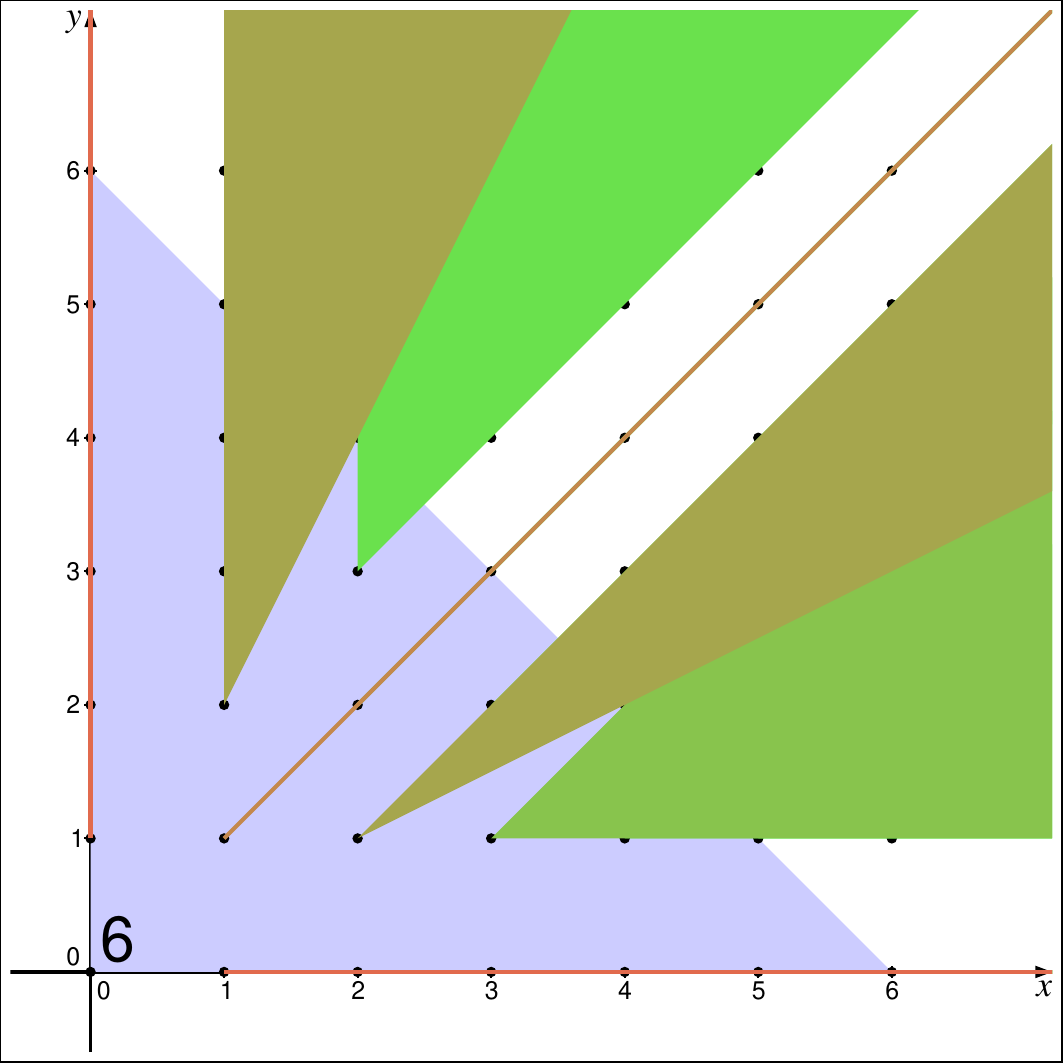}{0.3\textwidth} &
   \Gr{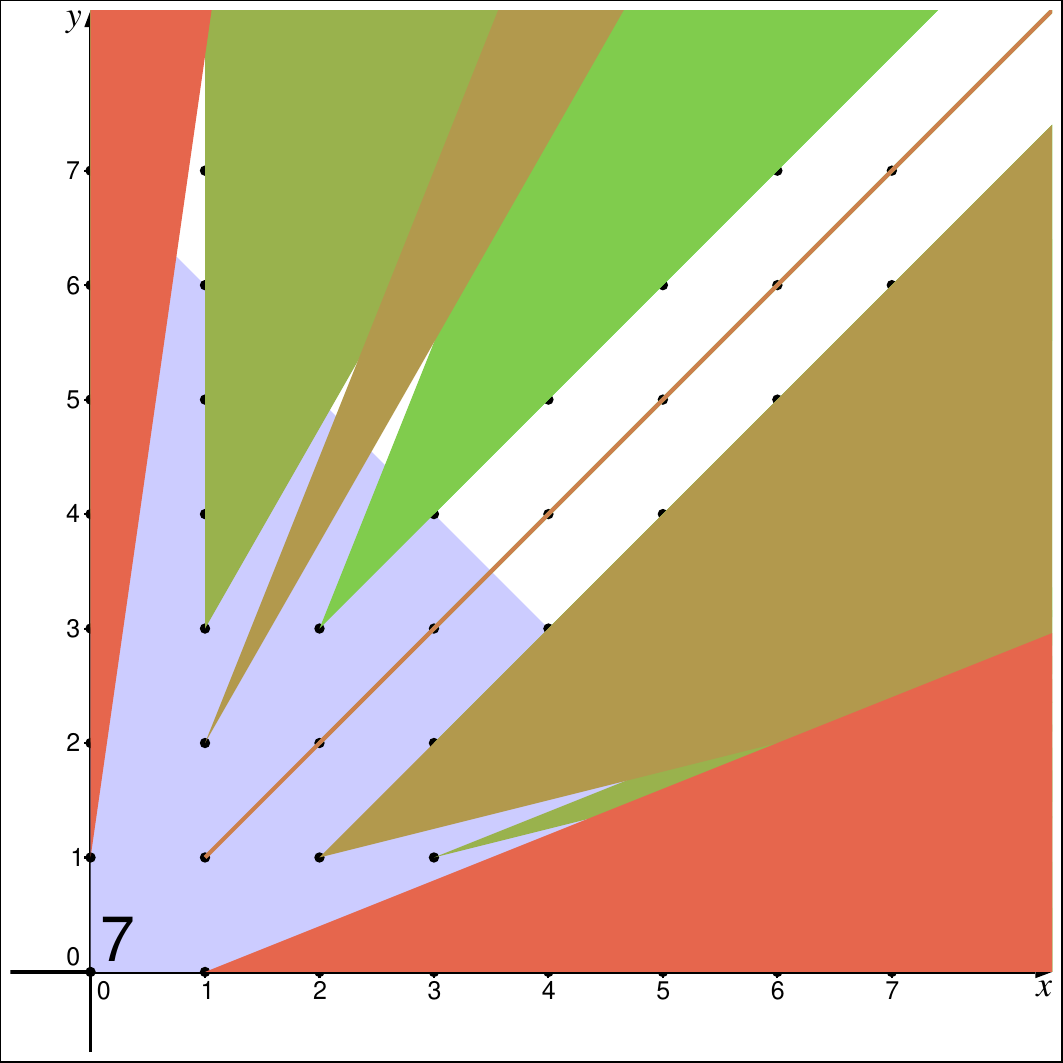}{0.3\textwidth} \\[8pt]
   \Gr{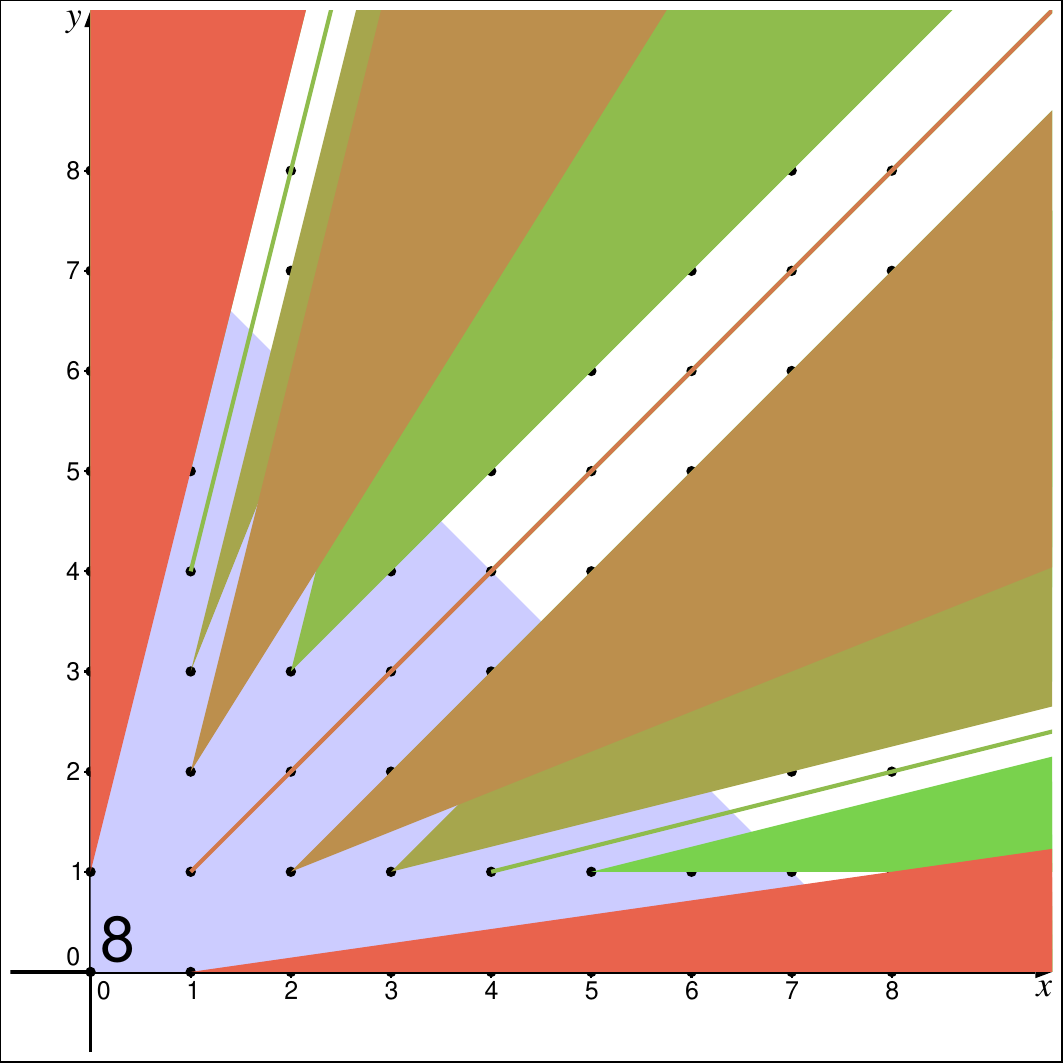}{0.3\textwidth} &
   \Gr{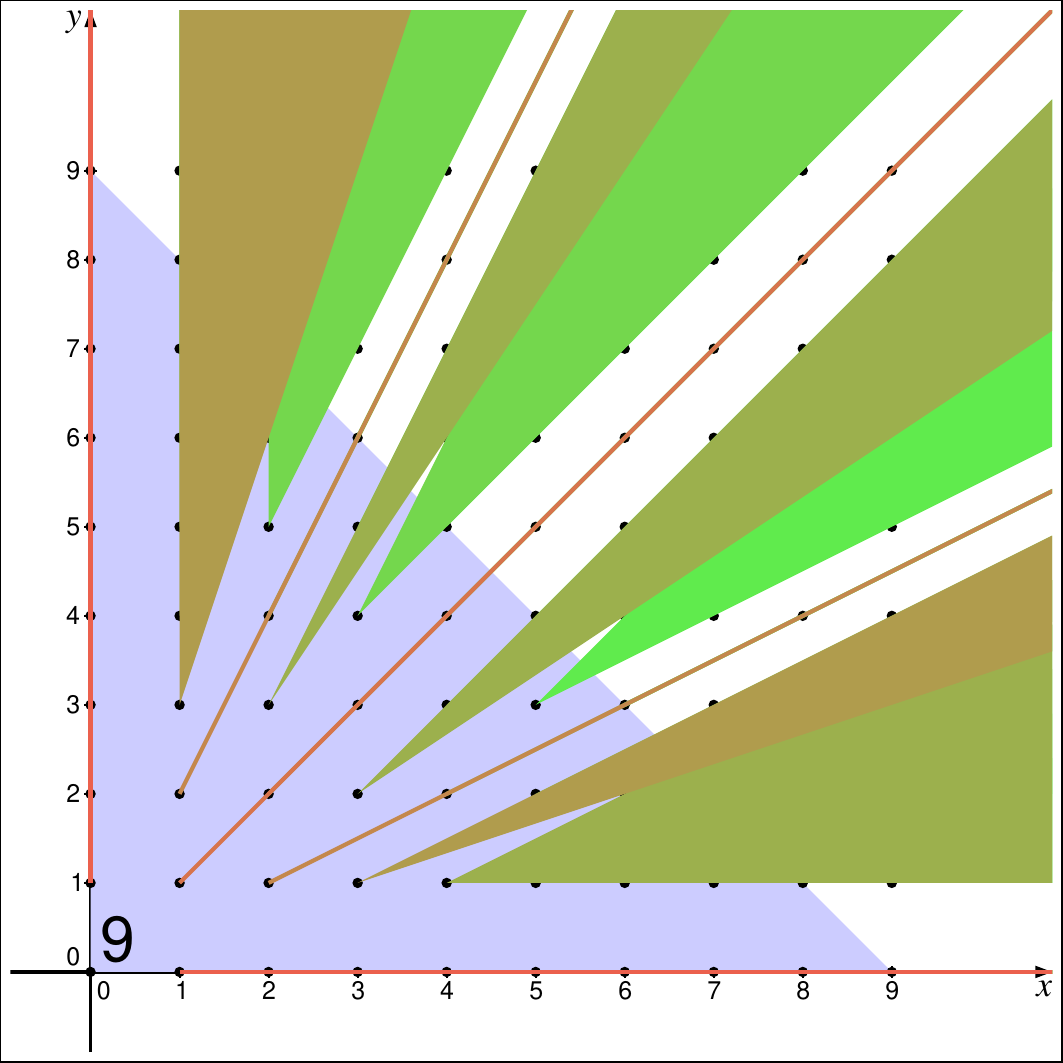}{0.3\textwidth} &
   \Gr{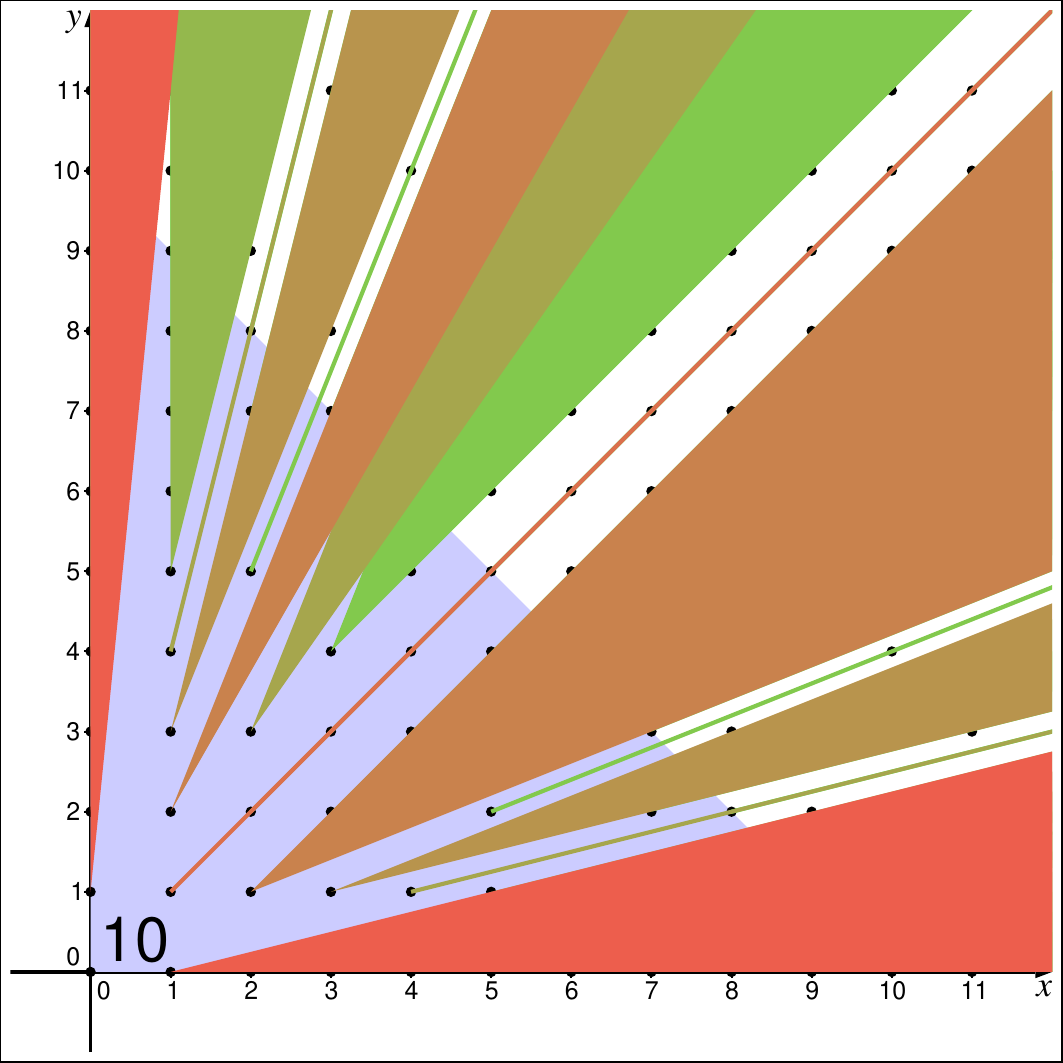}{0.3\textwidth} \\[8pt]
   \Gr{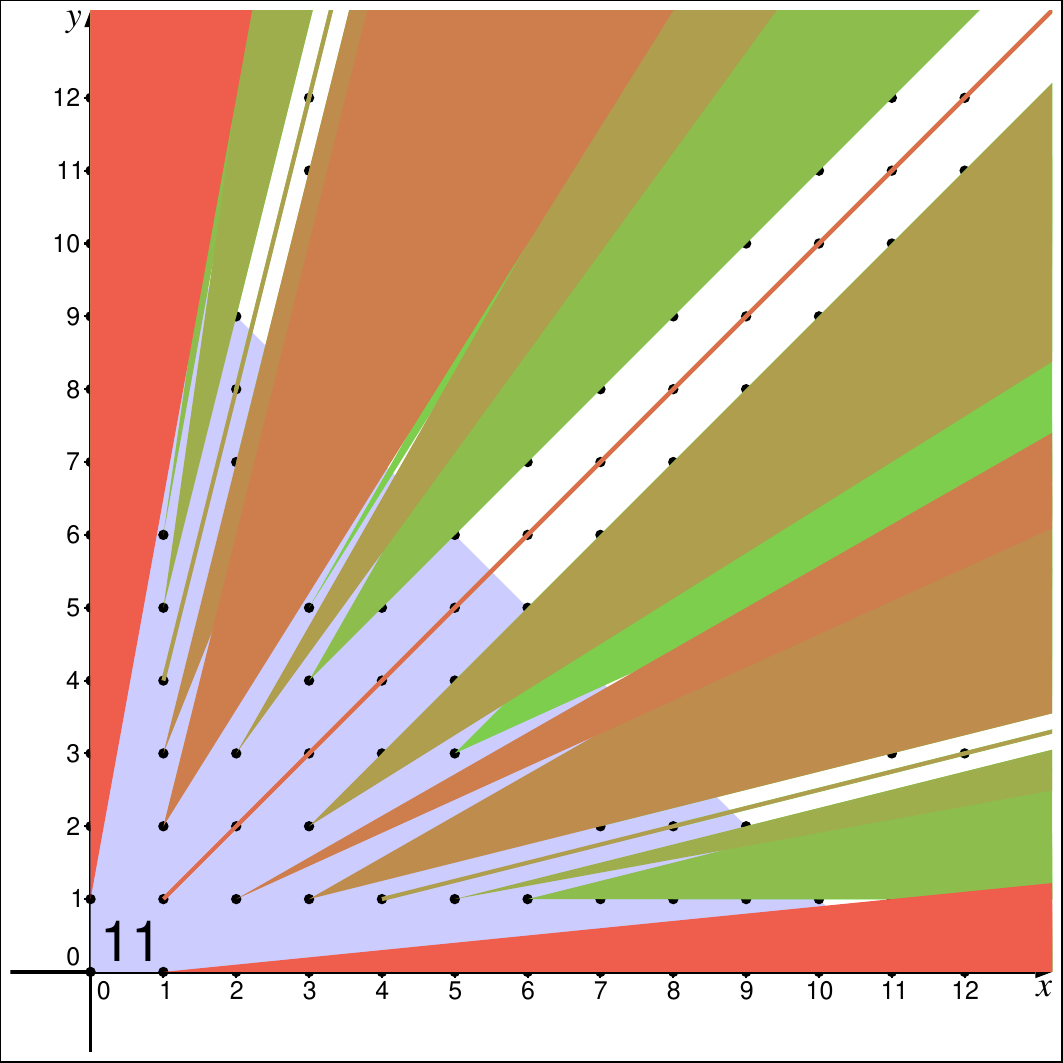}{0.3\textwidth} &
   \Gr{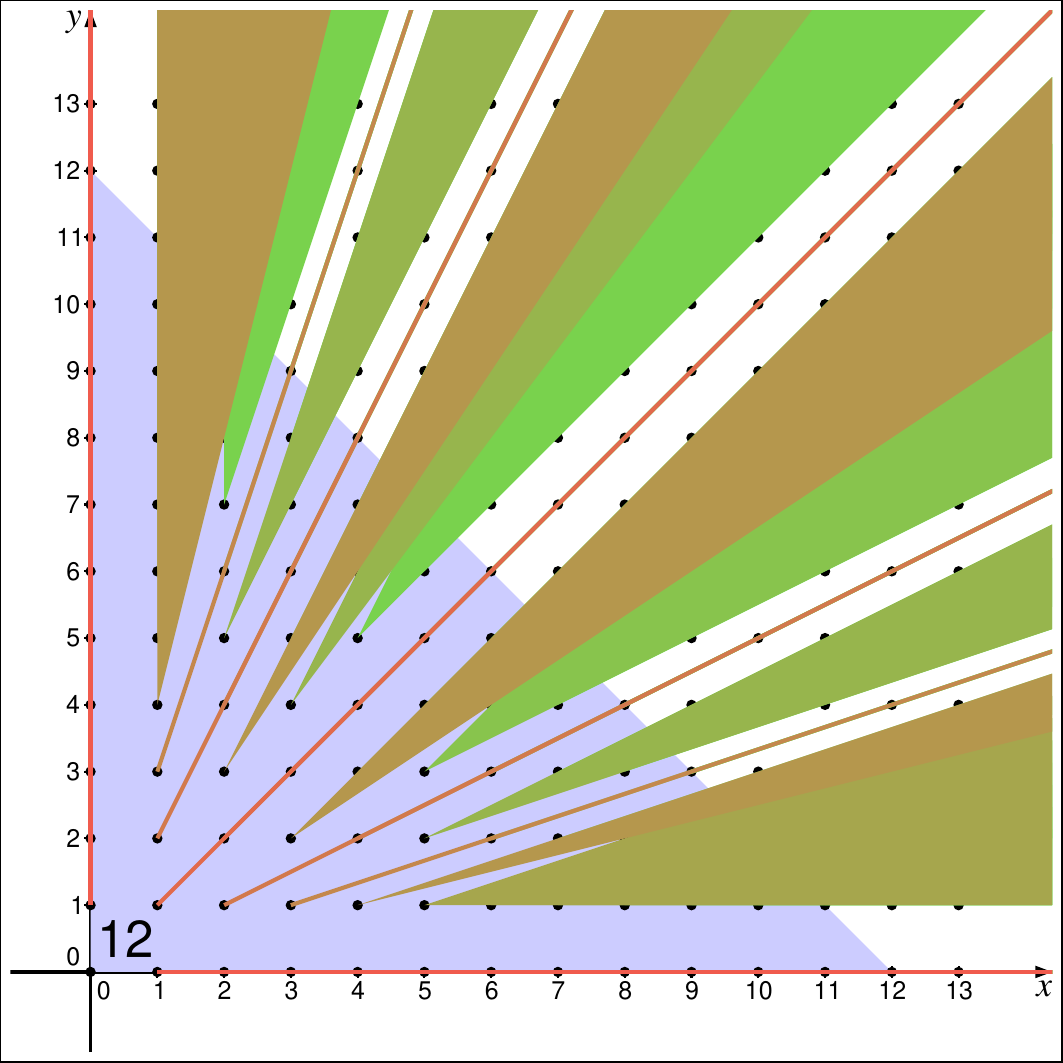}{0.3\textwidth} &
   \Gr{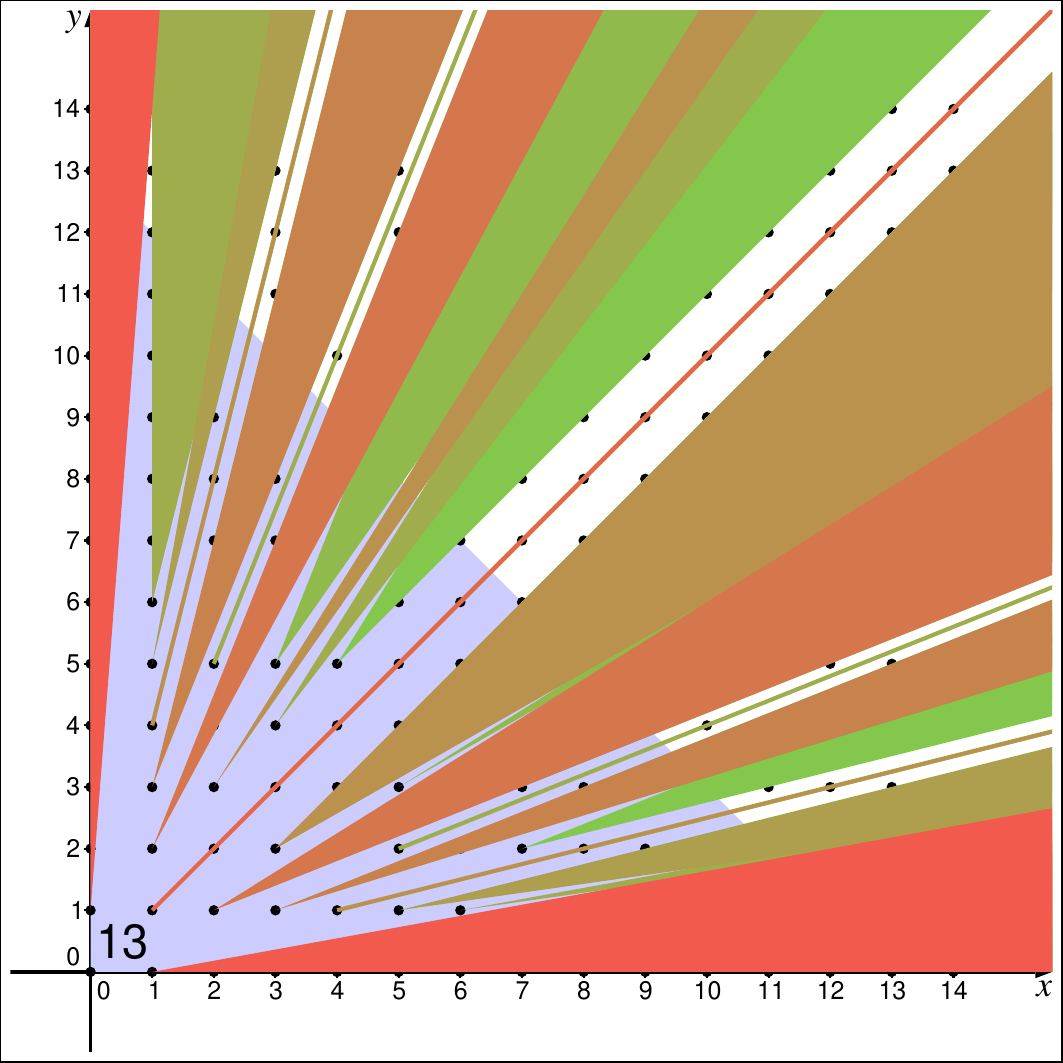}{0.3\textwidth}
  \end{tabular}

  \caption{Complete set of weight vectors for $n = 2$ and $d = 2, 3, \ldots, 13$. \newline
           The lattice points are $[z_1, z_2]$, corresponding to $w = [0, z_1, z_1+z_2]$;
           the colored wedges contain the vectors dominated by their vertex.}
  \label{FigDom}
\end{figure}

From now on, we will work with the coordinates $z_1, \dots, z_n$ in~$\R^n$.
In particular, we identify $W_n$ with $\Z_{\ge 0}^n \subseteq \R^n$.


\section{Proof of Theorem~\ref{Thmn=2}} \label{S:n=2}

Fix the degree~$d$.
To prove the statement of Theorem~\ref{Thmn=2}, it is sufficient to show that
every weight vector $w = [0, z_1, z_1+z_2]$ with $z_1+z_2 > d$ is dominated
by another weight vector whose last coordinate is $\le d$.
We write $\|w\| = z_1 + z_2$.

Since a weight vector~$w$ dominates all multiples~$mw$ with $m \ge 1$, we can
assume that $w$ is primitive, so $\gcd(z_1, z_2) = 1$. We then have a one-to-one
correspondence between primitive weight vectors and fractions~$z_2/z_1$
between $0 = 0/1$ and $+\infty = 1/0$. We will write $\zeta$ for the fraction
associated to~$w$ in this way.

Let $I = \interval{\zeta}{\zeta'}$ be an interval with rational endpoints satisfying
$0 \le \zeta < \zeta' \le \infty$. (To avoid confusion with our notation for
vectors, we use boldface square brackets to denote closed intervals.)
We say that $I$ is \emph{basic} if
$\zeta = a/b$, $\zeta' = a'/b'$ in lowest terms with $a'b - ab' = 1$.
It is well-known that every nonnegative rational number occurs as an endpoint
of a basic interval and that if $c/d \in I$, then $[c, d] = k [a, b] + k' [a', b']$
with $k, k' \in \Z_{\ge 0}$.

To show that a given weight vector~$w$ is dominated by a
weight vector~$w'$ with $\|w'\| \le d$, we will use the criterion of Lemma~\ref{LemmaDom}.
Consider some $i = [i_0, i_1, i_2] \in J = J_{2,d}$; then
\[ \langle v_i, w \rangle = (2 d - 3 i_1 - 3 i_2) z_1 + (d - 3 i_2) z_2
                          = g (a_i z_1 + b_i z_2) \,,
\]
where $g = \gcd(d - 3 i_1, d - 3 i_2)$ and $a_i = (2 d - 3 i_1 - 3 i_2)/g$,
$b_i = (d - 3 i_2)/g$. Then $\langle v_i, w \rangle \ge 0$ for all~$w$
when $a_i, b_i \ge 0$ and $\langle v_i, w \rangle < 0$ for all~$w$
when $a_i, b_i < 0$. When $a_i \ge 0 > b_i$, the condition on~$w$ to have
$\langle v_i, w \rangle \ge 0$ is $\zeta \le |a_i/b_i|$, whereas when $b_i \ge 0 > a_i$,
the condition is $\zeta \ge |a_i/b_i|$. We note that $g(|a_i| + |b_i|) \le d$
in the first case and $\max\{g|a_i|, g|b_i|\} \le d$ in the second case.
We set
\[ S_{\le} = \Bigl\{-\frac{a_i}{b_i} : i \in J, a_i \ge 0 > b_i\Bigr\} \quad\text{and}\quad
   S_{\ge} = \Bigl\{-\frac{a_i}{b_i} : i \in J, b_i \ge 0 > a_i\Bigr\} \,.
\]

\begin{Lemma} \label{LemmaI}
  Let $I = \interval{\zeta_-}{\zeta_+}$ be a basic interval and let $w$ be a primitive
  weight vector such that $\zeta \in I$ for the associated fraction~$\zeta$.
  Write $w_-$ and~$w_+$ for the primitive weight vectors associated to $\zeta_-$
  and~$\zeta_+$, respectively.

  If $I \cap S_{\le} \subseteq \{\zeta_+\}$ or $I \cap S_{\ge} \subseteq \{\zeta_-\}$,
  then $w$ is dominated by $w_-$ or by~$w_+$.
\end{Lemma}

\begin{proof}
  If $\zeta = \zeta_-$ or $\zeta = \zeta_+$, then the claim is trivially true.
  So we now assume that $\zeta_- < \zeta < \zeta_+$; then $w = k_- w_- + k_+ w_+$
  with $k_-, k_+ \in \Z_{\ge 1}$. We also assume that $I \cap S_{\le} \subseteq \{\zeta_+\}$;
  we claim that $w_-$ dominates~$w$ in this case. We use the criterion
  of Lemma~\ref{LemmaDom}. So consider $i \in J$ such that $\langle v_i, w_- \rangle \ge 0$.
  Then not both of $a_i$ and~$b_i$ can be negative.
  We claim that $\langle v_i, w_+ \rangle \ge 0$ as well. This is clear
  if $a_i, b_i \ge 0$ and also if $b_i \ge 0 > a_i$ (since $\zeta_+ > \zeta_-$).
  If $a_i \ge 0 > b_i$, let $\zeta_i = -a_i/b_i$; then $\zeta_i \in S_{\le}$ and
  the condition on~$w_-$ is $\zeta_- \le \zeta_i$. Our assumption on~$I$
  then implies that $\zeta_i \ge \zeta_+$, so $\langle v_i, w_+ \rangle \ge 0$ as well.
  Now
  \[ \langle v_i, w \rangle = \langle v_i, w_- \rangle
                                + \bigl((k_- - 1) \langle v_i, w_- \rangle + k_+ \langle v_i, w_+ \rangle\bigr)
                            \ge \langle v_i, w_- \rangle \,,
  \]
  so the criterion is satisfied. In the case that $I \cap S_{\ge} \subseteq \{\zeta_-\}$,
  we show in the same way (mutatis mutandis) that $w_+$ dominates~$w$.
\end{proof}

The idea for the proof of Theorem~\ref{Thmn=2} is now to cover $\interval{0}{\infty}$
with basic intervals whose endpoints dominate everything in the interval.
We use minimal basic intervals whose endpoints are fractions in lowest terms such that the
sum of their numerator and denominator is bounded by~$d$. This means that
$I = \interval{a_-/b_-}{a_+/b_+}$ with
\[ a_- + b_- \le d\,, \quad a_+ + b_+ \le d\,, \quad
   a_+ b_- - a_- b_+ = 1 \quad\text{and}\quad a_- + a_+ + b_- + b_+ > d\,.
\]
We will call such intervals \emph{feasible}.
These feasible intervals cover $\interval{0}{\infty}$.
To show this, we first note that $\interval{d-1}{\infty} = \interval{\frac{d-1}{1}}{\frac{1}{0}}$
is feasible. Further, $\interval{0}{d-1}$ is covered by the basic intervals
\[ \interval{0}{1},\interval{1}{2},\ldots,\interval{d-2}{d-1}\, , \]
which may not be feasible as
the sum condition $a_- + a_+ + b_- + b_+ > d$ may not be satisfied.
However, every basic interval
$\interval{a_-/b_-}{a_+/b_+}$ can be split into
\[ \interval{a_-/b_-}{(a_- + a_+)/(b_- + b_+)} \cup \interval{(a_- + a_+)/(b_- + b_+)}{a_+/b_+}\, . \]
An iteration of this splitting results in a decomposition into feasible intervals.

We note that if a fraction $a/b$ lies in the interior of a feasible
interval~$\interval{a_-/b_-}{a_+/b_+}$, then $a + b > d$. This is because
$[a, b] = k_- [a_-, b_-] + k_+ [a_+, b_+]$ with $k_-, k_+ \ge 1$, so
$a + b = k_- a_- + k_+ a_+ + k_- b_- + k_+ b_+ \ge a_- + a_+ + b_- + b_+ > d$.

We show that Lemma~\ref{LemmaI} applies to each such interval.
Since $\|w_-\|, \|w_+\| \le d$, the theorem then follows.

We first consider the case that $d$ is a multiple of~$3$, so $d = 3\delta$
with $\delta \in \Z_{\ge 1}$. In this case $g$ as defined above is always divisible by~$3$.
It follows that
\[ S_{\le} \subseteq \Bigl\{\frac{a}{b} : a, b \ge 0,\; a \perp b,\; a + b \le \delta\Bigr\}
   \quad\text{and}\quad
   S_{\ge} \subseteq \Bigl\{\frac{a}{b} : a, b \ge 0,\; a \perp b,\; a, b \le \delta\Bigr\} \,.
\]
(These inclusions are actually equalities.)
Here we write $a \perp b$ to denote that $a$ and~$b$ are coprime.
Since $a + b \le 2 \delta < d$ in both cases, this implies
that $I$ can meet $S_{\le} \cup S_{\ge}$ at most in its endpoints.
We have to rule out the possibility that $\zeta_- \in S_{\le}$ and $\zeta_+ \in S_{\ge}$.
But then we would have that
\[ d < a_- + a_+ + b_- + b_+ \le \delta + 2\delta = d \,, \]
a contradiction. So Lemma~\ref{LemmaI} is always applicable.

Now we consider the case that $d$ is not divisible by~$3$. Then
\[ g a_i = 2d - 3 (i_1 + i_2) \equiv -d \bmod 3 \quad\text{and}\quad
   g b_i = d - 3 i_2 \equiv d \bmod 3 \,.
\]
We deduce that
\[ S_{\le} = S_{\le}^\sm \cup S_{\le}^\lr \qquad\text{and}\qquad
   S_{\ge} = S_{\ge}^\sm \cup S_{\ge}^\lr
\]
with
\begin{align*}
  S_{\le}^\lr
    &\subseteq \Bigl\{ \frac{a}{b} : a, b \ge 0,\; a \perp b,\; a \equiv b \equiv -d \bmod 3,\;
                             a + b \le d\Bigr\}\,, \\
  S_{\le}^\sm
    &\subseteq \Bigl\{ \frac{a}{b} : a, b \ge 0,\; a \perp b,\; a \equiv b \equiv d \bmod 3,\;
                             a + b \le \frac{d}{2}\Bigr\}\,, \\
  S_{\ge}^\lr
    &\subseteq \Bigl\{ \frac{a}{b} : a, b \ge 0,\; a \perp b,\; a \equiv b \equiv d \bmod 3,\;
                             a, b \le d\Bigr\}\,, \\
  S_{\ge}^\sm
    &\subseteq \Bigl\{ \frac{a}{b} : a, b \ge 0,\; a \perp b,\; a \equiv b \equiv -d \bmod 3,\;
                             a, b \le \frac{d}{2}\Bigr\} \,. \\
\end{align*}
(These inclusions are also in fact equalities.)
We see that $I \cap S_{\le}$ consists of endpoints of~$I$. If $\zeta_- \notin S_{\le}$,
then we can apply Lemma~\ref{LemmaI}. So we assume now that $\zeta_- \in S_{\le}$,
and we want to show that $I \cap S_{\ge} \subseteq \{\zeta_-\}$. If this is not the case,
then there is $\zeta \in S_{\ge}$ with $\zeta_- < \zeta \le \zeta_+$. Writing $\zeta = a/b$
in lowest terms, we then have that $[a, b] = k_- [a_-, b_-] + k_+ [a_+, b_+]$
with $k_- \in \Z_{\ge 0}$ and $k_+ \in \Z_{\ge 1}$ coprime.
The congruence conditions mod~$3$ imply that the determinant
\[ \begin{vmatrix} a & a_- \\ b & b_- \end{vmatrix}
     = k_- \begin{vmatrix} a_- & a_- \\ b_- & b_- \end{vmatrix}
        + k_+ \begin{vmatrix} a_+ & a_- \\ b_+ & b_- \end{vmatrix}
     = k_+
\]
is divisible by~$3$. This implies that $k_- \ge 1$ and $k_+ \ge 3$, so
$a + b \ge a_- + a_+ + b_- + b_+ > d$; in particular, $\zeta \in S_{\ge}^\lr$.
If $k_- = 1$, then $[a, b] \equiv [a_-, b_-] \bmod 3$, so $\zeta_- \in S_{\le}^\sm$.
Then $a_+ + b_+ > d - (a_- + b_-) \ge d/2$, and it follows that
\[ a + b \ge (a_- + a_+ + b_- + b_+) + 2(a_+ + b_+) > 2d \,, \]
a contradiction. If $k_- \ge 2$, then
\[ a + b \ge 2(a_- + a_+ + b_- + b_+) > 2d\,, \]
a contradiction again. So in both cases, we find that $I \cap S_{\ge} \subseteq \{\zeta_-\}$,
and so we can again apply Lemma~\ref{LemmaI}. This finishes the proof.

\medskip

It is not hard to turn the proof given here into an algorithm that computes a
complete set of weight vectors for plane curves of any given degree~$d$. We can then
extract the minimal complete set of weight vectors from it by removing weight
vectors that are dominated by some other vector in the set. We have computed
minimal complete sets of weight vectors for all $d \le 150$. In Figure~\ref{Fig3}
we show the difference between the largest entry in one of the weight vectors
and~$d$. This difference is $\le 0$ by Theorem~\ref{Thmn=2}. Write $m(d)$ for
the largest entry. For $d \le 150$, we see that $m(d) = d-2$ when $d \equiv 3 \bmod 6$
and $d \ge 15$ and that $m(d) = d-5$ when $d \equiv 0 \bmod 6$ and $d \ge 18$.
This can be shown to be true in general by considering the possibilities for
$\zeta_- \in S_{\le}$ and $\zeta_+ \in S_{\ge}$ when $a_- + a_+ + b_- + b_+$
is close to~$d = 3 \delta$ in the proof above. It is helpful that when $d$ is divisible
by~$3$, the descriptions of $S_{\le}$ and~$S_{\ge}$ are rather simple and
Lemma~\ref{LemmaDom} actually characterizes dominance.

When $d$ is not divisible by~$3$, the values $m(d) - d$ do not seem to follow
a simple pattern. In any case, they appear to get more and more negative as $d$
increases.

\begin{figure}[htb]
  \begin{center}
   \Gr{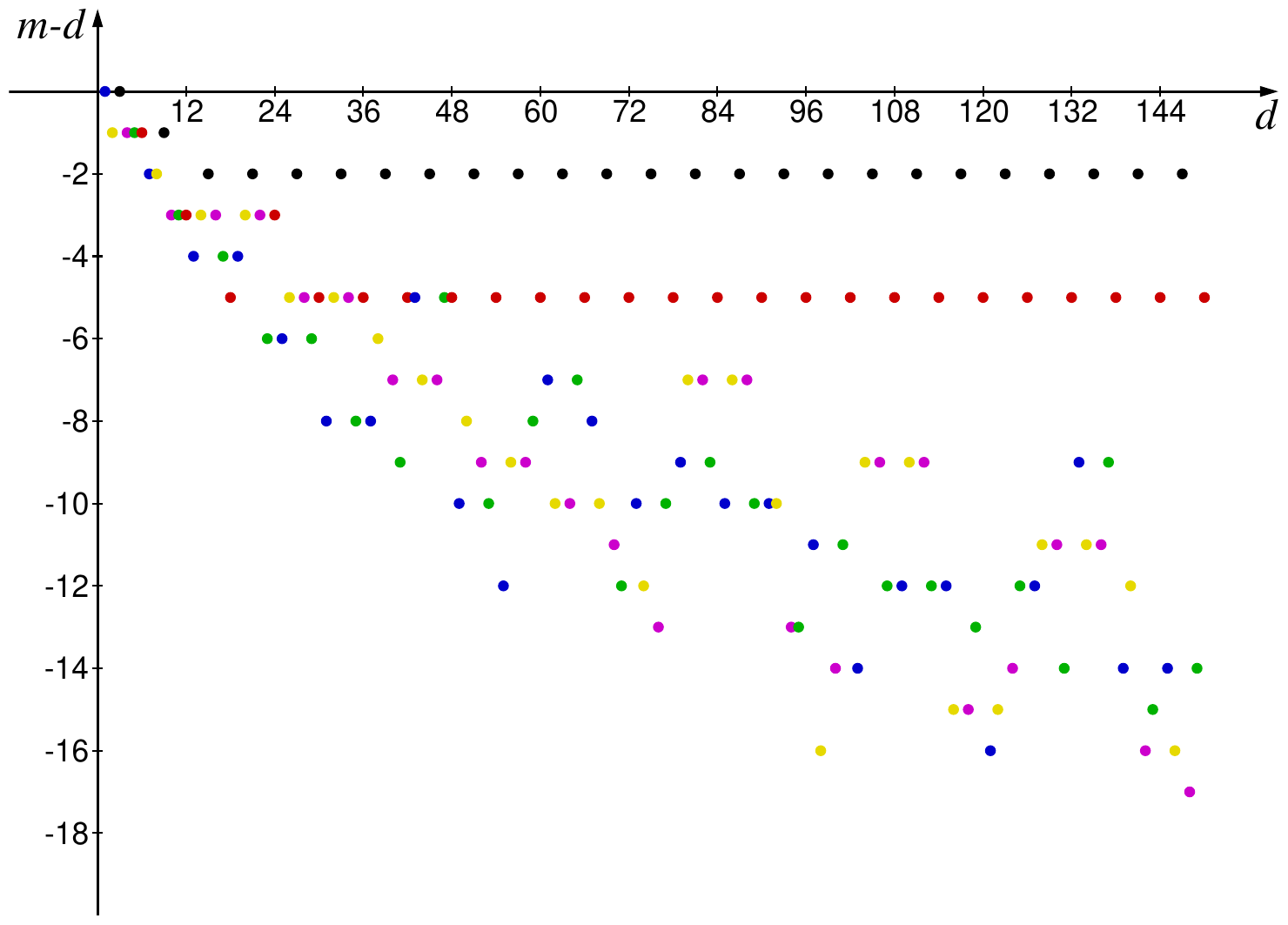}{0.8\textwidth}
  \end{center}

  \caption{Values of $m - d$ where $m$ is the largest entry occurring in some vector
           in the minimal complete
           set of weight vectors for plane curves of degree~$d$, for $1 \le d \le 150$.
           The data points are color-coded according to $d \bmod 6$.}
  \label{Fig3}
\end{figure}


\section{Proof of Theorem~\ref{ThmMain}} \label{S:proofgeneral}

We fix $n$ and $d$.
Our goal will be to show that every $w' \in W$ is dominated
by some vector in~$W$ whose largest entry is at most~$2 n d^{n-1}/\gcd(d,n+1)$;
this then implies the statement of Theorem~\ref{ThmMain}.

Recall that we use coordinates $z = [z_1, \ldots, z_n]$ to describe the (normalized) weight
vectors, where $w = [0, z_1, z_1 + z_2, \ldots, z_1 + \ldots, z_n]$.

We formalize the situation a bit more.

\begin{Definition} \label{D:ConeColl}
  A \emph{cone collection} in~$\R^n$ is a finite set~$\CC$ of closed cones in~$\R^n$
  (with vertex at the origin) such that
  \begin{enumerate}[(1)]\addtolength{\itemsep}{1mm}
    \item the intersection of any two cones in~$\CC$ is again in~$\CC$ and
    \item $\bigcup \CC \supset \R_{\ge 0}^n$.
  \end{enumerate}
  If $\CC$ is a cone collection and $w \in \R_{\ge 0}^n$, then there is a
  smallest cone in~$\CC$ containing~$w$ (by the first property above). We
  call it the \emph{minimal cone of~$w$} (w.r.t.~$\CC$) and write it $\CC(w)$.
\end{Definition}

For example, any finite set~$\CH$ of closed half-spaces in~$\R^n$ whose union
contains~$\R_{\ge 0}^n$ defines a cone collection~$\CC_{\CH}$. It consists of all
intersections of nonempty subsets of~$\CH$.
We then have $\CC_{\CH}(w) = \bigcap\,\{H \in \CH \colon w \in H\}$.

\begin{Definition}
  The cone collection defined by the set of closed half-spaces
  $\langle v_i, w \rangle \ge 0$ for $i \in J_{n,d}$
  is the \emph{collection of weight cones} (for $n$ and~$d$), $\CW_{n,d}$.
  Recall that the cones are defined for~$z \in \R_{\ge 0}^n$,
  where $w = [0, z_1, z_1+z_2, \ldots, z_1 + \ldots + z_n]$.
  In terms of~$z$, the half-spaces are given by $\langle v'_i, z \rangle \ge 0$
  with
  \[ v'_i = [dn - (n+1)(i_1 + \ldots + i_n), d(n-1) - (n+1)(i_2 + \ldots + i_n), \ldots, d - (n+1) i_n] \,. \]
\end{Definition}

Note that the second condition in Definition~\ref{D:ConeColl} is satisfied:
Given $z \in \R_{\ge 0}^n$, which corresponds to $w$ as above, we have
$\langle v'_{[d,0,\ldots,0]}, z \rangle = d (n z_1 + (n-1) z_2 + \ldots + z_n) \ge 0$,
so $z$ is contained in the half-space corresponding to $i = [d, 0, \ldots, 0]$.

\begin{Definition}
  Let $\CC$ be a cone collection in~$\R^n$. A subset $S \subseteq \Z_{\ge 0}^n$
  is \emph{complete for~$\CC$} if
  \[ \Z_{\ge 0}^n \subseteq \bigcup_{s \in S} \bigl(s + \CC(s)\bigr) \,. \]
\end{Definition}

Lemma~\ref{LemmaDom} then says the following.

\begin{Corollary} \label{CorCW}
  If a subset $S \subseteq \Z_{\ge 0}^n$ is complete for~$\CW_{n,d}$,
  then $S$ is a complete set of weights for dimension~$n$ and degree~$d$.
\end{Corollary}

We now prove a lemma that gives us a bound on the sizes of the vectors
in a minimal dominating set for the relative interior~$C^0$ of a cone~$C$
in terms of the sizes of the vectors spanning the cone.
We measure the `size' of a vector~$v$ in terms of the absolute value
of the sum~$\Sigma v$ of the entries.

\begin{Lemma} \label{LemmaHD}
  Let $C \subseteq \R^n$ be a polyhedral cone spanned by integral vectors
  $u_1, \dots, u_m$ such that $0 \le \Sigma u_j \le a$ for $j = 1, \ldots, m$.
  Assume that $C$ has dimension $k \le n$.
  Then for every $z \in \Z^n \cap C^0$, there is $z' \in \Z^n \cap C^0$
  such that $\Sigma z' \le k a$ and $z \in z' + C$.
\end{Lemma}

\begin{proof}
  It suffices to show that if $z \in \Z^n \cap C^0$ with $\Sigma z > ka$,
  then there is some~$j$ such that $z' = z - u_j \in C^0$ and $\Sigma u_j > 0$.
  (Then $z \in z' + C$ and
  $\Sigma z' < \Sigma z$; by induction we reach $\Sigma z' \le ka$.)
  Since $z \in C^0$, we can write $z = \sum_{j=1}^m \lambda_j u_j$ with
  all $\lambda_j > 0$. Pick $\eps > 0$ such that $\lambda_j > \eps$ for all~$j$
  and such that $\Sigma z > ka + \eps \sum_{j=1}^m \Sigma u_j$. The point
  $z^* = z - \eps \sum_{j=1}^m u_j$ is still in~$C$, hence it is in the
  closed cone spanned by some subset of $k$~vectors $u_j$; we can assume
  that they are $u_1, \dots, u_k$. We therefore have
  $z^* = \sum_{j=1}^k \mu_j u_j$ with $\mu_j \ge 0$.
  Now we observe that
  \[ ka < \Sigma z^* = \sum_{j=1}^k \mu_j \Sigma u_j \le \sum_{j=1}^k \mu_j a \]
  and conclude that one of the~$\mu_j$ such that $\Sigma u_j > 0$,
  say $\mu_{j_0}$, must be greater than~$1$.
  But then $z^* - u_{j_0}$ is still a nonnegative linear combination of the~$u_j$,
  and $z - u_{j_0}$ is a linear combination of $u_1, \dots, u_m$ with all
  coefficients positive, hence $z' = z - u_{j_0} \in C^0$.
\end{proof}

Note that the bound in the lemma is sharp, as can be seen by taking
$C = \R_{\ge 0}^n$, which is spanned by the standard basis vectors of size~$1$,
but for which we need to take $z' = \one_n$ of size~$n = \dim C$.

\medskip

The cone collection~$\CW_{n,d}$ can contain minimal cones of vectors such that
one cone is properly contained in the other, but they have the same dimension.
This makes $\CW_{n,d}$ somewhat unwieldy to work with. We remedy this by
`regularizing' our cone collection in some sense. We first introduce the following notion.

\begin{Definition}
  Let $\CC$ and~$\CC'$ be two cone collections. We say that $\CC$ \emph{refines}
  $\CC'$ if for every $w \in \R_{\ge 0}^n$ we have that $\CC(w) \subseteq \CC'(w)$.
\end{Definition}

\begin{Lemma} \label{LemmaRef}
  Let $\CC$ and~$\CC'$ be two cone collections such that $\CC$ refines~$\CC'$,
  and let $S$ be a complete set for~$\CC$. Then $S$ is also a complete set
  for~$\CC'$.
\end{Lemma}

\begin{proof}
  For every $s \in \Z^n_{\ge 0}$, we have that $\CC(s) \subseteq \CC'(s)$
  by assumption. Hence
  \[ \Z_{\ge 0}^n \subseteq \bigcup_{s \in S} \bigl(s + \CC(s)\bigr)
                  \subseteq \bigcup_{s \in S} \bigl(s + \CC'(s)\bigr) \,. \qedhere
  \]
\end{proof}

If $\CC$ is defined by a set of closed half-spaces, then any larger set
of closed half-spaces defines a refinement of~$\CC$. We refine~$\CW_{n,d}$
by including the `opposite' half-spaces.

\begin{Definition}
  We let $\widetilde{\CW}_{n,d}$ be the refinement of~$\CW_{n,d}$ that is
  generated by the set of closed half-spaces given by
  \[ \langle v_i, w \rangle \ge 0 \quad\text{or}\quad
     \langle v_i, w \rangle \le 0
  \]
  for all $i \in J_{n,d}$.
\end{Definition}

We now prove the following proposition, which by Lemma~\ref{LemmaRef}
and Corollary~\ref{CorCW} implies the statement of Theorem~\ref{ThmMain}
for general $n$ and~$d$.

\begin{Proposition} \label{PropMain}
  The set
  \[ S_{n,d} = \Bigl\{(z_1, \dots, z_n) \in \Z_{\ge 0}^n
                       : z_1 + \dots + z_n \le \frac{2 n d^{n-1}}{\gcd(d,n+1)}\Bigr\}
  \]
  is complete for~$\widetilde{\CW}_{n,d}$.
\end{Proposition}

\begin{proof}
  In terms of the coordinates $z_1, \dots, z_n$, we have for $i \in J_{n,d}$
  (using $\sum_j i_j = d$)
  \begin{align*}
    \bigl\langle d \one_{n+1} - (n+1) i, w \bigr\rangle
    &= \sum_{j=0}^n (d - (n+1) i_j) \sum_{k=1}^j z_k \\
    &= \sum_{k=1}^n \sum_{j=k}^n (d - (n+1) i_j) z_k \\
    &= \sum_{k=1}^n \Bigl((n+1) \sum_{j=0}^{k-1} i_j - k d\Bigr) z_k \,.
  \end{align*}

  Each $i \in J = J_{n,d}$ defines a hyperplane and a half-space in~$\R^n$.
  The rays that occur as intersections of $n-1$~independent such hyperplanes
  are spanned by integer vectors whose entries are obtained as $(n-1) \times (n-1)$
  minors of the $(n-1) \times n$ matrix whose rows are the coefficient vectors
  defining the hyperplanes. Let
  \[ I = (i^{(1)}, \dots, i^{(n-1)}) \in J_{n,d}^{n-1} \]
  be a linearly independent family and define $A_I$ to be the corresponding matrix. Then
  \[ A_I = (n+1) B_I - d \one_{n-1}^\top \cdot \kk \]
  with $\kk = [1, 2, \dots, n]$ and
  $B_I = \Bigl(\sum_{j=0}^{k-1} i^{(l)}_j\Bigr)_{1 \le l \le n-1, 1 \le k \le n}$.
  If for a matrix~$M$, $M^{[j]}$ denotes $M$ with the $j$th column removed,
  then a vector spanning the intersection of the relevant hyperplanes is
  \[ \tilde{v}_I = \bigl[(-1)^j \det A_I^{[j]}\bigr]_{1\le j\le n}
        = \bigl[(-1)^j \det \bigl((n+1) B_I^{[j]}
                                    - d \one_{n-1}^\top \cdot \kk^{[j]}\bigr)\bigr]_{1 \le j \le n}
          \,.
  \]
  Now
  \[ \det \bigl((n+1) B_I^{[j]} - d \one_{n-1}^\top \cdot \kk^{[j]}\bigr)
      = (n+1)^{n-1} \det B_I^{[j]}
          - (n+1)^{n-2} d \sum_{k \neq j} k \det \tilde{B}_I^{[j,k]} \,,
  \]
  where $\tilde{B}_I^{[j,k]}$ is the matrix~$B_I$ with the $k$th column
  replaced by all ones and the $j$th column removed. We see that all entries
  of~$\tilde{v}_I$ are divisible by~$(n+1)^{n-2} \gcd(d, n+1)$. We set
  \[ v_I = \frac{1}{(n+1)^{n-2} \gcd(d,n+1)} \tilde{v}_I \in \Z^n \,. \]

  We are interested in the maximal absolute value of the sum of the entries
  of~$v_I$. Recall the notation $\Sigma v$ for the sum of
  the entries of a vector~$v$.
  The sum $\Sigma v_I$ is affine linear as a function of each entry in~$I$ separately
  (it is the determinant of the matrix $A_I$ with the row~$\one_n$ added, up to the
  factor of~$(n+1)^{n-2} \gcd(d,n+1)$ that we have removed; the $l$th row of~$A_I$
  is an affine linear function of~$i^{(l)}$), therefore
  it takes its extremal values when the $i^{(l)}$ are extremal (and linearly
  independent) points in the simplex
  \[ \{i \in \R^{n+1} : i_0, \dots, i_n \ge 0, i_0 + \dots + i_n = d\} \,. \]
  This means that for such an extremal value, $B_I$ arises from the $(n+1) \times n$ matrix
  \[ B = d \begin{bmatrix} 1 & 1 & \cdots & 1 & 1 \\
                           0 & 1 & \cdots & 1 & 1 \\
                           \vdots & 0 & \ddots & \vdots & \vdots \\
                           \vdots & \vdots & \ddots & 1 & \vdots \\
                           0 & 0 & \cdots & 0 & 1 \\
                           0 & 0 & \cdots & 0 & 0
           \end{bmatrix}
  \]
  by removing two rows, so $A_I$ arises from $A = (n+1) B - d \one_{n+1}^\top \cdot \kk$
  in the same way. To get $\Sigma \tilde{v}_I$, we add the row $\one_n$
  at the top of~$A_I$ and take the determinant. This determinant is unchanged
  when we multiply the matrix from the right by the $n \times n$ matrix
  \[ C = \begin{bmatrix} 1 & -1 & 0 & \cdots & 0 \\
                         0 & 1 & -1 & \ddots & \vdots \\
                         \vdots & \ddots & \ddots & \ddots & 0 \\
                         0 & \cdots & 0 & 1 & -1 \\
                         0 & \cdots & \cdots & 0 & 1
         \end{bmatrix} \,.
  \]
  (This is essentially going back from the $z$ coordinates to the $w$ coordinates.)
  We obtain the determinant of the following matrix with two rows (but not the first one) removed.
  \[ \begin{bmatrix} \one_n \\ A \end{bmatrix} C
       = \begin{bmatrix} \one_n C \\ (n+1) B C - d \one_{n+1}^\top \cdot \kk C \end{bmatrix}
       = \begin{bmatrix}
            1   & \zero_{n-1} \\
            d n & -d \one_{n-1} \\
            -d \one_{n-1}^\top & d (n+1) E_{n-1} - d \one_{n-1}^\top \cdot \one_{n-1} \\
            -d  & -d \one_{n-1}
         \end{bmatrix}
  \]
  Here $E_m$ denotes the $m \times m$ identity matrix.
  We can remove the first column and first row without changing the
  value of the determinant. The remaining matrix has two equal rows,
  so to get a nonzero determinant, at least one of them has to be removed.
  So what we are looking at is $d^{n-1}$~times the $(n-1) \times (n-1)$ minors of
  the $n \times (n-1)$ matrix
  \[ D = \begin{bmatrix} (n+1) E_{n-1} - \one_{n-1}^\top \cdot \one_{n-1} \\
                                        -\one_{n-1}
         \end{bmatrix} \,.
  \]
  Such a minor is $\pm (n+1)^{n-2}$ when the last row is included, whereas
  the remaining minor has the value~$2 (n+1)^{n-2}$. We conclude that
  \[ |\Sigma v_I| = \frac{|\Sigma \tilde{v}_I|}{(n+1)^{n-2} \gcd(d,n+1)}
                  \le \frac{2 d^{n-1}}{\gcd(d,n+1)} \,.
  \]

  The cones of the cone collection~$\widetilde{\CW}_{n,d}$ are
  closed polyhedral cones~$C$ that are
  spanned by vectors~$\pm v_I$ for suitable tuples~$I$.

  Since the hyperplanes themselves and their intersections are elements
  of~$\widetilde{\CW}_{n,d}$, it follows that the cone $\widetilde{\CW}_{n,d}(w)$
  is the unique cone~$C$ in the collection that contains~$w$ in its relative
  interior~$C^0$.

  We apply Lemma~\ref{LemmaHD} to each of the cones in~$\widetilde{\CW}_{n,d}$,
  which as we have seen
  are spanned by vectors~$v$ with $\Sigma v \le 2 d^{n-1}/\gcd(d,n+1)$.
  The lemma shows that everything in the relative interior of each of these cones~$C$
  is dominated by some vector of size at most
  $2 (\dim C) d^{n-1}/\gcd(d,n+1) \le 2 n d^{n-1}/\gcd(d,n+1)$.
  Recall that $S_{n,d}$ denotes the subset of~$\Z_{\ge 0}$ of vectors~$w$ with
  $\Sigma w \le 2 n d^{n-1}/\gcd(d,n+1)$. We conclude that
  \[ \Z_{\ge 0}^n = \bigcup_{C \in \widetilde{\CW}_{n,d}} \bigl(\Z_{\ge 0}^n \cap C^0\bigr)
       \subseteq \bigcup_{C \in \widetilde{\CW}_{n,d}}
                 \bigcup_{s \in C^0 \cap S_{n,d}} (s + C)
       = \bigcup_{s \in S_{n,d}} \bigl(s + \widetilde{\CW}_{n,d}(s)\bigr) \,.
  \]
  This proves Proposition~\ref{PropMain}.
\end{proof}

\begin{Remark} \label{Rkthrow}
  In our proof, we throw away some information: the cones $\CW_{n,d}(w)$
  are in general larger than $\widetilde{\CW}_{n,d}(w)$. The difference
  is shown in Figure~\ref{Fig2} in the case $n = 2$, $d = 5$.
  On the left, the relevant
  half-planes are shown, with the resulting shifted cones covering the
  weight vectors dominated by the vertex. On the right, the fan resulting
  from the subdivision by all the rays is shown, together with the
  resulting shifted cones. We see that the maximal weight needed for
  a complete covering increases from~$4$ to~$7$. (In particular,
  it is larger than $d^{n-1} = 5$.)

  \begin{figure}[h]
    \Gr{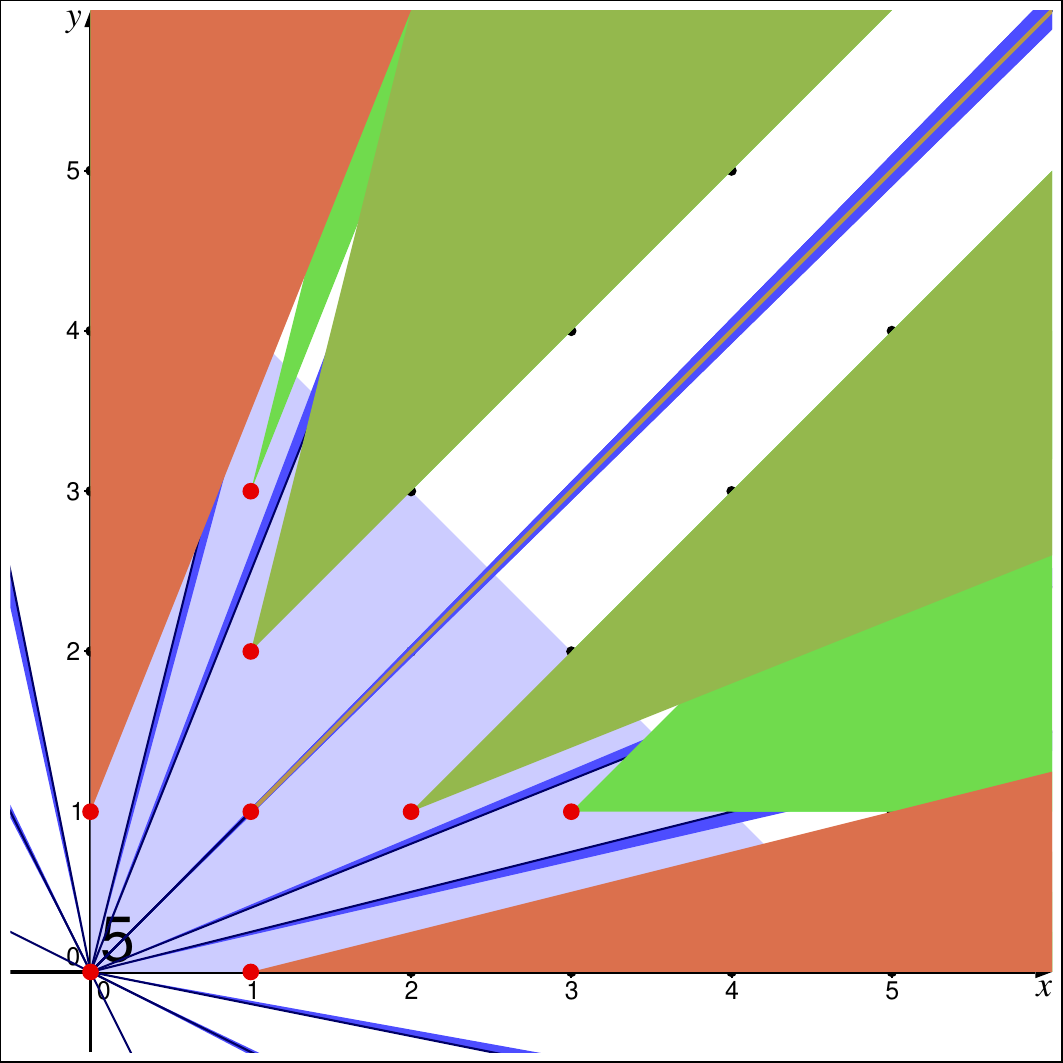}{0.45\textwidth}
    \hfill
    \Gr{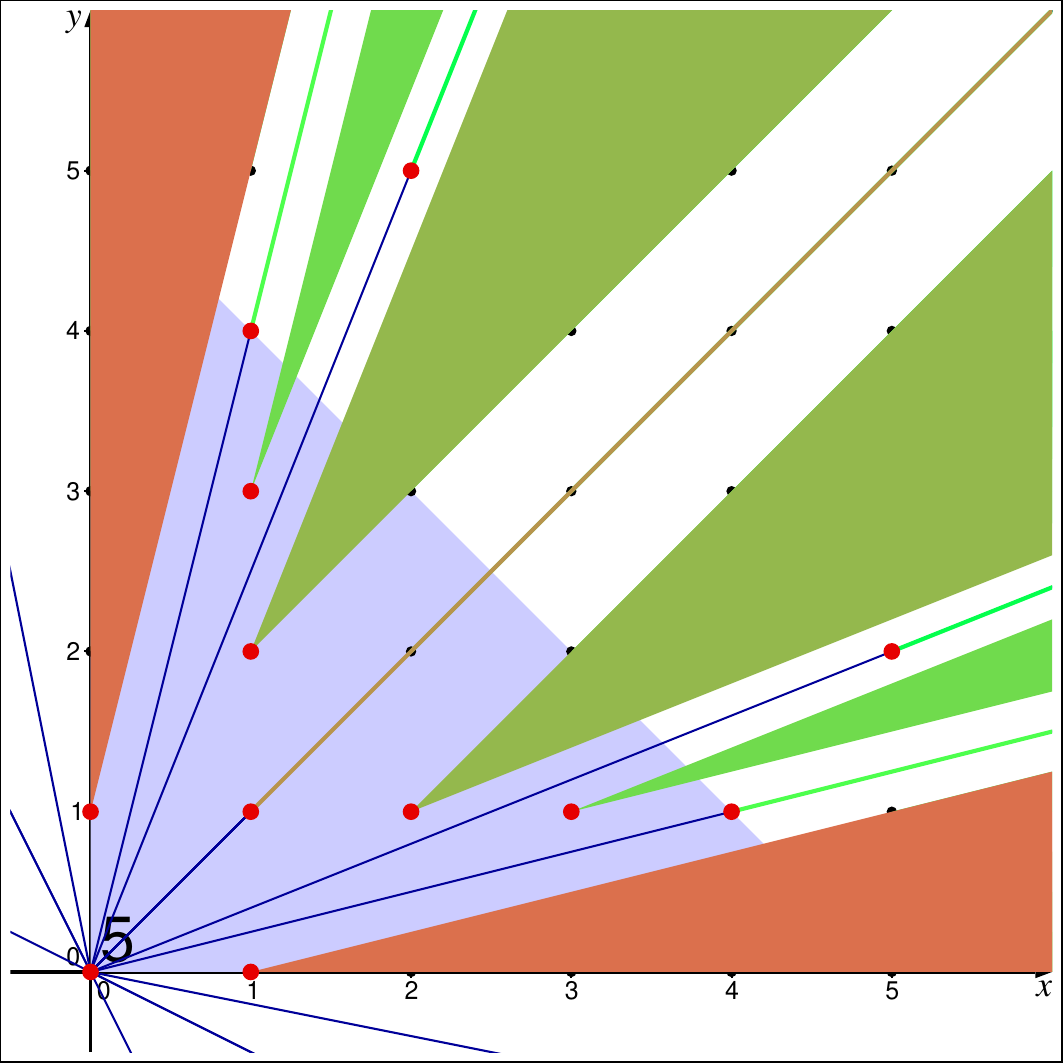}{0.45\textwidth}
    \caption{Illustration for Remark~\ref{Rkthrow}}
    \label{Fig2}
  \end{figure}

  This is likely related to the factor~$2$ that arises in the largest minor
  of the matrix~$D$. To get rid of that seems to necessitate working
  with the original cone collection~$\CW_{n,d}$ instead of the refinement.

  There is also the factor of~$n$ that comes from Lemma~\ref{LemmaHD}.
\end{Remark}

\begin{Remark} \label{Rkminimal}
  In any case, our result leads to an algorithm that determines the minimal
  complete set of weights for given dimension~$n$ and degree~$d$. We initialize
  $S$ to be the set of primitive weight vectors in~$S_{n,d}$.
  Then we successively take some $w \in S$ (in some order
  such that the last coordinate $w_n = z_1 + \dots + z_n$ weakly increases)
  and eliminate all vectors from~$S$ that are dominated by~$w$.
\end{Remark}


\section{Effective minimization} \label{S:algos}

As discussed in Remark~\ref{Rkminimal} above, we can determine the minimal
complete set of weight vectors relevant for a minimization algorithm for hypersurfaces
of degree~$d$ in~$\BP^n$. Table~\ref{Table1} gives some examples.
The list for plane cubics recovers~\cite{CFS2010}*{Lemma~4.4}.
We note that the minimal complete set of weights for cubic surfaces
is already mentioned (without proof) in~\cite{Kollar}*{Prop.~6.4.2}.

\begin{table}[h]
  \renewcommand{\arraystretch}{1.2}
  \begin{tabular}{c|c}
    case                 & minimal complete set of weight vectors \\ \hline
                         & \\[-13pt]
    conic                &   $[0,0,1], [0,1,1]$ \\
    plane cubic          &   $[0,0,1], [0,1,1], [0,1,2], [0,2,3]$ \\
    plane quartic        &   $[0,0,1], [0,1,1], [0,1,3]$ \\
    plane quintic        &   $[0,0,1], [0,1,1], [0,1,2], [0,1,3], [0,2,3], [0,3,4]$ \\
    quadric surface      &   $[0,0,0,1], \textcolor{gray}{[0,0,1,2]}, [0,1,1,1]$ \\
    cubic surface        &   $[0,0,0,1], [0,0,1,1], [0,1,1,1], [0,1,2,2], [0,2,2,3]$ \\
    quadric in~$\BP^4$   &   $[0,0,0,1,1], \textcolor{gray}{[0,0,1,1,2]}, [0,1,1,1,1]$
  \end{tabular}

  \vspace{10pt}

  \caption{Minimal complete sets of weight vectors for certain classes of hypersurfaces.
           Vectors that can be eliminated by changing coordinates are shown in light gray;
           compare Example~\ref{Ex:compl}.}
  \label{Table1}
\end{table}




For any given weight vector~$w$, it is a finite problem to determine whether a given
form~$F$ is unstable for~$(T, w)$ for a suitable unimodular matrix~$T$.
This is a consequence of the following result. Note that this is where in
the more general setting of a PID we have to assume that the residue
class field is finite.

\begin{Lemma} \label{L:finite}
  Let $w \in \Z_{\ge 0}^{n+1}$ be a weight vector; we write $M_w = \diag(p^{w_0}, \ldots, p^{w_n})$
  for the diagonal matrix with entries the powers of~$p$ given by~$w$. We set $G = \GL(n+1, \Z)$
  and define $G_w = G \cap M_w^{-1} G M_w$. Then $G_w$ is a finite-index
  subgroup of~$G$. Let $F \in \Z[x_0, \ldots, x_n]$ and $T \in G_w$; write $F' = {}^T F$.
  Then we have that
  \[ v_p(F'(p^{w_0} x_0, \ldots, p^{w_n} x_n)) = v_p(F(p^{w_0} x_0, \ldots, p^{w_n} x_n)) \,. \]
\end{Lemma}

\begin{proof}
  Let $m = \max\{w_0, \ldots, w_n\}$. Then any $T \in G$ such that $T \equiv E_{n+1} \bmod p^m$
  is in~$G_w$. Since the principal congruence subgroup mod~$p^m$ has finite index in~$G$,
  the same is true of~$G_w$.

  Now consider $F$ and~$T$ as in the statement above.
  Let $T' = M_w T M_w^{-1} \in G$. Note that $F(p^{w_0} x_0, \ldots, p^{w_n} x_n) = {}^{M_w} F$.
  We then have that
  \begin{align*}
    v_p(F'(p^{w_0} x_0, \ldots, p^{w_n} x_n))
      &= v_p({}^{M_w} F') = v_p({}^{M_w T} F) = v_p({}^{T' M_w} F) = v_p({}^{T'} ({}^{M_w} F)) \\
      &= v_p({}^{M_w} F)
       = v_p(F(p^{w_0} x_0, \ldots, p^{w_n} x_n))\,. \qedhere
  \end{align*}
\end{proof}

It follows that when $F$ is unstable for the weight system~$(T, w)$, then $F$
is also unstable for every weight system~$(T', w)$ with $T' \in G_w T$. To check
whether $F$ is unstable for~$w$, it is therefore sufficient to test one representative
of each coset of~$G_w$ in~$G$. Since $G_w$ has finite index in~$G$, this is a finite
problem. Assuming as usual that $w_0 \le w_1 \le \ldots \le w_n$, the condition for $T$
to be in~$G_w$ is that the reduction of~$T$ mod~$p$ is a block lower triangular
matrix, and each $(i,j)$~entry above the diagonal must be divisible by~$p^{w_j-w_i}$.
In particular, the coset is determined by the reduction of~$T$ modulo~$p^{w_n}$.

In our algorithms, we will make use of the following procedure.
The input consists of a form $F \in \Z[x_0,\ldots,x_n]$ of degree~$d$,
a unimodular matrix~$T$ of size~$n+1$, a weight vector $w \in W_n$ and the prime~$p$.

{\sf ApplyWeight}($F$, $T$, $w$, $p$) \\[1mm]
\sq $F_1 := {}^T F$; \\
\sq $F_2 := F_1(p^{w_0} x_0, p^{w_1} x_1, \ldots, p^{w_n} x_n)$; \\
\sq $e := v_p(F_2)$; \\
\sq \textbf{return} $p^{-e} F_2$, $e$.

We say that \emph{$w$ applies to~$F$}, if there is a~$T$ such that
$e \ge e(w) = \lfloor d \Sigma w/(n+1) \rfloor + 1$
in the above. This is shorthand for saying that $F$ is unstable with respect
to $(T, w)$ for some unimodular matrix~$T$.

\begin{Definition} \label{D:inv}
  An \emph{invariant} of forms of degree~$d$ in $n+1$~variables is a homogeneous
  polynomial~$I$ with integral coefficients in the coefficients of the form~$F$
  such that \hbox{$I({}^T F) = I(F)$} for all $T \in \SL(n+1)$.
\end{Definition}

\begin{Definition}
  A form $F \in k[x_0, \ldots, x_n]$ of degree~$d$ over a field~$k$ is
  \emph{semistable} if it is not a `nullform' in the sense of
  Hilbert~\cite{HilbertInvTh},
  i.e., there is an invariant~$I$ of forms of degree~$d$ in $n+1$~variables
  such that $I(F) \neq 0$. (This agrees with the notion of semistability
  in Geometric Invariant Theory~\cite{Mum77}*{Table~1}.)
  Otherwise, $F$ is \emph{unstable}.
\end{Definition}

\begin{Proposition} \label{P:effectivity}
  There is an algorithm that, given a semistable form $F \in \Z[x_0,\ldots,x_n]$
  of degree~$d$ and a prime~$p$, computes a matrix $T \in \GL(n+1, \Q) \cap \Mat(n+1, \Z)$
  and $e \in \Z_{\ge 0}$ such that $p^{-e} \cdot {}^T F$ has coefficients in~$\Z$
  and is minimal at~$p$.
\end{Proposition}

\begin{proof}
  By Theorems \ref{Thmn=2} (for $n = 2$) or~\ref{ThmMain} (in general), we
  can effectively find a complete set~$S$ of weight vectors for forms of degree~$d$
  in $n+1$~variables. For a given $w \in S$, we can determine a finite set~$T_w$ of coset
  representatives for~$G_w$ by Lemma~\ref{L:finite}.
  Let $P_{n,d} = \bigcup_{w \in S} T_w \times \{w\}$.
  The algorithm then is as follows.

  {\sf MinimizeForm}($F$, $p$) \\
  \sq $d := \deg(F)$; $n := $(number of variables in $F$) $-1$; \\
  \sq $T := E_{n+1}$; $e := 0$; // \emph{initialize transformation data} \\
  \sq success $:=$ {\sf true}; // \emph{flag indicating if a minimization step was successful} \\
  \sq \textbf{while} success \textbf{do} \\
  \sqq success $:=$ {\sf false}; // \emph{no success yet in this round} \\
  \sqq \textbf{for} $(T_1, w) \in P_{n,d}$ \textbf{do} \\
  \sqqq $F_1, e_1 := \text{\sf ApplyWeight}(F, T_1, w, p)$; \\
  \sqqq \textbf{if} $e_1 \ge e(w)$ \textbf{then} \\
  \sqqqq // \emph{minimization step successful} \\
  \sqqqq $F := F_1$; $T := T_1 T$; $e := e + e_1$; // \emph{update data} \\
  \sqqqq success $:=$ {\sf true}; \\
  \sqqq \textbf{end if}; \\
  \sqq \textbf{end for}; \\
  \sq \textbf{end while}; \\
  \sq \textbf{return} $F, T, e$;

  When $F$ is not unstable at~$p$ for any~$(T, w) \in P_{n,d}$,
  then $F$ is minimal at~$p$. So when the algorithm terminates, the return
  values satisfy the specification.
  Since $F$ is semistable, there is some invariant~$I(F)$
  of~$F$ that is nonzero. Since $I(F) \in \Z$ and $v_p(I(F_1)) < v_p(I(F))$
  when $F_1$ is obtained from~$F$ by a successful minimization step,
  the procedure must terminate after finitely many passes through the loop.
\end{proof}

In practice, running through all the cosets would be much too inefficient:
their number grows like a power of~$p$.
Therefore we look for necessary `geometric' conditions the form~$F$ has to
satisfy for a minimization step to be possible. We will see in the next
section that this can be done in the case $n = 2$ of plane curves.


\section{Minimization of plane curves at a prime $p$} \label{S:gencurves}

In this section we explain how one can construct an algorithm that minimizes
a (semistable) plane curve of any degree~$d$ at a prime~$p$. There are two main
ingredients.

The first ingredient is that we can split a minimization
step with respect to some weight vector~$w$ into a succession of steps with respect
to the simplest weight vectors $[0,0,1]$ and~$[0,1,1]$. During these intermediate
steps, the current form will not be `more minimal' than the original one,
but the last step will make it so (if the form can indeed be strictly minimized).
We are thus led to explore a tree of steps of this kind, until we either find a
more minimal form (then we restart the procedure with the new form), or else
can determine that progress is impossible; this is based on the bound from
Theorem~\ref{Thmn=2}.

The second ingredient consists in establishing a geometric criterion in terms
of the singular locus of the reduction of the curve mod~$p$ that reduces the
set of `directions' (corresponding to the cosets of~$G_w$ in Lemma~\ref{L:finite})
that we have to consider for each of the two simple weight vectors to an
easily computable set of size bounded in terms of the degree~$d$ only;
in particular, this bound does not depend on~$p$.

We note that the minimization algorithm for plane cubics given
in~\cite{CFS2010}*{Theorem~4.3} proceeds along similar lines.

The key result underlying this approach is as follows.

\begin{Proposition} \label{P:highmult}
  Let $F \in \Z[x_0,x_1,x_2]$ be a form of degree~$d$ that is unstable at~$p$ for the
  weight system~$(E, [0,w_1,w_2])$ with $0 \le w_1 \le w_2$ and $w_2 > 0$,
  so that $t := w_1/w_2 \in \interval{0}{1}$. We set
  \[ v_{011}(F) = v_p\bigl(F(x_0, px_1, px_2)\bigr) \quad\text{and}\quad
     v_{001}(F) = v_p\bigl(F(x_0, x_1, px_2)\bigr) \,.
  \]
  Then
  \[ v_{011}(F) > (1 + t) \frac{d}{3} \qquad\text{and}\qquad
     v_{001}(F) > (1 - 2t) \frac{d}{3} \,.
  \]
\end{Proposition}

\begin{proof}
  Write $F = \sum_{i+j+k = d} a_{i,j,k} x_0^i x_1^j x_2^k$. By assumption, we have
  \[ v_p\bigl(F(x_0, p^{w_1} x_1, p^{w_2} x_2)\bigr)
       \ge e = \Bigl\lfloor \frac{w_1 + w_2}{3} d \Bigr\rfloor + 1 > \frac{w_1 + w_2}{3} d \,,
  \]
  so $v_p(a_{i,j,k}) \ge \max(0, e - w_1 j -w_2 k)$. From this we get
  \begin{align*}
    v_{011}(F)
      &= \min \{i + v_p(a_{d-i,j,i-j}) : 0 \le j \le i \le d \} \\
      &\ge \min \{i + \max(0, e - w_1 j - w_2 (i-j)) : 0 \le j \le i \le d \} \\
      &= \min \{i + \max(0, e - w_2 i) : 0 \le i \le d \} \qquad \text{(as $w_1 \le w_2$)} \\
      &\ge \frac{e}{w_2} \qquad \text{(this is the minimum for $i \in \R$ such that $0 \le i \le d$)} \\
      &> \frac{w_1 + w_2}{w_2} \frac{d}{3}
       = (1 + t) \frac{d}{3} \,.
  \end{align*}
  This proves the first claim. The second claim is clear when $t > \frac{1}{2}$.
  Otherwise, we get in a similar way
  \begin{align*}
    v_{001}(F)
      &= \min \{ i + v_p(a_{d-j,j-i,i}) : 0 \le i \le j \le d \} \\
      &\ge \min \{ i + \max(0, e - w_1 (j-i) - w_2 i) : 0 \le i \le j \le d \} \\
      &= \min \{ i + \max(0, e - w_1 d - (w_2 - w_1) i) : 0 \le i \le d \} \\
      &\ge \frac{e - w_1 d}{w_2 - w_1}
       > \frac{w_2 - 2 w_1}{w_2 - w_1} \frac{d}{3}
       = \frac{1 - 2t}{1 - t} \frac{d}{3} \ge (1 - 2t) \frac{d}{3} \,. \qedhere
  \end{align*}
\end{proof}

\begin{Remark} \label{R:geomcrit}
  It is easily seen that $v_{011}(F)$ is a lower bound for the multiplicity
  of the point~$[1:0:0]$ on the reduction of the curve $F = 0$ and that $v_{001}(F)$
  is a lower bound for the multiplicity of the line $x_2 = 0$ as a component of
  the reduction of $F = 0$. So Proposition~\ref{P:highmult} implies a similar
  statement, where $v_{011}(F)$ is replaced by the multiplicity of~$[1:0:0]$ and
  $v_{001}(F)$ is replaced by the multiplicity of $x_2 = 0$ with respect to the
  reduction of the curve.
\end{Remark}

We can view changing the model of a plane curve as moving from one $\Z_p$-lattice
in~$\Q_p^3$ to another one, where the original lattice is generated by the standard
basis and we express the form~$F$ on a basis of the new lattice and then scale
by a power of~$p$ to normalize the resulting form. Any two lattices are commensurable,
and so we can define the distance $d(\Lambda, \Lambda')$ of two lattices by
\[ p^{d(\Lambda, \Lambda')} = (\Lambda : \Lambda \cap \Lambda') \cdot (\Lambda' : \Lambda \cap \Lambda') \,. \]
If $F$ is unstable at~$p$ for $(T, [0,w_1,w_2])$, then changing
$F$ to ${}^T F$ does not change the initial lattice (we just move to a different
basis), but the subsequent scaling of the variables according to the weight vector
enlarges the original lattice to one that contains it with index~$p^{w_1+w_2}$,
so the distance between the two is $w_1 + w_2$. (Note that $F \mapsto {}^M F$
corresponds to $\Lambda \mapsto \Lambda \cdot M^{-1}$.) Moving instead to an intermediate
lattice will possibly not yet minimize~$F$, but will bring us closer to a minimized
model. We can use Proposition~\ref{P:highmult} to tell us which way to go.

Before we formulate this more precisely, we make the following observations.
Recall that $M_w = \diag(p^{w_0}, \ldots, p^{w_n})$.

\begin{Lemma} \label{L:step}
  Let $\Lambda_0 = \Z_p^3$ and
  $\Lambda = \Lambda_0 \cdot M_{[0,w_1,w_2]}^{-1} = \langle [1,0,0], [0,p^{-w_1},0], [0,0,p^{-w_2}] \rangle$,
  where $0 \le w_1 \le w_2$ and $w_2 > 0$.
  \begin{enumerate}[\upshape(1)]
    \item \label{L:case1}
          If $T \in \GL(3, \Z_p)$ is a matrix such that $\bar{T}$ fixes the line $x_2 = 0$ in $\BP^2(\F_p)$ and
          $\Lambda' = \Lambda_0 \cdot (M_{[0,0,1]} T)^{-1}$,
          then $\Lambda_0 \subseteq \Lambda' \subseteq \Lambda$; in particular,
          $d(\Lambda', \Lambda) = w_1 + w_2 - 1$.
    \item \label{L:case2}
          Assume that $w_1 \ge 1$.
          If $T \in \GL(3, \Z_p)$ is a matrix such that $\bar{T}$ fixes the point
          \hbox{$[1:0:0] \in \BP^2(\F_p)$} and
          $\Lambda' = \Lambda_0 \cdot (M_{[0,1,1]} T)^{-1}$,
          then $\Lambda_0 \subseteq \Lambda' \subseteq \Lambda$; in particular,
          $d(\Lambda', \Lambda) = w_1 + w_2 - 2$.
    \item \label{L:case3}
          If $T \in p \Mat(3, \Z_p) \cap \GL(3, \Q_p)$ and $\Lambda' = \Lambda_0 \cdot T^{-1}$,
          then $\Lambda' \not\subseteq \Lambda$; in particular,
          $d(\Lambda_0, \Lambda') + d(\Lambda', \Lambda) > w_1 + w_2$.
  \end{enumerate}
\end{Lemma}

\begin{proof}
  The `in particular' statements follow from the fact that for three lattices
  $\Lambda_1$, $\Lambda_2$, $\Lambda_3$ with $\Lambda_1 \subseteq \Lambda_3$, we have
  \[ d(\Lambda_1, \Lambda_2) + d(\Lambda_2, \Lambda_3) = d(\Lambda_1, \Lambda_3)
       \iff \Lambda_1 \subseteq \Lambda_2 \subseteq \Lambda_3 \,.
  \]
  \begin{enumerate}[(1)]
    \item The condition is ${}^{\bar{T}} x_2 = \gamma x_2$ with $\gamma \in \F_p^\times$,
          so the third column of~$\bar{T}^{-1}$ is $[0, 0, \gamma^{-1}]^\top$,
          which implies that
          \[ (M_{[0,0,1]} T)^{-1} = T^{-1} \diag(1, 1, p^{-1})
               = \begin{pmatrix} t_{11} & t_{12} & t_{13} \\
                                 t_{21} & t_{22} & t_{23} \\
                                 t_{31} & t_{32} & p^{-1} t_{33}
                 \end{pmatrix}
          \]
          with $t_{ij} \in \Z_p$. The lattice~$\Lambda'$ is generated by the rows of
          this matrix and is visibly contained in~$\Lambda$.
    \item Here the condition is that the first row of~$\bar{T}^{-1}$ has the form $[\gamma,0,0]$
          with $\gamma \in \F_p^\times$. So
          \[ (M_{[0,1,1]} T)^{-1} = T^{-1} \diag(1, p^{-1}, p^{-1})
               = \begin{pmatrix} t_{11} & t_{12} & t_{13} \\
                                 t_{21} & p^{-1} t_{22} & p^{-1} t_{23} \\
                                 t_{31} & p^{-1} t_{32} & p^{-1} t_{33}
                 \end{pmatrix}
          \]
          with $t_{ij} \in \Z_p$, and we conclude as in the previous case.
    \item In this case, $\Lambda_0 \cdot T^{-1}$ contains $p^{-1} \Lambda_0$,
          but $[p^{-1},0,0] \notin \Lambda$.
    \qedhere
  \end{enumerate}
\end{proof}

\begin{Corollary} \label{C:ptorline}
  Assume that the form $F \in \Z[x_0, x_1, x_2]$ of degree~$d$ is unstable at~$p$ for a
  weight system $(T, [0,w_1,w_2])$ with $0 \le w_1 \le w_2$ and $w_2 > 0$.
  We denote by $\Lambda = \Z_p^3 \cdot (M_{[0,w_1,w_2]} T)^{-1}$
  the lattice associated with this weight system.
  As usual, we write $\bar{F}$ for the reduction of~$F$ mod~$p$ and $\bar{T}$
  for the reduction of~$T$ mod~$p$. Then one of the following is true.
  \begin{enumerate}[\upshape(1)]
    \item \label{C:case1}
          $\bar{F} = L^m \cdot G$ with a linear form~$L$ defined over~$\F_p$
          and $m > \frac{d}{3}$, with the
          property that if $T' \in \GL(3, \Z)$ is such that ${}^{\bar{T}'} L = \lambda x_2$,
          then the lattice associated to $(T', [0,0,1])$ has distance $w_1+w_2-1$ from~$\Lambda$.
    \item \label{C:case2}
          $\bar{F} = L^m \cdot G$ with a linear form~$L$ defined over~$\F_p$
          and $0 < m \le \frac{d}{3}$
          such that $L \nmid G$ and the line $L = 0$ intersects $G = 0$ in a point~$P$ defined over~$\F_p$
          of multiplicity $> \frac{d-3m}{2}$ on $G = 0$, with the property that
          if $T' \in \GL(3, \Z)$ is such that ${}^{\bar{T}'} L = \lambda x_2$,
          then the lattice associated to $(T', [0,0,1])$ has distance $w_1+w_2-1$ from~$\Lambda$.
    \item \label{C:case3}
          The curve $\bar{F} = 0$ has a point~$P$ defined over~$\F_p$
          of multiplicity $> \frac{d}{2}$
          that does not lie on a line defined over~$\F_p$ that is
          contained in the curve, with the property that
          if $T' \in \GL(3, \Z)$ is such that $[1:0:0] \cdot \bar{T}' = P$, then
          the lattice associated to $(T', [0,1,1])$ has distance $w_1+w_2-2$ from~$\Lambda$.
          Such a point~$P$ is unique.
  \end{enumerate}
\end{Corollary}

\begin{figure}[htb]
  \Gr{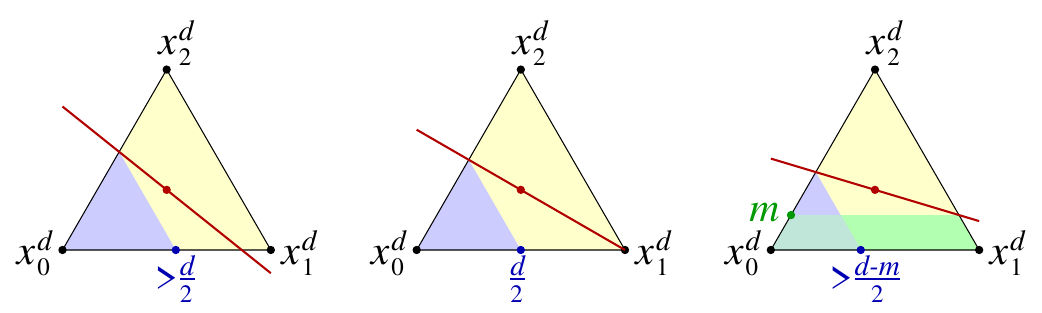}{\textwidth}
  \caption{Illustration of Corollary~\ref{C:ptorline}. The slope of the red line
  depends on $w_1/w_2$; up to symmetry, it can have one of the indicated positions.
  The coefficients of~$\bar{F}$ corresponding to the area on and below the line
  vanish (compare~\cite{Mum77}*{p.~46}).
  The blue triangle of vanishing coefficients corresponds to a point of
  high multiplicity, and the green trapezoid in the right-hand figure corresponds
  to a multiple line.}
  \label{Fig_triang}
\end{figure}

This allows us to find candidates for~$\bar{T}$ by determining the possible points~$P$
or lines~$L$ from~$\bar{F}$. Note that the number of these objects is bounded
in terms of~$d$ only. The worst case is when $F$ is a product of distinct linear
factors with the corresponding lines passing through a common point; then we have
to consider these $d$ linear factors.

\begin{proof}
  We write~$X$ for the curve defined by $\bar{F} = 0$ and $X'$ for the curve
  defined by ${}^{\bar{T}} \bar{F} = 0$ (then $X = X' \cdot \bar{T}$).
  The assumption on~$F$ implies that ${}^T F$ satisfies the assumption
  of Proposition~\ref{P:highmult}. Let $m_P$ denote the multiplicity of the
  point~$[1:0:0]$ on~$X'$ and let $m_L$ denote the multiplicity
  of~$x_2$ as a factor of~${}^{\bar{T}} \bar{F}$.
  Then by Remark~\ref{R:geomcrit}, $m_P \ge v_{011}({}^T F)$
  and $m_L \ge v_{001}({}^T F)$.

  Let $t = w_1/w_2$ as before. According to Proposition~\ref{P:highmult},
  if $t \ge \frac{1}{2}$, then $m_P \ge v_{011}({}^T F) > \frac{d}{2}$, which implies that
  there is a point $P = [1:0:0] \cdot \bar{T}$ of multiplicity $> \frac{d}{2}$
  on~$X$. If $P$ is not on a line contained in this curve, then we
  have case~\eqref{C:case3}. Since the line joining two distinct points of
  multiplicity $> \frac{d}{2}$ on~$X$ must be contained in~$X$, there can be
  at most one such point. The claim regarding the lattice then follows from
  Lemma~\ref{L:step}~\eqref{L:case2}. If $P$ is on a line contained in~$X$, then we are
  in cases \eqref{C:case2} or~\eqref{C:case1}, where the line is $L = 0$,
  and the claim on the lattice follows from Lemma~\ref{L:step}~\eqref{L:case1}.

  If $t < \frac{1}{2}$, then $m_L \ge v_{001}({}^T F) > 0$, so ${}^{\bar{T}} \bar{F}$
  splits off a factor~$x_2^m$ for some $m \ge 1$. If $m \le \frac{d}{3}$, then
  $m > (1 - 2t) \frac{d}{3}$ implies $t > \frac{1}{2} - \frac{3m}{2d}$, hence
  $m_P > (1 + t) \frac{d}{3} > \frac{d-m}{2}$. So $[1:0:0]$ must have
  multiplicity $> \frac{d-m}{2} - m = \frac{d-3m}{2}$ on the remaining part
  of~$X'$. Applying~$T$, we see that $X$ contains a line of multiplicity~$m$
  that intersects the remaining part of~$X$ in a point~$P$ that has multiplicity
  $> \frac{d-3m}{2}$ on this remaining part, so we are in case~\eqref{C:case2}.
  If, finally, $m > \frac{d}{3}$, then we are in case~\eqref{C:case1}.
  In both cases, the claim on the lattice follows from Lemma~\ref{L:step}~\eqref{L:case1}.
\end{proof}

In each case, if we apply $(T', [0,0,1])$ or $(T', [0,1,1])$ to~$F$ and then
normalize the resulting form, we either obtain a form~$F'$ with smaller valuation
of the invariants (in which case we have successfully performed a minimization
step), or else $F'$ can be minimized using some $(T'', [0,w'_1,w'_2])$ such
that $w'_1 + w'_2 = w_1 + w_2 - 1$ (in cases \eqref{C:case1} or~\eqref{C:case2})
or  $w'_1 + w'_2 = w_1 + w_2 - 2$ (in case \eqref{C:case3}). We can use
Lemma~\ref{L:step}~\eqref{L:case3} to detect when we deviate from the path
(at least in some cases). Since we know that $w_1 + w_2 \le 2d-1$ (when $d \ge 2$)
by Theorem~\ref{Thmn=2} and the fact that we can take $w_1$ and~$w_2$ coprime,
we have a bound on the number of steps that are maximally necessary to achieve
minimization when minimization is possible.

This results in the following algorithm.

\begin{Algorithm}
  The input of {\sf MinimizePlaneCurveOneStep} and {\sf MinimizePlaneCurve}
  consists in a semistable ternary form $F \in \Z[x_0, x_1, x_2]$ of degree~$d \ge 2$
  and a prime number~$p$. The result of {\sf MinimizePlaneCurveOneStep}
  consists of a boolean flag indicating
  whether a minimization step could be performed successfully and in this case,
  a form~$G$ of degree~$d$, a matrix~$T$ and a number $e \in \Z_{\ge 0}$
  such that $G = p^{-e} \cdot {}^T F$ is the result of the minimization step; otherwise
  $F$, $E_3$ and~$0$ are returned as the last three values.
  The result of {\sf MinimizePlaneCurve} consists of a form~$G$ of degree~$d$
  that is a minimized representative of the orbit of~$F$, together with a
  matrix~$T$ and a number $e \in \Z_{\ge 0}$ as above.

  We define $\delta(d)$ to be the maximum of $w_1+w_2$
  over the minimal complete set of weight vectors $[0,w_1,w_2]$ for plane curves of degree~$d$
  (or an upper bound for this quantity).
  This can be precomputed for the relevant values of~$d$ using the procedure hinted
  at near the end of Section~\ref{S:n=2}; alternatively, we can set
  $\delta(d) := 2d-1$; compare~Theorem~\ref{Thmn=2}.

  {\sf MinimizePlaneCurveOneStep}($F$, $p$) \\
  \sq $d := \deg(F)$; \\
  \sq \textbf{function} {\sf Recurse}($F$, $r$, $\gamma$, $T_0$) \\
  \sqq  // \emph{$r \in \Z$: bound for the distance to the goal lattice,} \\
  \sqq  // \emph{$\gamma \in \Z$: change of valuation so far,} \\
  \sqq  // \emph{$T_0 \in \Mat(3, \Z)$: transformation matrix so far} \\
  \sqq \textbf{if} $\gamma < 0$ \textbf{then}
         \textbf{return} {\sf true}, $F$, $T_0$, $0$;
       \textbf{end if}; // \emph{success!} \\
  \sqq \textbf{if} $T_0 \bmod p = 0$ \textbf{or}  $r \le 0$ \textbf{then} \\
  \sqqq \textbf{return} {\sf false}, $F$, $T_0$, $0$;
        // \emph{veering off the path or maximal distance reached} \\
  \sqq \textbf{end if}; \\
  \sqq $\bar{F} = F \bmod p \in \F_p[x_0, x_1, x_2]$; \\
  \sqq write $\bar{F} = L_1^{m_1} \cdots L_s^{m_s} G$ \\
  \sqq \quad with pairwise non-proportional linear forms $L_j$, $m_j \ge 1$, \\
  \sqq \quad and $G$ not divisible by a linear form; \\
  \sqq \textbf{for} $j := 1$ \textbf{to} $s$ \textbf{do} \\
  \sqqq $T := $ a matrix in $\GL(3, \Z)$ such that ${}^{\bar{T}} L_j = \lambda x_2$; \\
  \sqqq \textbf{if} $m_j \le d/3$ \textbf{then} // \emph{see Corollary~\ref{C:ptorline}~\eqref{C:case2}} \\
  \sqqqq $H(X,Y) := ({}^{\bar{T}} \bar{F}/x_2^{m_j})(x_0, x_1, 0) \in \F_p[x_0, x_1]$; \\
  \sqqqq \textbf{if} $H$ has no linear factors of multiplicity $> \frac{d-3m}{2}$ \textbf{then}
           go to the next $j$;
        \textbf{end if}; \\
  \sqqq \textbf{end if}; // \emph{else we use Corollary~\ref{C:ptorline}~\eqref{C:case1}}\\
  \sqqq $F_1, e := $ {\sf ApplyWeight}($F$, $T$, $[0,0,1]$, $p$); \\
  \sqqq success, $F_2$, $T_1$, $e_1 := $ {\sf Recurse}($F_1$, $r-1$, $\gamma + d - 3e$, $M_{[0,0,1]} T T_0$); \\
  \sqqq \textbf{if} success \textbf{then}
          \textbf{return} {\sf true}, $F_2$, $T_1$, $e+e_1$;
       \textbf{end if}; \\
  \sqq \textbf{end for}; \\
  \sqq \textbf{if} there is a point $P$ of multiplicity $> d/2$ on $G = 0$
        with $\forall j \colon L_j(P) \neq 0$ \textbf{then} \\
  \sqqq // \emph{Corollary~\ref{C:ptorline}~\eqref{C:case3}} \\
  \sqqq $T := $ a matrix in $\GL(3, \Z)$ such that $[1:0:0] \cdot \bar{T} = P$; \\
  \sqqq $F_1, e := $ {\sf ApplyWeight}($F$, $T$, $[0,1,1]$, $p$); \\
  \sqqq success, $F_2$, $T_1$, $e_1 := $ {\sf Recurse}($F_1$, $r-2$, $\gamma + 2d - 3e$, $M_{[0,1,1]} T T_0$); \\
  \sqqq \textbf{if} success \textbf{then}
          \textbf{return} {\sf true}, $F_2$, $T_1$, $e+e_1$;
        \textbf{end if}; \\
  \sqq \textbf{end if}; \\
  \sqq \textbf{return} {\sf false}, $F$, $T_0$, $0$; \\
  \sq \textbf{end function}; \\
  \sq \textbf{return} {\sf Recurse}($F$, $\delta(d)$, $0$, $E_3$);

  The quantity~$\gamma$ is used to keep track of the increase and decrease
  in the $p$-adic valuation of the invariants caused by scaling the variables
  and the form by powers of~$p$. If $\gamma$ is negative, then the condition
  for instability with respect to the weight vector accumulated so far is satisfied.

  {\sf MinimizePlaneCurve}($F$, $p$) \\
  \sq $T := E_3$; $e := v_p(F)$; $G := p^{-e} F$; // \emph{initialize; do $w = [0,0,0]$} \\
  \sq success, $G$, $T_1$, $e_1 := $ {\sf MinimizePlaneCurveOneStep}($G$, $p$); \\
  \sq \textbf{while} success \textbf{do} \\
  \sqq $T := T_1 T$; $e := e + e_1$; // \emph{update transformation data} \\
  \sqq success, $G$, $T_1$, $e_1 := $ {\sf MinimizePlaneCurveOneStep}($G$, $p$); \\
  \sq \textbf{end while}; \\
  \sq \textbf{return} $G$, $T$, $e$;
\end{Algorithm}

As written, the algorithm performs a depth-first search in the tree of
lattices that are constructed depending on the lines and points found
on the reduction. Alternatively, one can implement a breadth-first version
or also a best-first version that expands the node with the smallest
value of~$\gamma$. Experiments seem to indicate that the tree rarely
branches heavily, so that we expect there to be no penalty in practice
for using the simpler depth-first code.

An implementation of this algorithm is available in Magma~\cite{Magma}
under the name \texttt{MinimizeTernaryFormAtp}.

We note that the algorithm can be adapted to an arbitrary PID~$R$
with a prime element~$\pi$ in place of $\Z$ and~$p$,
as long as we can do computations in~$R$ and the residue class field
$k = R/\langle \pi \rangle$, and the map $R \to k$ is computable
and allows the determination of a preimage for a given element of~$k$.
Modulo computations in~$k$ and in~$R$, its complexity depends only on~$d$.


\section{Global minimization and reduction of plane curves} \label{S:gencurves-glob}

When we have a plane curve~$X$ over~$\Q$ defined by a ternary
form~$F$, for which we would
like to find a nice model, we first need to determine a finite set of
primes~$p$ such that the given model might be non-minimal at~$p$,
so that we can then apply the procedure derived in Section~\ref{S:gencurves}
for these finitely many primes~$p$.

If the curve is smooth (and the degree~$d$ satisfies $d \ge 2$), then
a necessary condition is that the reduction of~$X$ mod~$p$ is singular.
So we could compute the discriminant of the given model and find its
prime divisors, or alternatively, set up a system of equations that a
singular point has to satisfy and do a Gr\"obner basis computation over~$\Z$
to obtain a nonzero integer~$N$ such that all relevant primes must divide~$N$.
The disadvantage of this approach is that usually there are quite a few
large primes~$p$ such that $X$ is singular, but still semistable, mod~$p$
(in the sense that there is an invariant~$I$ such that $p \nmid I(F)$),
and so we have to factor a large number, even though we are interested
only in certain of its prime factors. So instead, we should try to cut
the set of primes down as closely as possible to the set of primes such
that the reduction of~$X$ mod~$p$ is unstable (i.e., $\bar{F}$ is a nullform).
For this, we can use the
necessary conditions coming from the `geometric' version of
Proposition~\ref{P:highmult} as mentioned in Remark~\ref{R:geomcrit}.
Write $\bar{X}$ for the reduction of~$X$ mod~$p$.
Then for $\bar{X}$ to be unstable, $\bar{X}$ either has to contain a
line~$L$ of multiplicity~$m$ such that $m > d/3$ or else
there is a point of multiplicity (on~$\bar{X}$) $> (d-m)/2$ on~$L$,
or $\bar{X}$ has a point of multiplicity $> d/2$ (which is the case
$m = 0$ of the previous condition).
For each $m = 0, 1, \ldots, \lfloor d/3 \rfloor + 1$,
we can write down equations (depending on the location of the line
and/or the point relative to the standard affine patches) that must
be satisfied; a Gr\"obner basis computation over~$\Z$ then results
in a basis of the corresponding ideal that contains a unique nonzero
integer~$N$ (here we assume that the curve $X$ over~$\Q$ does not generically
satisfy one of these conditions; otherwise $X$ would be very close
to being unstable), whose prime divisors give us candidates for the
primes at which we might be able to minimize~$X$. Unless $d$ is
very small, the conditions we impose cut out subvarieties of codimension
at least~$2$ of the moduli space of plane curves of degree~$d$, and so
we can expect `spurious' large primes to occur only in rare cases.

The Gr\"obner basis computations can still take some time, though. They will be
more efficient if we can add the information that the relevant primes
have to divide some given nonzero integer~$N$. (This has the effect
of computing over~$\Z/N\Z$ and thus avoids intermediate coefficient growth.)
A necessary condition for
the curve to be non-minimal at~$p$ is that the form defining it becomes
unstable when reduced mod~$p$. This means that all its invariants
are divisible by~$p$. So we can get a suitable integer by computing some
invariants and taking their~gcd. Recall that an \emph{invariant} of
ternary forms of degree~$d$ is a homogeneous polynomial~$I(F)$ with
integral coefficients in the coefficients of the form~$F$ such that
$I({}^T F) = I(F)$ for all $T \in \SL(3)$. A \emph{covariant}
is a map associating to a form~$F$ of degree~$d$ another form~$C(F)$
of some degree whose coefficients are homogeneous polynomials with integral
coefficients in the coefficients of~$F$ and such that $C({}^T F) = {}^T C(F)$
for all $T \in \SL(3)$. Covariants of covariants are again
covariants, and invariants of covariants are invariants. One possibility
of generating covariants is to use the $k$th \emph{\"Uberschiebung}.
We define the differential operator
\[ \Delta = \det\begin{pmatrix}
                   \ppx{x_0} & \ppx{x_1} & \ppx{x_2} \\[4pt]
                   \ppx{y_0} & \ppx{y_1} & \ppx{y_2} \\[4pt]
                   \ppx{z_0} & \ppx{z_1} & \ppx{z_2}
                 \end{pmatrix} \,.
\]
Then the $k$th \"Uberschiebung (or transvectant) of three ternary
forms $F$, $G$, $H$ is
\[ \Ue^k(F, G, H) = \Delta^k F(x_0,x_1,x_2) G(y_0,y_1,y_2) H(z_0,z_1,z_2) \Big|_{y_j, z_j \leftarrow x_j} \,. \]
One can show that when $F, G, H$ are covariants of a form, then
$\Ue^k(F, G, H)$ is again a covariant. (This comes down to the fact that
the determinant of a matrix does not change when the matrix is multiplied
by a matrix in~$\SL(3)$. The analogous statement for binary forms is
classical. See Salmon~\cite{Salmon_Algebra}*{Lesson~XIV}, who gives credit to Cayley.)
For example, $\Ue^2(F, F, F)$
is the Hessian of~$F$, and $\Ue^1(F, G, H)$ is the Wronskian determinant
of $F$, $G$ and~$H$. In particular, we obtain an invariant
when the \"Uberschiebung is constant. It is easy to see that
$\Ue^k(F, G, H) = 0$ when $k$ is odd and two of $F$, $G$, $H$ are the same.
When $d = \deg F$ is even, then $I_1(F) = \Ue^d(F, F, F)$ is an invariant that is
generically nonzero, and $G = \Ue^{d-2}(F, F, F)$ is a sextic covariant of~$F$
such that $I_2(F) = \Ue^6(G, G, G)$ is another invariant that is generically
nonzero and independent of~$I_1(F)$. We can then use $\gcd(I_1(F), I_2(F))$
in the approach described above. When $d$ is odd, then $G = \Ue^{d-1}(F, F, F)$
is a cubic covariant of~$F$, and we can use the invariants of~$G$ instead.

After we have determined a finite set of candidate primes~$p$, we can
successively minimize our curve at these~$p$ using {\sf MinimizePlaneCurve}.
We then have a globally minimal plane model $F(x_0, x_1, x_2) = 0$ of our curve.
This minimal model can still have quite large coefficients.
To remedy this, we want to find a transformation $T \in \SL(3, \Z)$
so that ${}^T F$ has reasonably small coefficients. (Note that applying~$T$
does not change the invariants of~$F$, hence preserves minimality.)
This process is called \emph{reduction}. As explained in~\cite{Stoll2011b},
one possible approach is to associate to the curve~$X$ a zero-dimensional subscheme
(or point cluster)~$C$ of~$\BP^2$ and then reduce~$C$ using the algorithm
described in loc.~cit. A suitable choice is the scheme of inflection
points, which is given as the intersection of $F = 0$ with $H = 0$,
where $H$ is the Hessian of~$F$ (i.e., the determinant of the matrix of
second partial derivatives of~$F$, up to a constant factor).
This has the disadvantage that the degree of this scheme grows quadratically
with $d = \deg F$. Instead we can use any scheme obtained from the
intersection of the curves defined by two covariants
of~$F$, as long as it is stable in the sense of~\cite{Stoll2011b}.
When $d$ is odd, we can also take the
cubic covariant~$G$ from above and reduce it (if it is stable), which
is equivalent to using the scheme of inflection points of the curve
given by $G = 0$.

In practice, it seems to be most efficient to do an `ad-hoc' reduction
first (or only). For this, we apply a certain set of `small' elements
of~$\SL(3, \Z)$ to our form~$F$ and check if the size of~$F$ (measured,
for example, as the euclidean length of the coefficient vector) gets smaller
for one of them, say~$T$. If so, we replace~$F$ by~${}^T F$ and continue;
otherwise, we stop.
Combining our general minimization algorithm with this reduction procedure
finally results in an algorithm that produces a `maximally nice'
model of the curve, in the sense that it is globally minimal and its
defining equation has small coefficients.

We have implemented this procedure in Magma~\cite{Magma}. Global minimization
is performed by \texttt{MinimizeTernaryForm}, reduction by
\texttt{ReduceTernaryForm}, and both together by \texttt{MinRedTernaryForm}.

The following examples
give some indication of the performance of the implementation.

\begin{Example}
  The following sextic form occurs in~\cite{DNS}.
  \begin{align*}
    F &= 5 x^6 - 50 x^5 y + 206 x^4 y^2 - 408 x^3 y^3 + 321 x^2 y^4 + 10 x y^5
         - 100 y^6 + 9 x^4 z^2 \\
      &\qquad {} - 60 x^3 y z^2 + 80 x^2 y^2 z^2
            + 48 x y^3 z^2 + 15 y^4 z^2 + 3 x^2 z^4 - 10 x y z^4 + 6 y^2 z^4 - z^6 \,.
  \end{align*}
  The plane curve defined by it has four simple double points (hence geometric
  genus~$6$); it is a model of a certain modular curve.

  We compute the two invariants $I_1(F)$ and~$I_2(F)$ and find that
  \[ N = \gcd(I_1(F), I_2(F)) = 867041280 \,.  \]
  (The prime divisors of~$N$
  are $2$, $3$, $5$ and~$7$, but we don't need to know this.) Then we do the
  Gr\"obner basis computations with $N$ added to the generators of the ideals.
  This shows that $F$ can be non-minimal at most at $p = 2$ and $p = 7$.
  This part of the procedure took about a quarter second.

  The minimization algorithm with $p = 2$ traverses a tree with $13$~nodes
  and finds a successful minimization step on the way. The minimization
  algorithm with $p = 7$ traverses a tree with three nodes before it concludes
  that no proper minimization is possible. This part of the procedure took
  less than a tenth of a second.

  Finally, we apply ad-hoc reduction (in fact, this is done first and also
  in between the local minimization steps to keep the coefficients of
  reasonable size) and the cluster reduction, which does not actually improve
  the final result, to obtain the polynomial below. This part of the procedure
  took about a third of a second. The total time was about~$0.7$~seconds.

  \begin{align*}
    F_0 &= -x^6 - 2 x^5 y + 2 x^5 z + 23 x^4 y z - 5 x^3 y^3 - x^3 y^2 z
                + x^3 y z^2 + 5 x^3 z^3 - x^2 y^4 - 8 x^2 y^3 z \\
        &\qquad{} + 17 x^2 y^2 z^2 - 8 x^2 y z^3 - x^2 z^4 + 3 x y^5 - 7 x y^4 z
                  + 10 x y^3 z^2 - 10 x y^2 z^3 + 7 x y z^4 \\
        &\qquad{} - 3 x z^5 + y^6 - 3 y^5 z + 3 y^4 z^2 - 6 y^3 z^3 + 3 y^2 z^4 - 3 y z^5 + z^6 \\
        &= \frac{1}{16} {}^T F(x_0, x_1, x_2)
  \end{align*}
  with
  \[ T = \begin{pmatrix} 1 & 1 & 0 \\ -1 & 0 & 1 \\ 1 & 0 & 1 \end{pmatrix} \,. \]
  We note that $F$ is even as a polynomial in~$z$, showing that the curve
  has an involution. This feature is lost after minimization.
\end{Example}

\begin{Example}
  We start with a form of degree~$10$ with small random coefficients,
  \begin{align*}
    F &= 7 x^{10} + 4 x^9 y - 9 x^9 z - x^8 y^2 + 9 x^8 y z - 5 x^8 z^2 - 4 x^7 y^3
          - 8 x^7 y^2 z - 7 x^7 y z^2 - 9 x^7 z^3 \\
      &\qquad{} - 3 x^6 y^4 - 5 x^6 y^3 z + 2 x^6 y^2 z^2 - 7 x^6 y z^3 + 4 x^6 z^4
                + 8 x^5 y^5 + 10 x^5 y^4 z + 5 x^5 y^3 z^2 \\
      &\qquad{} - 3 x^5 y^2 z^3 + 2 x^5 y z^4 - x^4 y^6 + 9 x^4 y^5 z
                - 3 x^4 y^4 z^2 + 5 x^4 y^3 z^3 + x^4 y z^5 - 2 x^4 z^6 \\
      &\qquad{} + 6 x^3 y^7 + 8 x^3 y^6 z + 9 x^3 y^4 z^3 + 9 x^3 y^3 z^4 + 5 x^3 y^2 z^5
                - 5 x^3 y z^6 + 3 x^3 z^7 - 10 x^2 y^8 \\
      &\qquad + 8 x^2 y^6 z^2 - 5 x^2 y^5 z^3 + 8 x^2 y^4 z^4 - 10 x^2 y^3 z^5 - 5 x^2 y^2 z^6
              - x^2 z^8 - 3 x y^9 + 8 x y^8 z \\
      &\qquad{} - 10 x y^7 z^2 + 7 x y^6 z^3 + 4 x y^5 z^4 - 9 x y^4 z^5 + x y^3 z^6
                - 4 x y^2 z^7 - 9 x y z^8 - 2 x z^9 - 9 y^{10} \\
      &\qquad{} - 7 y^9 z + 5 y^8 z^2 - 7 y^7 z^3 + 2 y^6 z^4 - 2 y^5 z^5 + 3 y^4 z^6
                - 2 y^3 z^7 + 2 y^2 z^8 + 8 y z^9 + 5 z^{10} \,.
  \end{align*}
  We set up a random integral $3 \times 3$ matrix with ten-digit entries,
  \[ T = \begin{pmatrix}
                  -6822460139 & -8617905122 & 4801170083 \\
                   5588128275 &  3128463726 & 3491404315 \\
                  -3274111511 &   371050596 & 2931443838
         \end{pmatrix} \,;
  \]
  then $F_1 = {}^T F$ is an integral form of degree~$10$ with coefficients of about
  a hundred digits. Running our implementation on~$F_1$ recovers the original form~$F$
  up to interchanging $x$ and~$z$. The time for this is less than four minutes, most
  of which is spent in determining the (potentially, but in this instance really)
  unstable primes $2$, $5573747$ and~$2748254186176163904623$.
\end{Example}


\section{Determination of all minimal models} \label{S:all-minimal}

Recall that a form~$F$ of degree~$d$ in $n+1$~variables is \emph{(properly) stable}
if its orbit under~$\SL(n+1)$ is closed, is \emph{semistable} if the
closure of its orbit
does not contain the zero form, and is \emph{unstable} otherwise~\cite{Mum77}.
Koll\'ar gives equivalent definitions using weight systems; see~\cite{Kollar}*{Def.~2.3}.

Minimal models of plane curves (and more generally, of projective hypersurfaces)
need not be unique modulo the action of~$\GL(n+1, \Z)$. This comes from the
corresponding statement on minimality at a prime~$p$. The following example
shows that a semistable form can have infinitely many pairwise $\Z_p$-inequivalent
models.

\begin{Example} \label{Ex:min-inf}
  Let $p$ be a prime and consider
  \[ F = x y z + y^3 + z^3 \in \Z_p[x,y,z] \,. \]
  Then $F$ is minimal (its invariants $c_4$ and~$c_6$ are $1$ and~$-1$;
  see~\cite{Fisher2006}*{Sect.~1} for a definition)
  and is $\Q_p$-equivalent with
  \[ F_{i,j} = p^{-i-j} F(x, p^i y, p^j z) = x y z + p^{2i-j} y^3 + p^{2j-i} z^3 \in \Z_p[x,y,z] \]
  for all $i, j \in \Z_{\ge 0}$ with $i \le 2j$ and $j \le 2i$, and $F_{i,j}$
  has the same invariants as~$F$.

  Among the $F_{i,i}$ for $i \ge 0$, there are infinitely many pairwise
  $\Z_p$-inequivalent ones. Otherwise, there would be $\lambda_i \in \Z_p^\times$
  and $M_i \in \GL(3, \Z_p)$ such that $\lambda_i {}^{M_i} F = F_{i,i}$
  for infinitely many~$i$. Since $\Z_p^\times \times \GL(3, \Z_p)$ is
  compact, there would be a convergent sub-sequence $(\lambda_{i_k}, M_{i_k})_{k \ge 0}$
  with $i_k \to \infty$ as $k \to \infty$; let $(\lambda, M)$ be its limit. Then
  \[ \lambda \, {}^M F = \lim_{k \to \infty} \lambda_{i_k} {}^{M_{i_k}} F
                       = \lim_{k \to \infty} F_{i_k,i_k}
                       = x y z \,,
  \]
  so $F$ would be equivalent to~$x y z$, which is clearly absurd.
\end{Example}

When $F$ is stable, there are finitely many pairwise $\Z_p$-inequivalent
models of~$F$ \cite{Kollar}*{Theorems 4.1.2 and~5.2.3},
but their number is not uniformly bounded.

\begin{Example} \label{Ex:multi1}
  Fix $k \in \Z_{>0}$ and a prime~$p$ and consider
  \[ F = x y z + p^k x^3 + y^3 + z^3 \in \Z_p[x,y,z] \,. \]
  This form defines a smooth cubic over~$\Q_p$ (its discriminant
  is $-p^k (27 p^k + 1)^3 \neq 0$). It is minimal, since it is congruent
  mod~$p$ to the form~$F$ in Example~\ref{Ex:min-inf}.

  For all pairs $(i,j) \in \Z_{\ge 0}^2$ with $i \le 2j$, $j \le 2i$
  and $i+j \le k$, $F$ is $\Q_p$-equivalent to the form
  \[ F_{i,j} = p^{-i-j} F(x, p^i y, p^j z)
             = x y z + p^{k-i-j} x^3 + p^{2i-j} y^3 + p^{2j-i} z^3 \in \Z_p[x,y,z]
  \]
  with the same invariants. The forms $F_{i,j}$ and~$F_{i',j'}$ are $\Z_p$-equivalent
  if and only if the multisets $\{2i-j, 2j-i, k-i-j\}$ and $\{2i'-j', 2j'-i', k-i'-j'\}$
  agree, so the number of pairwise $\Z_p$-inequivalent $p$-minimal models
  becomes arbitrarily large as $k \to \infty$.

  The `only if' part follows from the fact that the points
  $[ 1 : 0 : 0 ], [ 0 : 1 : 0 ], [ 0 : 0 : 1 ] \in \BP^2(\Z / p \Z)$ can be lifted
  to $\BP^2(\Z_p)$-points
  with $\|\nabla F_{i,j}\|_p = |3 p^{k-i-j}|_p$, $|3 p^{2i-j}|_p$, and $|3 p^{2j-i}|_p$,
  respectively (where $\nabla F$ denotes the gradient of a form~$F$),
  but not to points with smaller $p$-adic norm of $\nabla F_{i,j}$.
  Furthermore, these are the only potentially singular points of the reduction modulo~$p$.%

  Viewing $F(x,y,z) = 0$ as a $p$-adic elliptic curve, one can compute
  its invariants as
  \[ c_4 = -216 p^k + 1\,, \quad c_6 = 5832 p^{2k} - 540 p^k - 1\,,
     \quad \Delta = -p^k (27 p^k + 1)^3 \,. \]
  This shows once more that the model is $p$-minimal. It has split multiplicative reduction.
  Tate's algorithm as described in~\cite[Chap.~4.9]{Silverman_ATAEC} results in reduction type~$I_k$.
  Further, $k$ is the thickness of the singular point~$[1:0:0]$ of the reduction as defined
  in~\cite[Chap. 10, Def. 3.23, Ex. 3.24]{Liu_AGAC}.
  Finally, the exponents $k-i-j$, $2i-j$ and~$2j-i$ are the thicknesses of the singular points
  $[ 1 : 0 : 0 ], [ 0 : 1 : 0 ], [ 0 : 0 : 1 ] \in \BP^2(\Z / p \Z)$
  of the reduction of $F_{i,j} = 0$.
\end{Example}

This raises the question how one can determine a system of representatives
of the $\Z_p$-isomorphism classes of $p$-minimal models of a given form~$F$.
We can certainly assume that $F$ is $p$-minimal itself (otherwise we apply
the minimization algorithm to it first). Replacing the strict inequality
in Definition~\ref{D:unstable} and in~\eqref{InEqFund} by a non-strict one,
we obtain a similar theory of dominance of weight vectors (we have to
exclude the vectors $[k,k,\ldots,k]$, though, which would otherwise dominate
everything; they correspond to applying a unimodular transformation, which
gives an equivalent form), so that we can determine a set of weight vectors~$w$ such that
if there is another $p$-minimal model of~$F$ that is not equivalent to~$F$ over~$\Z_p$,
then one such model can be obtained via an application of~$w$. Note that
the minimal complete sets of weight vectors we obtain can be different
from those we use for minimization. For example, the set for conics
is now $\{[0,0,1], [0,1,2]\}$ instead of $\{[0,0,1], [0,1,1]\}$,
and for plane cubics, we can use $\{[0,0,1], [0,1,1]\}$.

This leads
to an algorithm that decides if another $p$-minimal model (or a
``more minimal'' one) exists, and if so,
produces one. One then has to repeat the procedure with each new model that
was found (taking care of keeping only one representative of each equivalence
class over~$\Z_p$) until no new models are found. Note that the `distance'
between two $p$-minimal models in the sense of Section~\ref{S:gencurves}
can be arbitrarily large as shown by Example~\ref{Ex:multi1}, so we cannot
hope to find all of them in one go by applying a finite set of weight vectors.

To construct a list of representatives of all $\Z$-equivalence classes
of (globally) minimal models of~$F$, one combines the various $p$-minimal
models for all~$p$. The problem is then to produce a finite list of
primes~$p$ such that $F$ can have several inequivalent $p$-minimal models.
Note that the reduction of~$F$ mod~$p$ does not need to be unstable;
the reduction can be semistable but not stable
(this is illustrated by Example~\ref{Ex:multi1} above),
so one has to use slightly weaker geometric conditions.
It is not clear (to us, at least) how to determine invariants that vanish
on all semistable forms that are not stable; if we had two or more independent
such invariants, we could use a method like that described in
Section~\ref{S:gencurves-glob}. The discriminant is one such invariant
(at least when the degree~$d$ is not very small), so one can use it,
at least when dealing with smooth hypersurfaces, to avoid intermediate expression
growth in the Gr\"obner basis computations, but this will be significantly
less efficient than using the gcd of two suitable invariants.

We leave the task of devising a reasonably efficient algorithm
that finds representatives of all equivalence classes of locally
or globally minimal models to future work.


\section{Minimization in higher dimensions} \label{S:higher}

Our approach to minimization of plane curves is based on the following
four observations.
\begin{enumerate}[(1)]
  \item \label{I:flags}
        If the ternary form~$F$ is unstable at~$p$ for some weight system~$(T, w)$,
        then the curve given by $\bar{F} = 0$ contains a flag of linear subspaces
        with certain multiplicities. (Concretely, we have a line containing a point,
        with multiplicities~$m$ for the line and~$\max\{m, \floor{\frac{d-m}{2}} + 1\}$
        for the point, with $0 \le m \le d$; see Corollary~\ref{C:ptorline}.)
  \item \label{I:bound for flags}
        The number of such flags that can be contained in the reduced
        curve is bounded in terms of the degree~$d$ only.
  \item \label{I:direction}
        In each case, we can use one of the linear subspaces contained
        in the flag with positive multiplicity to move closer to the
        form obtained by applying~$(T, w)$ to~$F$ (in the sense of lattice
        distance; see again Corollary~\ref{C:ptorline}).
  \item \label{I:bound for steps}
        The number of simple steps from one lattice to another is bounded
        in terms of~$d$ by Theorem~\ref{Thmn=2}.
\end{enumerate}

Part of this carries over to the case $n \ge 3$. Part~\eqref{I:bound for steps}
is taken care of by Theorem~\ref{ThmMain}. Part~\eqref{I:flags}
generalizes as follows.

\begin{Proposition} \label{P:flags in general}
  Let $F \in \Z[x_0, \ldots, x_n]$ be a form of degree~$d$ that is unstable
  at~$p$ for the weight system $(E, w)$ with $w = [w_0,w_1,\ldots,w_n] \in W$
  (so $0 = w_0 \le w_1 \le \ldots \le w_n$). We assume that $v_p(F) = 0$; then $w_n > 0$.
  Write $L_k$ for the $k$-dimensional linear subspace of~$\BP^n_{\F_p}$ given
  by $x_{k+1} = \ldots = x_n = 0$. Then the hypersurface defined by $\bar{F} = 0$
  contains~$L_k$ with multiplicity at least
  \[ m_k = \begin{cases}
             \hfill 0 & \text{if $(n+1) w_k > \Sigma w$,} \\
             \displaystyle\floor{\frac{d}{n+1}\,\frac{\Sigma w - (n+1) w_k}{w_n - w_k}}  + 1
                      & \text{otherwise.}
           \end{cases}
  \]
\end{Proposition}

\begin{proof}
  Write $F = \sum_{i \in I} a_i x^i$ as usual. We have to show that $v_p(a_i) > 0$
  if the degree of~$x^i$ in $x_{k+1}, \ldots, x_n$ is less than~$m_k$. By our assumption
  on~$F$, we know that
  \[ v_p(a_i) \ge \max\left\{0, 1 + \floor{\frac{d}{n+1} \Sigma w} - \langle i, w \rangle\right\} \,. \]
  Thus, $v_p(a_i) > 0$ whenever $\frac{d}{n+1} \Sigma w \ge \langle i, w \rangle$.
  Let $m$ denote the degree of~$x^i$ in $x_{k+1}, \ldots, x_n$. Then the
  weak monotonicity of the entries of~$w$ implies that
  \[ \langle i, w \rangle \le (d - m) w_k + m w_n = d w_k + m (w_n - w_k) \,. \]
  We can assume that $(n+1) w_k \le \Sigma w$; otherwise, there is nothing
  to show. Then
  \begin{align*}
    \frac{d}{n+1} \Sigma w - \langle i, w \rangle
      &\ge \frac{d}{n+1} \Sigma w - d w_k - m (w_n - w_k) \\
      &\ge \frac{d}{n+1} \Sigma w - d w_k - (m_k - 1) (w_n - w_k) \ge 0
  \end{align*}
   as desired (the last inequality follows from the definition of~$m_k$).
\end{proof}

From the proof, it is clear that the bound in Proposition~\ref{P:flags in general}
is sharp.

\begin{Remark}
  In a similar way as in the proof of Proposition~\ref{P:highmult},
  one can show the stronger statement that
  \[ v_p\bigl(F(x_0, \ldots, x_k, p x_{k+1}, \ldots, p x_n)\bigr) \ge m_k \,. \]
\end{Remark}

Lemma~\ref{L:step} extends in an obvious way to a general
version of Part~\eqref{I:direction} above.

The obstacle in establishing an efficient general minimization procedure
for surfaces in~$\BP^3$ (say) is in Part~\eqref{I:bound for flags} above:
it is in general no longer true that the number of flags with multiplicities
that we have to consider can be bounded in terms of $d$ alone (for fixed
dimension~$n$). For example, $w = [0,1,2,2]$ is an element of the minimal complete
set of weight vectors for cubic surfaces (see Table~\ref{Table1}).
The multiplicities given by Proposition~\ref{P:flags in general} for~$w$
are
\[ (m_2, m_1, m_0) = (0, 1, 2) \,, \]
so we can conclude that there is a line on $\bar{F} = 0$ that passes through
a singular point, but no flag with higher multiplicities needs to occur.
Now consider the case that $\bar{F} = 0$ is a cone over a nodal cubic curve.
This surface contains a one-parameter family of lines passing through the
vertex of the cone and in addition a double line, which gives a one-parameter
family of lines with a singular point on them (by fixing the line and varying
the point). So we would have to run through on the order of~$p$ lines or
points and try the corresponding directions for minimization. What saves us
in the cubic case is that when $[0,1,2,2]$ applies to~$F$ and the reduction
of~$F$ defines a cone, then $[0,1,1,1]$ also applies to~$F$, and here the
direction is determined by the point of multiplicity~$3$.

Similarly, for $w = [0,2,2,3]$, we find
\[ (m_2, m_1, m_0) = (0, 0, 2) \,, \]
so the only geometric condition we obtain is that there is a singular point.
Again, if the reduction is a cone over a cubic curve, then $[0,1,1,1]$
applies as well. In addition, we can use for both $[0,1,2,2]$ and~$[0,2,2,3]$
that the singular point is `very singular' in the sense that the value
of~$F$ at any lift of it is divisible by~$p^2$; the number of such points
is uniformly bounded when neither $[0,0,0,1]$ nor~$[0,0,1,1]$ apply.
See Section~\ref{S:cubsurf} below for details.

For quartic surfaces, we have a similar situation. For the weight vectors
$[0,2,3,6]$ and~$[0,3,5,9]$ (which both are in the minimal complete set),
the multiplicity bounds are $(m_2,m_1,m_0) = (0,1,2)$ as above, and there
are configurations for the reduced surface that contain one-parameter families
of lines and singular points. For example, this is the case when the reduced
surface is a union of two quadrics of rank $3$ or~$4$ (and at least one
of the two is of rank~$3$ or split).

Similar difficulties arise with cubic threefolds.
For example, the weight vector $[0,2,2,2,3]$ is part of a minimal system of weights.
We get \[ (m_3, m_2, m_1, m_0) = (0, 0, 0, 2) \,, \] thus the reduction of the
threefold has a singular point. A refined analysis leads to a point in $\BP^4(\Q)$
such that a primitive integral representative satisfies the equation modulo $p^6$ and the gradient
vanishes modulo $p^3$.
As the singular locus of the chordal cubic $xzv - xu^2 - y^2v + 2yzu - z^3 \in \Q[x,y,z,u,v]$
is given by a rational normal curve of degree 4 with parametrization
$t \mapsto [1 : t : t^2 : t^3 : t^4]$, it is not clear how to treat threefolds that
are $p$-adically close to it in an efficient way.

It is certainly possible that in cases
like these, another weight vector applies or more stringent conditions
can be obtained that depend on $F$ mod~$p^2$ (like in the cubic case)
or involve higher derivatives, so that one can work around these potential difficulties.
However, we will not attempt to follow this line in the present paper. Instead, we will focus
on the case of cubic surfaces; we present a suitable algorithm in the
next section.


\section{Minimization of cubic surfaces} \label{S:cubsurf}

Unstable cubic surfaces were already studied by Hilbert. Their classification is as follows:
\begin{Remark}
  A cubic surface is unstable if and only if it satisfies one of the following.
  \begin{enumerate}[(1)]
    \item It has a singular point such that the tangent cone degenerates
          to a plane of multiplicity 2.
    \item It has a singular point such that the tangent cone degenerates
          to two planes and the intersection of the two planes is a line
          contained in the surface.
    \item It has a triple point. I.e., the surface degenerates to a cone.
  \end{enumerate}
  This list is given in~\cite{MF}*{Chap.~4.2, page~80}.
  The first two options are already listed in~\cite{Hilbert}*{page~367}.
  Note that reducible cubic surfaces and cubic surfaces with a singular line are covered by the above.

  In more modern language, a normal cubic surface that is not a cone is unstable if and only if
  it has a singular point of type $A_3$, $A_4$, $A_5$, $D_4$, $D_5$ or~$E_6$.
  This follows from a comparison of the list~\cite{CAG}*{Table~9.1} with the above result.
\end{Remark}

We now describe the ingredients for an algorithm that minimizes semistable
(e.g., smooth) cubic surfaces.
Note that $w = [0, \ldots, 0]$ applies to~$F$ if and only if $v_p(F) \ge 1$,
so all coefficients are divisible by~$p$. We will always scale our equations
to have coprime coefficients, so we do not have to consider this case,
or rather, we normalize the equation right at the beginning and keep it
so during the procedure. So we do not have to consider the zero weight
vector further. The simplest remaining weight vector is~$[0,0,0,1]$.

\begin{Lemma} \label{L:0001}
  Let $F \in \Z[x_0,x_1,x_2,x_3]$ be primitive and homogeneous of degree~$3$.
  Then $[0,0,0,1]$ applies to~$F$ if and only if $\bar{F}$ splits off
  a linear factor defined over~$\F_p$.
\end{Lemma}

\begin{proof}
  First assume that $[0,0,0,1]$ applies to~$F$, so there is a unimodular matrix~$T$
  such that
  \begin{equation} \label{E:cond0001}
    v_p\bigl({}^T F(x_0, x_1, x_2, px_3)\bigr) \ge 1 + \floor{\frac{3}{4} \cdot 1} = 1 \,.
  \end{equation}
  Reducing mod~$p$, we see that ${}^{\bar{T}} \bar{F}(x_0, x_1, x_2, 0) = 0$,
  which means that $x_3$ divides~${}^{\bar{T}} \bar{F}$. This implies that
  $\bar{F}$ splits off a linear factor defined over~$\F_p$ as well.

  Conversely, assume that $\bar{F}$ splits off a linear factor defined over~$\F_p$.
  After applying a suitable unimodular matrix~$T$, we can assume that $x_3$ divides~$\bar{F}$.
  Then ${}^{\bar{T}} \bar{F}(x_0, x_1, x_2, 0) = 0$, which implies~\eqref{E:cond0001}.
\end{proof}

This has the following consequence.

\begin{Corollary}
  Each semistable (for example, smooth) cubic surface over~$\Q$ has an integral model
  such that the reductions modulo all primes are irreducible.
\end{Corollary}

\begin{proof}
  Since the given surface is semistable, there is some invariant~$I$ that does not
  vanish on equations for the surface.
  Let $F(x_0,x_1,x_2,x_3) \in \Z[x_0,x_1,x_2,x_3]$ be a defining polynomial of the surface such that
  the absolute value of~$I(F)$ is minimal among all integral defining equations.
  If the reduction of~$F$ mod~$p$ were reducible
  for some prime~$p$, then Lemma~\ref{L:0001} would imply that we can apply
  the weight vector~$[0,0,0,1]$, leading to a new model with smaller absolute
  value of the invariant, contradicting our choice of~$F$.
\end{proof}

\begin{Remark}
  The same argument shows that every semistable Fano (i.e., such
  that $d \le n$) hypersurface over~$\Q$
  has an integral model such that the reduction modulo any prime does not
  contain a hyperplane. In particular, any semistable quadric in~$\BP^n$ with
  $n \ge 2$ and any semistable cubic in~$\BP^n$ with $n \ge 3$
  has an integral model such that the reductions modulo all primes are irreducible.
\end{Remark}

Before we look at the other weight vectors, we state some facts on singular
points on cubic surfaces. We will use the terms `$k$-plane', `$k$-line' and
`$k$-point' to refer to a plane, line or point defined over the field~$k$.

\begin{Lemma} \label{L:singcubsurf}
  Let $F \in k[x_0,x_1,x_2, x_3]$ be nonzero and homogeneous of degree~$3$.
  We denote by $X \subseteq \BP^3_k$ the cubic surface defined by~$F$
  and by $X_\sing$ its singular subscheme.

  If $X$ does not contain a $k$-plane, then either $X_\sing(k)$ consists
  of the $k$-points on a single $k$-line, or else of finitely many affinely
  independent points (in particular, $\#X_\sing(k) \le 4$ in this case).
\end{Lemma}

\begin{proof}
  As $X$ does not contain a $k$-plane, $F$ is either absolutely irreducible or
  a norm form. In the later case $X_\sing(k)$ consists of the intersection of three
  planes and is therefore either a single point or a line.

  Now, let $F$ be absolutely irreducible. If $X_\sing$ has a one-dimensional part, it
  is a line; see~\cite{CAG}*{Sec.~9.2.1}. As cubic surfaces contain the line joining each
  pair of singular points, we can conclude that a cubic surface containing a singular
  line and a singular point not on the line will contain the plane spanned by them
  and is therefore reducible.

  Next we inspect the case that 3 singular points are on a line. Then the intersection
  of the cubic surface with any plane containing this line is a cubic curve with at least
  these 3 singular points. Thus, all these cubic curves contain a double line and therefore
  the entire line is singular.

  If 4 singular points are contained in a plane~$S$ but not in a line, then the lines joining these
  singularities are contained in the surface. Therefore, the intersection of the cubic surface
  with~$S$ will contain at least 4 lines. Thus, the entire plane~$S$ is contained in the surface.
  As $S$ is a $k$-plane, this is a contradiction.

  Finally, the bound~$4$ for the number of singular points is proven
  in~\cite{CAG}*{Cor.~9.2.3}.
\end{proof}

\begin{Definition} \label{D:very singular}
  For the discussion below, we say that a point~$\bar{P}$ on the
  surface~$\bar{X}$ defined by $\bar{F} = 0$ (where $F \in \Z[x_0,x_1,x_2,x_3]$
  is a cubic form) is \emph{very singular} if $\bar{P}$ is a singular point
  of~$\bar{X}$ and $v_p(F(P)) \ge 2$ for some lift~$P$ of~$\bar{P}$.
  This latter condition is independent of the choice of the lift.
  In other words, a very singular point is a singular point of the reduction that
  is not a regular point of the $\Z_p$-scheme.
\end{Definition}

\begin{Lemma} \label{L:very singular line}
  Keeping the notation of Definition~\ref{D:very singular}, we assume
  that $\bar{X}$ does not contain an $\F_p$-plane.
  If there are more than four very singular $\F_p$-points on~$\bar{X}$,
  then all these points are contained in a line and every point of the line is very
  singular.
\end{Lemma}

\begin{proof}
  This follows from Lemma~\ref{L:singcubsurf} and the fact that a singular
  line that contains points that are not very singular can contain at
  most three very singular points: let the line be given by $x_2 = x_3 = 0$.
  Since the line is singular, $F$ can be written in the form
  \[ F = f_3(x_2,x_3) + f_{2,0}(x_2,x_3) x_0 + f_{2,1}(x_2,x_3) x_1
          + p x_2 g_{2,2}(x_0,x_1) + p x_3 g_{2,3}(x_0,x_1) + p g_3(x_0,x_1) \,,
  \]
  where $f_3$ and~$g_3$ are binary cubic forms and $f_{2,0}$, $f_{2,1}$,
  $g_{2,2}$, $g_{2,3}$ are binary quadratic forms. A point $(\xi_0 : \xi_1 : 0 :0)$
  on the line is very singular if and only if $\bar{g}_3(\xi_0, \xi_1) = 0$.
  Either $\bar{g}_3$ is identically zero, then the line consists of very
  singular points, or else $\bar{g}_3$ has at most three zeros on~$\BP^1_{\F_p}$.
\end{proof}

We now consider the weight vector $[0,0,1,1]$.
We keep the notation~$\bar{X}$ for the surface given by $\bar{F} = 0$.

\begin{Lemma} \label{L:0011}
  Let $F \in \Z[x_0,x_1,x_2,x_3]$ be primitive and homogeneous of degree~$3$.
  We assume that $[0,0,0,1]$ does not apply to~$F$.
  \begin{enumerate}[\upshape(1)]
    \item \label{I:0011.1}
          If $[0,0,1,1]$ applies to~$F$, then $\bar{X}$ contains a (unique)
          singular line defined over~$\F_p$ that consists of very singular points.
    \item \label{I:0011.2}
          If $\bar{X}$ is singular along the line $x_2 = x_3 = 0$,
          then $w = [0,0,1,1]$ applies to~$F$ if and only if $F$ is unstable at~$p$
          for~$(E, w)$,
          i.e., if and only if $v_p(F(x_0,x_1,px_2,px_3)) \ge 2$. Equivalently,
          the line consists of very singular points.
  \end{enumerate}
\end{Lemma}

\begin{proof}
  By Lemma~\ref{L:0001}, the assumption that $[0,0,0,1]$ does not apply to~$F$
  means that $\bar{F}$ is irreducible.
  \begin{enumerate}[\upshape(1)]
    \item We have that $v_p({}^T F(x_0,x_1,px_2,px_3)) \ge 2$
          with a suitable unimodular matrix~$T$. This implies that
          ${}^{\bar{T}} \bar{F} \in \langle x_2, x_3 \rangle^2$, and so
          $\bar{X} \cdot \bar{T}^{-1}$ is singular along the line $x_2 = x_3 = 0$.
          We also see that $g_3$ for ${}^T F$ as in the proof of
          Lemma~\ref{L:very singular line} is divisible by~$p$,
          which implies that the line consists of very singular points.
    \item The `if' direction is clear. For the `only if', first note
          that by part~\eqref{I:0011.1} and its proof, there must be
          a unimodular matrix~$T$ such that $\bar{T}$ fixes the line
          $x_2 = x_3 = 0$ and $F$ is unstable w.r.t.~$(T,w)$.
          Now one easily checks that the latter condition is independent
          of the choice of~$T$ with these properties, so it holds for
          some~$T$ if and only if it holds for $T = E$.
    \qedhere
  \end{enumerate}
\end{proof}

So to check whether $[0,0,1,1]$ applies to~$F$, we find the singular
$\F_p$-lines on~$\bar{X}$, of which there is at most one.
If such a line exists, we check the criterion given in part~\eqref{I:0011.2}.

When neither $[0,0,0,1]$ nor~$[0,0,1,1]$ apply, Lemmas~\ref{L:singcubsurf}
and~\ref{L:very singular line} tell
us that there are at most four very singular $\F_p$-points on~$\bar{X}$.
This will be useful for dealing with the remaining minimal weight vectors.
We begin with~$[0,1,1,1]$

\begin{Lemma} \label{L:0111}
  Let $F \in \Z[x_0,x_1,x_2,x_3]$ be primitive and homogeneous of degree~$3$.
  We assume that $[0,0,0,1]$ does not apply to~$F$.
  \begin{enumerate}[\upshape(1)]
    \item \label{I:0111.1}
          If $[0,1,1,1]$ applies to~$F$, then $\bar{X}$ is a cone over a
          cubic curve. The vertex of the cone is an $\F_p$-point~$\bar{P}$ of
          multiplicity~$3$ that is very singular.
    \item \label{I:0111.2}
          If the point $[1:0:0:0]$ has multiplicity~$3$ and is very singular
          on~$\bar{X}$, then $F$ is unstable at~$p$
          for some~$(T, w)$ such that $\bar{T}$ fixes $[1:0:0:0]$
          if and only if this is true for $T = E$,
          i.e., if and only if $v_p(F(x_0,px_1,px_2,px_3)) \ge 3$.
  \end{enumerate}
\end{Lemma}

\begin{proof}
  As before, $\bar{F}$ is irreducible.
  \begin{enumerate}[\upshape(1)]
    \item We have that $v_p({}^T F(x_0,px_1,px_2,px_3)) \ge 3$
          with a suitable unimodular matrix~$T$. This implies that
          ${}^{\bar{T}} \bar{F} \in \langle x_1, x_2, x_3 \rangle^3$, and so
          $\bar{X}$ is a cone; its vertex~$\bar{P}$ is defined over~$\F_p$
          and has multiplicity~$3$.
          There is one lift~$P$ such that $v_p(F(P)) \ge 3$; this implies that
          $\bar{P}$ is very singular.
    \item This is shown in a similar way as part~\eqref{I:0011.2} of Lemma~\ref{L:0011}.
    \qedhere
  \end{enumerate}
\end{proof}

Finally, we deal with the weight vectors $[0,1,2,2]$ and~$[0,2,2,3]$.

\begin{Lemma} \label{L:0122,0223}
  Let $F \in \Z[x_0,x_1,x_2,x_3]$ be primitive and homogeneous of degree~$3$.
  We assume that $[0,0,0,1]$, $[0,0,1,1]$ and~$[0,1,1,1]$ do not apply to~$F$.
  \begin{enumerate}[\upshape(1)]
    \item \label{I:0122.1}
          If $[0,1,2,2]$ or~$[0,2,2,3]$ applies to~$F$, then $\bar{X}$
          contains a very singular $\F_p$-point~$\bar{P}$ with the following property.
          Lift $\bar{P}$ to a point~$P$ and write
          \[ F(P + x) = f_0 + f_1(x) + f_2(x) + f_3(x) \]
          with $f_j$ homogeneous of degree~$j$. Then $p^2 \mid f_0$, $p \mid f_1$,
          and the quadric $\bar{f}_2$ has rank $1$ or~$2$.
    \item \label{I:0122.2}
          Assume that the point $[1:0:0:0]$ is very singular on~$\bar{X}$; then
          \[ F_1 := p^{-2} F(x_0,px_1,px_2,px_3) \in \Z[x_0,x_1,x_2,x_3] \]
          is primitive. Write
          \[ \bar{F}(x_0,x_1,x_2,x_3) = x_0 f_2(x_1,x_2,x_3) + f_3(x_1,x_2,x_3)
                                         \in \F_p[x_0,x_1,x_2,x_3]
          \]
          with $f_j$ homogeneous of degree~$j$ and assume that $f_2$ has rank $1$ or~$2$.
          Then $F$ is unstable for~$(T, [0,1,2,2])$ or~$(T, [0,2,2,3])$
          for some~$T$ such that $\bar{T}$ \hbox{fixes~$[1:0:0:0]$} if and only if
          \begin{enumerate}[\upshape(a)]
            \item \label{I:0122.2a}
                  either $f_2$ has rank~$2$ and $[0,0,1,1]$ applies to~$F_1$
                  (which is the case if and only if the singular line of $f_2 = 0$
                  is very singular on~$\bar{F}_1 = 0$); then $[0,1,2,2]$ applies to~$F$,
            \item \label{I:0122.2b}
                  or else $f_2$ has rank~$1$ and $\bar{F}_1$ has a linear factor
                  $l \in \F_p[x_0,x_1,x_2,x_3]$ of multiplicity 2.
                  After transforming $l$ to $x_3$
                  the weight vector~$[0,0,0,1]$ applies. Now, either
                  $[0,0,0,1]$ or~$[0,1,1,1]$ applies to the resulting form.
                  In the first case, $[0,1,2,2]$ applies to~$F$, in the second
                  case, $[0,2,2,3]$ applies.
          \end{enumerate}
  \end{enumerate}
\end{Lemma}

\begin{proof}
  By assumption, $\bar{F}$ is irreducible and $\bar{X}$ does not contain a
  very singular line.
  \begin{enumerate}[\upshape(1)]
    \item After applying a suitable unimodular transformation~$T$, we can assume
          that $F$ is unstable for~$(E, [0,1,2,2])$ or~$(E, [0,2,2,3])$,
          i.e., that
          \[ v_p\bigl(F(x_0, p x_1, p^2 x_2, p^2 x_3)\bigr) \ge 4
              \quad\text{or}\quad
             v_p\bigl(F(x_0, p^2 x_1, p^2 x_2, p^3 x_3)\bigr) \ge 6 \,.
          \]
          This implies in both cases that $[1:0:0:0]$ is very singular
          on~$\bar{X}$ and that
          \[ \bar{F} = x_0 f_2(x_2,x_3) + f_3(x_1,x_2,x_3) \]
          with a binary quadratic form~$f_2$ and a ternary cubic form~$f_3$.
          If $f_2 = 0$, then one easily checks that $F$ is unstable with respect
          to~$(E, [0,1,1,1])$, but this is excluded by assumption. Therefore $f_2$
          must have rank $1$ or~$2$.
    \item The first claim is easily checked. Note that $F_1$ cannot be divisible
          by~$p$, since otherwise $[0,1,1,1]$ would apply to~$F$.
          It is also easily checked that
          in both cases \eqref{I:0122.2a} and~\eqref{I:0122.2b} the resulting
          form is `more minimal' than~$F$. Moving the line in case~\eqref{I:0122.2a}
          to $x_2 = x_3 = 0$, we also see that the sequence of steps amounts
          to an application of~$[0,1,2,2]$. Similarly, moving
          the double plane $l^2 = 0$ in case~\eqref{I:0122.2b} to $x_3^2 = 0$,
          the application of~$[0,0,0,1]$ to~$F_1$ is with respect to~$x_3 = 0$,
          and we see that together with the last step, we obtain an application
          of $[0,1,2,2]$ or~$[0,2,2,3]$ to~$F$.
          It remains to show that when one of $[0,1,2,2]$ and~$[0,2,2,3]$
          applies to~$F$, then we are in one of the two cases.
          To check this, we write out the implied minimal valuations of all
          coefficients. Then we track the effect of the listed transformations
          and confirm that all the coefficients of the intermediate results have
          the required minimal valuations.
          This cumbersome task is conveniently done by using a computer
          algebra system.

          In particular, we confirm that
          $f_2$ involves at most the monomials $x_2^2, x_2 x_3, x_3^2$ and
          is therefore of rank at most 2 in case  $[0,1,2,2]$, and it involves
          only $x_3^2$ in case~$[0,2,2,3]$.
          Finally, in case~$[0,2,2,3]$ the application of $[0,0,0,1]$ with respect to $l$ results in
          a cubic form, such that its reduction involves at most the monomials
          $x_1^3,x_1^2 x_2, x_1 x_2^2, x_2^3$ and is therefore reducible or a norm form.
          In the case of reducible reduction $[0,1,2,2]$ applies as well and we are done.
          In the case of a norm form we obtain at most 2 very singular points and $[0,1,1,1]$ applies.
    \qedhere
  \end{enumerate}
\end{proof}

So after checking whether $[0,0,0,1]$ or~$[0,0,1,1]$ apply
and finding that they do not, we determine the at most four very singular
$\F_p$-points on~$\bar{X}$, and for each of them, check the criteria of
Lemma~\ref{L:0111}~\eqref{I:0111.2} and Lemma~\ref{L:0122,0223}~\eqref{I:0122.2}
to see if one of $[0,1,1,1]$, $[0,1,2,2]$ or~$[0,2,2,3]$ applies.
Putting all these steps together gives us a procedure
{\sf MinimizeCubicSurfaceOneStep} similar to {\sf MinimizePlaneCurveOneStep},
which can then be called successively by a procedure {\sf MinimizeCubicSurface}
while successful minimization steps are performed. In this way,
the results of this section can be turned into an algorithm.
This has been implemented by the first author in Magma~\cite{Magma};
the procedure is available under the name \texttt{MinimizeCubicSurface}.


\section{Reduction of cubic surfaces} \label{S:csred}

In a similar way as for plane curves, we have to perform a reduction of a minimized cubic surface
to obtain an equation with small coefficients. Instead of a cluster-based approach we
will use a representation as a sum of cubes of linear forms.
This is based on the following classical result.

\begin{Theorem}[Sylvester {\cite{CAG}*{Theorem~9.4.1}}]
  Let $F = 0$ be a general cubic surface over~$\C$. Then there exist five linear forms
  $l_1, \ldots, l_5$ such that
  \[ F = l_1^3 + l_2^3 + l_3^3 + l_4^3 + l_5^3 \,. \]
  These linear forms are unique up to order and multiplication by third roots of
  unity. This is called the \textbf{pentahedral form of~$F$}.
\end{Theorem}

\begin{Remark}
  This statement does not hold for so-called \emph{cyclic} cubic surfaces.
  They are, up to linear equivalence, of the shape $w^3 + g(x,y,z) = 0$
  with a ternary cubic form~$g$; see~\cite{CAG}*{Sec.~9.4.1}.

  One of the most extreme examples is the diagonal cubic surface $x_0^3 + x_1^3 + x_2^3 + x_3^3 = 0$.
  It has infinitely many such representations.
  To overcome the difficulties, the algorithm will deform cyclic
  cubic surfaces to nearby non-cyclic ones.
\end{Remark}

\begin{Definition}
  Let $F = 0$ be a cubic surface. Its \emph{kernel surface}
  (sometimes also called the \emph{Hessian}) is the quartic surface given by the equation
  \[ \det \left(\frac{\partial^2 F}{\partial x_i \partial x_j}\right)_{i,j} = 0 \,. \]
\end{Definition}

\begin{Theorem}[Clebsch \cite{Clebsch}*{Theorem 7}, \cite{CAG}*{Sec.~9.4.2}]
\label{ThmCl}
  Let
  \[ F = l_1^3 + l_2^3 + l_3^3 + l_4^3 + l_5^3 \]
  be a general cubic surface in pentahedral form.
  Choose coefficients $a_1,\ldots,a_5$ and linear forms $k_1,\ldots,k_5$ such
  that $k_1 + k_2 + k_3 + k_4 + k_5 = 0$ and $a_i k_i = l_i$. Then the singular
  points of the kernel surface of~$F$ are the points given by
  \[ k_{i_1} = 1, \quad k_{i_2} = -1, \quad k_{i_3} = 0, \quad  k_{i_4} = 0, \quad k_{i_5} = 0 \]
  for $\{i_1, i_2, i_3, i_4, i_5\} = \{1,2,3,4,5\}$.
\end{Theorem}

The theorem above allows us to derive
the pentahedral form of a cubic surface from the singular points of its Hessian.
Each plane $l_i = 0$ (equivalently, $k_i = 0$) contains
six singular points of the kernel surface.
Thus, as soon as the combinatorial structure of the singular points is known, one
can compute the planes $k_i = 0$ by solving linear systems.
This is used in the algorithm below.

\begin{Algorithm} \label{red_cub_surf}
  Let $F = 0$ be a general cubic surface over~$\Q$. This algorithm computes
  a reduced form in the $\GL(4,\Z)$-orbit of~$F$ and the transformation matrix.

  {\sf ReduceCubicSurface}($F$) \\
  \sq $Q := \det \left(\frac{\partial^2 F}{\partial x_i \partial x_j}\right)_{i,j}$; \\
  \sq Compute the singular points of $Q = 0$; \\
  \sq // \emph{If we do not find 10 isolated singularities, we add a small perturbation to $F$.} \\
  \sq Solve the linear system of Theorem~\ref{ThmCl} to obtain $k_1, \ldots, k_5 \in \C[x]$; \\
  \sq Solve the linear system for the $b_i = a_i^3$ given by $F = b_1 k_1^3 + \ldots + b_5 k_5^3$; \\
  \sq $H(x) := |\sqrt[3]{b_1} k_1(x)|^2 + |\sqrt[3]{b_2} k_2(x)|^2 + |\sqrt[3]{b_3} k_3(x)|^2
                   + |\sqrt[3]{b_4} k_4(x)|^2 + |\sqrt[3]{b_5} k_5(x)|^2$; \\
  \sq // \emph{$H$ is a positive definite Hermitian form with real coefficients,} \\
  \sq // \quad \emph{so it is actually a positive definite real quadratic form} \\
  \sq Compute a matrix~$T$ whose rows are an LLL-reduced basis of~$\Z^4$ with respect to~$H$; \\
  \sq \textbf{return} $F(x T^{-1})$, $T^{-1}$;

  If one does not want to detect the combinatorial structure of the singular points by a floating
  point computation, one can use the approach described in~\cite{EJ2}*{Algo.~A.4}.
  I.e., one computes the field of definition of one of the planes $l_i = 0$ and splits
  the singular subscheme of the kernel surface over that field. One of the components will
  contain all the singular points contained in $l_i = 0$ and a second component will contain all
  the other ones.
\end{Algorithm}

This algorithm has been implemented by the first author in Magma~\cite{Magma}.
It is available via {\tt ReduceCubicSurface} and {\tt MinimizeReduce}.

One could also try to apply the cluster reduction of~\cite{Stoll2011b}
to the singular points of the kernel surface and apply the transformation
matrix obtained in this way to the initial cubic surface. In most cases,
the results obtained by Algorithm~\ref{red_cub_surf} are slightly better.

\begin{figure}[htb]
  {
  \begin{align*}
    {} &-866812507957452012700721792086587937 x^3 \\
    {} &+ 3728812982147606773738081898305547310 x^2 y \\
    {} &+ 64283763770985952786436023327908284160 x^2 z \\
    {} &+ 497718355086466637590632151750449246396 x^2 w \\
    {} &- 22244579889188354084172896622822533100 x y^2 \\
    {} &- 431923319964698868982551682351317273600 x y z \\
    {} &- 2446192338737080630831681553231971375920 x y w \\
    {} &- 1618017788538827453488905618589376819200 x z^2 \\
    {} &+ 15747155527321974660280650027255501486080 x z w \\
    {} &- 66025203088832123300929566152845689479856 x w^2 \\
    {} &- 65456138728936479908688098323552023000 y^3 \\
    {} &- 357488525368202205779029272883004032000 y^2 z \\
    {} &+ 20762944510278587277812066653228558975600 y^2 w \\
    {} &+ 20013727944438057575668128606471875584000 y z^2 \\
    {} &+ 64721500464867439337111893187712691097600 y z w \\
    {} &- 351425459041632833836477745377146122692640 y w^2 \\
    {} &+ 5759206855635558085134656966457081856000 z^3 \\
    {} &- 406645509553946606042771346800156046540800 z^2 w \\
    {} &- 3284853297122243046122373374040607648010240 z w^2 \\
    {} &- 2681060506817531405431579495959221739841728 w^3
  \end{align*}
  }
  \caption{Cubic form defining~$S_0$ (see Example~\ref{Ex:S0}).}
  \label{Fig:S0}
\end{figure}

\begin{Example} \label{Ex:S0}
  Let $S_0$ be the cubic surface given by the form in the variables $x,y,z,w$
  shown in Figure~\ref{Fig:S0}. $S_0$ has bad reduction at
  \[ p = 2,\; 3,\; 5,\; 7,\; 13,\; 113,\; 463,\; 733,\; 2141,\; 9643,\; 14143,\; 17278361,\; 22436341 \,. \]
  Choosing better models by using the methods described in section \ref{S:cubsurf}
  and running the LLL-based reduction algorithm,
  one gets the new surface~$S$ given by
  \begin{align*}
    &2 x^3 + 16 x^2 z - 12 x^2 w - 17 x y^2 + 61 x y z - 26 x y w \\
    &{} - 20 x z^2  + 95 x z w + 18 x w^2 + 5 y^3 + 33 y^2 z + 10 y^2 w \\
    &{} - 25 y z w  - 22 y w^2 - 11 z^3 - 21 z^2 w + 50 z w^2 - 52 w^3 = 0 \,.
  \end{align*}
  $S$ has bad reduction at
  \[ p = 2,\; 3,\; 5,\; 7,\; 13,\; 733,\; 22436341 \,. \]
  The reduction of~$S$ modulo $3$, $5$, $7$, $13$ and~$22436341$ has a unique singularity of type~$A_1$.
  The reduction modulo~$2$ has one singular point of type~$A_1$ and one of type~$A_3$.
  Finally, the reduction modulo~$733$ is a cone over a smooth curve.
\end{Example}

\begin{Remark}
  The initial equation for~$S_0$ was constructed by~\cite{EJ} such that the
  $27$~lines form orbits of lengths $6$, $9$ and~$12$
  under the action of $\Gal(\Bar{\Q} / \Q)$. In particular, the field~$K$ of
  definition of the $27$~lines is a degree~$144$ number field.
  The surface is of arithmetic Picard rank~$1$, and the lines in the
  orbit of length~$12$ form a double-six.
  The construction was started with the polynomial
  \[ t^6 + 330t^4 + 1452 t^3 + 13705 t^2 + 123508 t + 835540 \,. \]
  The field~$K$ is generated by~$\sqrt{5}$ together with all the roots of this polynomial.
\end{Remark}


\end{document}